\documentclass[reqno,english]{amsart}%
\usepackage{amsfonts,amsmath,latexsym,verbatim,amscd,mathrsfs,color,array}
\usepackage[colorlinks=true]{hyperref}
\usepackage{amsmath,amssymb,amsthm,amsfonts,graphicx,color}
\usepackage{amssymb}
\usepackage{pdfsync}
\usepackage{epstopdf}
\usepackage{cite}
\usepackage{graphicx}
\usepackage{amsmath}
\usepackage{amsfonts}%
\providecommand{\U}[1]{\protect\rule{.1in}{.1in}}

\numberwithin{equation}{section}

\newtheorem{theorem}{Theorem}[section]
\newtheorem{corollary}{Corollary}[section]
\newtheorem{lemma}{Lemma}[section]
\newtheorem{proposition}{Proposition}[section]
\newtheorem{remark}{Remark}[section]
\newtheorem{definition}{Definition}[section]

\numberwithin{equation}{section}

\newcommand{\bbr}{\mathbb{R}}

\newcommand{\bbn}{\mathbb{N}}
\newcommand{\ve}{\varepsilon}
\newcommand{\vt}{\vartheta}

\newcommand{\bd}{\begin{definition}}
	\newcommand{\ed}{\end{definition}}

\newcommand{\br}{\begin{remark}}
	\newcommand{\er}{\end{remark}}

\newcommand{\be}{\begin{equation}}
	\newcommand{\ee}{\end{equation}}

\newcommand{\bc}{\begin{corollary}}
	\newcommand{\ec}{\end{corollary}}

\begin{document}
	\title[Infinitely many positive solution]{Infinitely many nonradial positive solutions for multi-species nonlinear Schr\"odinger systems in $\bbr^N$}
	\author[T. Li]{Tuoxin Li}
	\address{\noindent Department of Mathematics, University of British Columbia,
		Vancouver, B.C., Canada, V6T 1Z2}
	\email{tuoxin@math.ubc.ca}
	
	\author[J. Wei]{Juncheng Wei}
	\address{\noindent Department of Mathematics, University of British Columbia,
		Vancouver, B.C., Canada, V6T 1Z2}
	\email{jcwei@math.ubc.ca}
	
	\author[Y. Wu]{Yuanze Wu}
	\address{\noindent  School of Mathematics, China
		University of Mining and Technology, Xuzhou, 221116, P.R. China }
	\email{wuyz850306@cumt.edu.cn}
	
	\begin{abstract}
		In this paper, we consider the multi-species nonlinear Schr\"odinger systems in $\bbr^N$:
		\begin{equation*}
			\left\{\aligned&-\Delta u_j+V_j(x)u_j=\mu_ju_j^3+\sum_{i=1;i\not=j}^d\beta_{i,j} u_i^2u_j\quad\text{in }\bbr^N,\\
			&u_j(x)>0\quad\text{in }\bbr^N,\\
			&u_j(x)\to0\quad\text{as }|x|\to+\infty,\quad j=1,2,\cdots,d,\endaligned\right.
		\end{equation*}
		where $N=2,3$, $\mu_j>0$ are constants, $\beta_{i,j}=\beta_{j,i}\not=0$ are coupling parameters, $d\geq2$ and $V_j(x)$ are potentials.  By Ljapunov-Schmidt reduction arguments, we construct infinitely many nonradial positive solutions of the above system under some mild assumptions on potentials $V_j(x)$ and coupling parameters $\{\beta_{i,j}\}$, {\it without any symmetric assumptions on the limit case of the above system}.  Our result, giving a positive answer to the conjecture in Pistoia and Vaira \cite{PV22} and extending the results in \cite{PW13,PV22}, reveals {\it new phenomenon} in the case of $N=2$ and $d=2$ and is {\it almost optimal} for the coupling parameters $\{\beta_{i,j}\}$.
		
		\vspace{3mm} \noindent{\bf Keywords:} nonlinear Schrodinger systems, infinitely many positive solutions, reduction method, min-max argument.
		
		\vspace{3mm}\noindent {\bf AMS} Subject Classification 2010: 35B09; 35B33; 35B40; 35J20.%
		
	\end{abstract}
	
	\date{}
	\maketitle
	
	\section{Introduction}
	\subsection{Backgrounds}
	In this paper, we consider the multi-species nonlinear Schr\"odinger systems in $\bbr^N$:
	\begin{equation}\label{eq0001}
		\left\{\aligned&-\Delta u_j+V_j(x)u_j=\mu_ju_j^3+\sum_{i=1;i\not=j}^d\beta_{i,j} u_i^2u_j\quad\text{in }\bbr^N,\\
		&u_j(x)>0\quad\text{in }\bbr^N,\\
		&u_j(x)\to0\quad\text{as }|x|\to+\infty,\quad j=1,2,\cdots,d,\endaligned\right.
	\end{equation}
	where $N=2,3$, $\mu_j>0$ are constants, $\beta_{i,j}=\beta_{j,i}\not=0$ are coupling parameters, $d\geq2$ and $V_j(x)$ are potentials.
	
	\vskip0.2in
	
	It is well known that solutions of \eqref{eq0001} are related to the bright solitons of the Gross-Pitaevskii equations (cf. \cite{HMEWC98}),
	\begin{equation}\label{eqn9877}
		\left\{\aligned&\iota\frac{\partial\Psi_j}{\partial t}=\Delta \Psi_j-V_j(x)\Psi_j+\mu_j|\Psi_j|^2\Psi_j+\sum_{i=1;i\not=j}^d\beta_{i,j}|\Psi_i|^{2}\Psi_j,\\
		&\Psi_j=\Psi_j(t,x)\in H^1(\bbr^{N+1}; \mathbb{C}),\ \ j=1,2,\cdots,d,\ \ N=2,3,\endaligned\right.%
	\end{equation}
	by the relation $\Psi_j(t,x)=e^{-\iota\lambda_j t}u_j(x)$, where $\iota$ is the imaginary unit.  The Gross-Pitaevskii equations~\eqref{eqn9877} have applications in many physical models, such as in nonlinear optics (cf. \cite{AA99}) and in Bose-Einstein condensates for multi-species condensates (cf. \cite{CLLL04,EGBB97,R03}).  In Bose-Einstein condensates for multi-species condensates, $\mu_j$ and $\beta_{i,j}$ in \eqref{eq0001} are the intraspecies and interspecies scattering lengths respectively, while $V_j(x)$ stands for the magnetic trap (cf. \cite{TW98}) arising from the chemical potentials.  The sign of the scattering length $\beta_{i,j}$ determines whether the interactions of states $i\rangle$ and $j\rangle$ are repulsive ($\beta_{i,j}<0$) or attractive ($\beta_{i,j}>0$).
	
	\vskip0.2in
	
In the autonomous case, i.e., the potentials $V_j$ are positive constants for all $j=1,2,\cdots,d$, multi-species nonlinear Schr\"odinger systems~\eqref{eq0001} have been studied extensively in the pase two decades after the pioneer work \cite{LW05}.  By using variational methods, Lyapunov-Schmidt reduction methods or bifurcation methods, various theorems, about the existence, multiplicity and qualitative properties of nontrivial solutions of autonomous multi-species nonlinear Schr\"odinger systems like \eqref{eq0001}, have been established in the literature under various assumptions on the coupling parameters.  Since it seems almost impossible for us to provide a complete list of references, we refer the readers only to \cite{AC06,AC07,BDW10,BJS16,BS17,BS19,BTWW13,BW06,BWW07,BZZ21,CD10,CLLL04,CTV05,CZ13,DW09,DWW10,DWZ12,GJ16,GJ18,MMP06,M15,NTTV10,NTTV12,PS18,PT17,PST20,S07,TY20,TYZ22,
		WW07,WW08,WW081,WZZ22} and the references therein for the two coupled case $d=2$, \cite{C16,C161,COT16,CS19,LW08,PV22,ST15,STTZ16,SZ15,TT12,TT121} and the references therein for the multi-coupled case~$d\geq3$ with the purely repulsive couplings or the purely attractive couplings and \cite{BLM21,BLW16,BMW21,BSW16,BSW161,C14,CP22,DW12,LW05,PWW19,SW15,SW151,S15,ST16,
		TTVW11,WW19} and the references therein for the multi-coupled case~$d\geq3$ with the mixed  couplings.  Here, we call the couplings $\{\beta_{i,j}\}$ is purely repulsive if $\beta_{i,j}<0$ for all $i\not=j$, we call the couplings $\{\beta_{i,j}\}$ is purely attractive if $\beta_{i,j}>0$ for all $i\not=j$ and we call the couplings $\{\beta_{i,j}\}$ is mixed if there exist $(i_1,j_1)$ and $(i_2,j_2)$ such that $\beta_{i_1,j_1}>0$ and $\beta_{i_2,j_2}<0$.
	
	\vskip0.2in
	
In the non-autonomous case, it is well known nowadays that the magnetic trapping potentials will play important roles in constructing nontrivial solutions of nonlinear Schr\"odinger equations or systems, see, for example, \cite{B13,CDS05,CPS13,DS12,PW13,PV22,R92,TW13,TV09,WY10} and the references therein.  In particular, in \cite{WY10}, Wei and Yan constructed infinitely many nonradial positive solutions of the nonlinear Schr\"odinger equation,
\begin{eqnarray}
\label{scalar}
		\left\{\aligned&-\Delta u+V(x)u=|u|^{p-1}u\quad\text{in }\bbr^N,\\
		&u(x)\to0\quad\text{as }|x|\to+\infty,\endaligned\right.
	\end{eqnarray}
	where $V(x)$ satisfies the assumption:
	\begin{enumerate}
		\item[$(V_*)$]\quad $V(x)>0$ is continuous and radial with $V(x)=1+\frac{\delta}{|x|^{\nu}}+O\bigg(\frac{1}{|x|^{\nu+\ve}}\bigg)$ as $|x|\to+\infty$, where $\delta\in\bbr$ and $\nu>1$ and $\ve>0$.
	\end{enumerate}
Let us briefly sketch Wei and Yan's construction in \cite{WY10}.  Let $w_j$ be the unique (up to translations) positive solution of the following equation:
	\begin{equation}\label{eq0003}
		\left\{\aligned&-\Delta u+\lambda_ju=\mu_ju^3\quad\text{in }\bbr^N,\\
		&u(x)>0\quad\text{in }\bbr^N,\\
		&u(x)\to0\quad\text{as }|x|\to+\infty,\endaligned\right.
	\end{equation}
where $j=1,2,\cdots,d$.
	Then it is well known that there exists $A_{N,j}>0$, which only depends on $N$ and $j$, such that
	\begin{eqnarray}\label{eqnnnew0008}
		w_j(x)=A_{N,j}(1+O(|x|^{-1}))|x|^{\frac{1-N}{2}}e^{-\sqrt{\lambda_j}|x|}\quad\text{as }|x|\to+\infty.
	\end{eqnarray}
For the sake of clarity, we denote $w_j=\widetilde{w}_j$ in the partially symmetric case $\lambda_j=1$ for all $j$ and we denote $w_j=w$ in the totally symmetric case $\lambda_j=1$ and $\mu_j=1$ for all $j$.  Then by the assumption~$(V_*)$, \eqref{eq0003} and \eqref{eqnnnew0008}, $\sum_{t=1}^{\vartheta}w(x-\eta_{t})+v_*$
is very close to a genuine solution of \eqref{scalar} if
	\begin{eqnarray*}
		\min_{t\not=s}|\eta_{t}-\eta_{s}|\to+\infty\quad\text{and}\quad\min_{t}|\eta_{t}|\to+\infty
	\end{eqnarray*}
	as $\vartheta\to+\infty$ by the Lyapunov-Schmidt reduction, where $\vartheta\in\bbn$, $t=1,2,\cdots,\vartheta$ and $v_*$ is much smaller than the approximation $\sum_{t=1}^{\vartheta}w(x-\eta_{t})$ in a suitable sense.  To prove $\sum_{t=1}^{\vartheta}w(x-\eta_{t})+v_*$ is a genuine solution of \eqref{scalar}, the adjustment of the locations of $\{\eta_{t}\}$ is needed.  In \cite{WY10},
the key idea is to put $\{\eta_{t}\}$ on a circle with a large radius, which are invariant under the action of a discrete subgroup of $SO(N)$, to reduce the number of parameters in adjusting $\{\eta_{t}\}$ and using $\vartheta$, the number of spikes, as a parameter in the construction of spiked solutions of \eqref{scalar}.  More precisely, Wei and Yan choose $\eta_{t}=\rho_\vartheta\xi_{t}$ where $\rho_{\vartheta}\sim\vartheta\log\vartheta$ is the radius and the locations $\xi_t$ satisfies
	\begin{eqnarray}\label{eqnnnew4444}
		\xi_{t}=\bigg(\cos\bigg(\frac{2(t-1)\pi}{\vt}\bigg), \sin\bigg(\frac{2(t-1)\pi}{\vt}\bigg), 0\bigg).
	\end{eqnarray}
Then the adjustment of the locations of $\{\eta_{t}\}$ is reduced to find a critical point of the reduced energy functional of the parameter $\vartheta$ which can be solved by taking the maximum of this reduced energy functional of the parameter $\vartheta$ over a suitable set.
We point out that generated by the fact that the building block shares the same decaying property~\eqref{eqnnnew0008}, the locations of $\{\eta_{t}\}$ in \cite{WY10} are invariant under the rotation of the angle $\frac{2\pi}{\vartheta}$ and this invariance is crucial in the above construction.

	\vskip0.12in

Wei and Yan's idea in \cite{WY10} is applied by Peng and Wang in \cite{PW13}, where by Ljapunov-Schmidt reduction arguments, Peng and Wang proved that for the two coupled case $d=2$, multi-species nonlinear Schr\"odinger systems~\eqref{eq0001} has infinitely many nonradial positive solutions in dimension three $N=3$ under the following assumptions on $V_j(x)$,
	\begin{enumerate}
		\item[$(V_0)$]\quad $V_j(x)>0$ is continuous and radial with $V_j(x)=1+\frac{\delta_j}{|x|^{\nu_j}}+O\bigg(\frac{1}{|x|^{\nu_j+\ve}}\bigg)$ as $|x|\to+\infty$, where $\delta_j\in\bbr$ and $\nu_j>1$ and $\ve>0$,
	\end{enumerate}
and some further assumptions on the parameters $\delta_j$ and $\beta_{1,2}$.  The solutions constructed by Peng and Wang in \cite{PW13} either look like $\vartheta$ copies of synchronized spikes
	\begin{eqnarray*}
		\bigg(\sum_{t=1}^{\vartheta} a w(x-\rho_\vartheta\xi_t), \sum_{t=1}^{\vartheta}  b w(x-\rho_\vartheta\xi_t)\bigg),
	\end{eqnarray*}
where $\xi_t$ is given by \eqref{eqnnnew4444},
	or look like $\vartheta$ copies of segregated spikes
	\begin{eqnarray*}
		\bigg(\sum_{t=1}^{\vartheta} \widetilde{w}_1(x-\rho_\vartheta\xi_t), \sum_{t=1}^{\vartheta} \widetilde{w}_2(x-\rho_\vartheta\xi'_t)\bigg),
	\end{eqnarray*}
where
	\begin{eqnarray*}
		\xi'_{t}=\bigg(\cos\bigg(\frac{(2t-1)\pi}{\vt}\bigg), \sin\bigg(\frac{(2t-1)\pi}{\vt}\bigg), 0\bigg),
	\end{eqnarray*}
	with $\rho_\vartheta\sim\vartheta\log\vartheta$ as $\vartheta\to+\infty$ and $(a,b)$ being the unique solution of the following algebraic equation:
	\begin{eqnarray*}
		\left\{\aligned&\mu_1a+\beta_{1,2}b=1,\\
		&\mu_2b+\beta_{1,2}a=1.
		\endaligned\right.
	\end{eqnarray*}
Again, we point out that generated by the fact that the building block shares the same decaying property~\eqref{eqnnnew0008}, the locations of spikes in \cite{PW13} are still invariant under the rotations of the angle $\frac{2\pi}{\vartheta}$ or $\frac{\pi}{\vartheta}$, respectively, and these invariance is also crucial in Peng and Wang's construction in \cite{PW13}, since the adjustment of the locations of spikes can still be reduced to find a critical point of the reduced energy functional of the parameter $\vartheta$ which can be solved by taking the maximum of this reduced energy functional of the parameter $\vartheta$ over a suitable set.  Moreover, to our best knowledge, there is no results about the existence of infinitely many nonradial positive solutions of multi-species nonlinear Schr\"odinger systems~\eqref{eq0001} for the two coupled case $d=2$ in the dimension two $N=2$.

	\vskip0.12in

	Peng and Wang's results in \cite{PW13} have been extended in a recent interesting paper \cite{PV22} by Pistoia and Vaira to the case $d\geq3$ and $N=2,3$.  More precisely, Pistoia and Vaira proved in \cite{PV22}, by Ljapunov-Schmidt reduction arguments too, that there exists $\vartheta_0>0$ such that for every $\vartheta\geq\vartheta_0$, \eqref{eq0001} has a positive solution $(\widetilde{u}_{1,\vartheta},\widetilde{u}_{2,\vartheta}, \cdots, \widetilde{u}_{d,\vartheta})$ for $d\geq3$ and $N=2,3$ under the following symmetry assumption
 \begin{equation}
 \label{PisSym}
{\bf V_i(x)=V_j(x), \mu_i=\mu_j, \beta_{i,j}=\beta, \mbox{ for all $i,j$ }}
 \end{equation}
 where $V_i(x)=V_j(x)$ satisfies the assumption $(V_*)$ and $\beta$ satisfies the smallness assumption
 \begin{equation}
 \label{PisBeta}
{\bf  \beta\in(0, \beta_\vartheta) \mbox{ for }\delta>0 \mbox{ or } \beta\in(-\beta_\vartheta, 0) \mbox{ for } \delta<0}
  \end{equation}
  with $\beta_\vartheta\to0$ as $\vartheta\to+\infty$.  Moreover,
	\begin{eqnarray*}
		\widetilde{u}_{j,\vartheta}\sim\sum_{t=1}^{\vartheta} w(x-\rho_\vartheta\widetilde{\xi}_{t,j})\quad\text{for all }j=1,2,\cdots,d,
	\end{eqnarray*}
	and $(\widetilde{u}_{1,\vartheta},\widetilde{u}_{2,\vartheta}, \cdots, \widetilde{u}_{d,\vartheta})$ is invariant under the rotation of the angle $\frac{2\pi}{d\vartheta}$ in $\bbr^2$,
	where $\rho_\vartheta\sim\vartheta\log\vartheta$ as $\vartheta\to+\infty$ and
	\begin{eqnarray*}
		\widetilde{\xi}_{t,j}=\left\{\aligned&\bigg(\cos\bigg(\frac{2(j-1)\pi}{d\vt}+\frac{2(t-1)\pi}{\vt}\bigg), \sin\bigg(\frac{2(j-1)\pi}{d\vt}+\frac{2(t-1)\pi}{\vt}\bigg)\bigg),\quad N=2,\\
		&\bigg(\cos\bigg(\frac{2(j-1)\pi}{d\vt}+\frac{2(t-1)\pi}{\vt}\bigg), \sin\bigg(\frac{2(j-1)\pi}{d\vt}+\frac{2(t-1)\pi}{\vt}\bigg), 0\bigg),\quad N=3.
		\endaligned\right.
	\end{eqnarray*}
The idea of Pistoia and Vairia in \cite{PV22} is to  combine the two symmetries of the system~\eqref{eq0001} under the assumption~\eqref{PisSym}: First of all, under the condition (\ref{PisSym}) the system \eqref{eq0001} is invariant under any permutation of $(u_1,..., u_d)$.  Secondly, the system \eqref{eq0001} is rotationally invariant if the potentials $V_j(x)$ are radial.  Under these two invariance, it is possible to arrange the $(u_1,..., u_d)$ first in a sector, and then use the rotational symmetry to extend.

	\subsection{Main Results}
	The main purpose of this paper is to extend the results in \cite{PV22,PW13} to obtain an almost optimal existence results of infinitely many nonradial positive solutions of \eqref{eq0001} under some mild assumptions on the potentials $V_j(x)$ which are more general than $(V_0)$. In particular, we shall remove the symmetry assumption (\ref{PisSym}) and the smallness assumption (\ref{PisBeta}) in \cite[Theorem~1]{PV22}. We also include the results in the case of $N=2, d=2$, not covered in \cite{PW13}.

	To state our result precisely, let us first introduce some necessary notations and assumptions.  We make the following assumptions on $V_j(x)$:
	\begin{enumerate}
		\item[$(V_1)$]\quad $V_j(x)>0$ is continuous and $V_j(x',x'')=V_j(|x'|, |x''|)$ in $\bbr^N$, where $x'=(x_1,x_2)\in\bbr^2$ and $x''\in\bbr^{N-2}$;
		\item[$(V_2)$]\quad $V_j(x)=\lambda_j+\frac{ \delta_j}{|x|^{\nu_j}}+O\bigg(\frac{1}{|x|^{\nu_j+\ve}}\bigg)$ as $|x|\to+\infty$ in the $C^1$-sense, where $\lambda_j>0$, $\delta _j\in\bbr$ and $\nu_j,\ve>0$.
	\end{enumerate}
	\begin{remark}
		In the assumptions~$(V_1)$ and $(V_2)$, we do not require that the potentials $V_j(x)$ are symmetric at infinity, i.e. we do not assume that $\lambda_j=\lambda_i$ for all $i\not=j$. Note that in \cite{PV22, PW13} the symmetry assumption $\lambda_i=\lambda_j$ at $\infty$ plays key roles in the construction.
	\end{remark}
	By rearranging if necessary, we may assume that
	\begin{eqnarray}\label{eq0002}
		\lambda_1=\cdots=\lambda_{n_1}<\lambda_{n_1+1}=\cdots=\lambda_{n_2}<\cdots<\lambda_{n_{k-1}+1}=\cdots=\lambda_{n_k},
	\end{eqnarray}
	where $1\leq n_1<n_2<\cdots<n_k=d$ and $1\leq k\leq d$.  For the sake of simplicity, we denote $n_0=0$ and $\mathfrak{n}_{\tau}=\{n_{\tau-1}+1, n_{\tau-1}+2, \cdots, n_{\tau}\}$ for all $\tau=1,2,\cdots,k$.  By \cite[Lemma~3.7]{ACR07} and \eqref{eqnnnew0008},
	\begin{eqnarray}
		\int_{\bbr^N}|x|^{-\nu_j}w_j^2(\cdot-\xi)dx&=&(B_j+o(1))|\xi|^{-\nu_j},\label{eqnnnew1127}\\
		\int_{\bbr^N}w_j^3w_j(\cdot-\xi)dx&=&(C_j+o(1))|\xi|^{\frac{1-N}{2}}e^{-\sqrt{\lambda_{n_\tau}}|\xi|},\label{eqnnnew1126}\\
		\int_{\bbr^N}w_j^2w_i^2(\cdot-\xi)dx&=&\left\{\aligned&(D_\tau+o(1))|\xi|^{-\frac12}e^{-2\sqrt{\lambda_{n_\tau}}|\xi|},\quad N=2,\\
		&(D_\tau+o(1))|\xi|^{-2}e^{-2\sqrt{\lambda_{n_\tau}}|\xi|}\log|\xi|,\quad N=3,\endaligned\right.
		\label{eqnnnew1123}
	\end{eqnarray}
	for all $i,j\in\mathfrak{n}_{\tau}$ with all $\tau=1,2,\cdots,k$,
	\begin{eqnarray}\label{eqnnnew1124}
		\int_{\bbr^N}w_{n_{\tau}}^2w_{n_{\tau}+1}^2(\cdot-\xi)dx=(D_\tau'+o(1))|\xi|^{1-N}e^{-2\sqrt{\lambda_{n_\tau}}|\xi|}
	\end{eqnarray}
	for all $\tau=1,2,\cdots,k-1$ and
	\begin{eqnarray}\label{eqnnnew1125}
		\int_{\bbr^N}w_{n_{k}}^2w_{n_{1}}^2(\cdot-\xi)dx=(D_k'+o(1))|\xi|^{1-N}e^{-2\sqrt{\lambda_{n_1}}|\xi|},
	\end{eqnarray}
	as $|\xi|\to+\infty$, where $B_j$, $C_j$, $D_\tau$ and $D_\tau'$ are positive constants.  Moreover, it is also well known that by \eqref{eqnnnew0008}, the spectrum of $-\Delta+\lambda_j$ in $L^2(\bbr^N; w_i^2)$ is discrete for all $i,j$.  Let $\beta_{i,j,*}$ be the first eigenvalue of $-\Delta+\lambda_j$ in $L^2(\bbr^N; w_i^2)$ and we denote $\nu_*=\min\{\nu_j\}$ and $\mathfrak{m}_{*}=\{j=1,2,\cdots,d\mid \nu_j=\nu_*\}$.  Then
	our main result can be stated as follows.
	\begin{theorem}\label{thm0001}
		Let $N=2,3$, $d\geq2$, $(V_1)$--$(V_2)$ hold and $\beta_{i,j}$ be not an eigenvalue of $-\Delta+\lambda_j$ in $L^2(\bbr^N; w_i^2)$ for all $i,j$ with $\beta_{j,j+1}<\beta_{j,j+1,*}$ for all $j$.  Assume that $\nu_*>1$ and the following {\em pinching condition} is satisfied
\begin{equation}
\label{pinch}
\lambda_{n_k}<4\lambda_{n_1}.
\end{equation}
 Then \eqref{eq0001} has infinitely many nonradial positive solutions, provided that one of the following conditions are satisfied
		\begin{enumerate}
			\item[$(a)$]\quad $\sum_{j\in\mathfrak{m}_{*}}B_j\delta_j>0$, and $\sum_{j=n_{\tau-1}+1}^{n_{\tau}-1}\beta_{j,j+1}>0$, $\beta_{n_{\tau},n_{\tau}+1}>0$ for all $\tau=1,2,\cdots,k-1$,
			\item[$(b)$]\quad $\sum_{j\in\mathfrak{m}_{*}}B_j\delta_j<0$, and $\sum_{j=n_{\tau-1}+1}^{n_{\tau}-1}\beta_{j,j+1}<0$, $\beta_{n_{\tau},n_{\tau}+1}<0$ for all $\tau=1,2,\cdots,k-1$ in the case of $d\geq3$,
			\item[$(c)$]\quad $\sum_{j\in\mathfrak{m}_{*}}B_j\delta_j>0$, and $-2\pi^{-\frac12}D_1^{-1}C_1<\beta_{1,2}<0$ in the case of $N=2$ and $d=2$ with $\lambda_1=\lambda_2$ while, $\beta_{1,2}<0$ in the cases of $N=3$ and $d=2$ or $N=2$ and $d=2$ with $\lambda_1\not=\lambda_2$,
			\item[$(d)$]\quad $\sum_{j\in\mathfrak{m}_{*}}B_j\delta_j<0$, and $\beta_{1,2}<-2\pi^{-\frac12}D_1^{-1}C_1<0$ in the case of $N=2$ and $d=2$ with $\lambda_1=\lambda_2$,
		\end{enumerate}
		where $B_j,C_j,D_{\tau}>0$ are given by \eqref{eqnnnew1127}, \eqref{eqnnnew1126}, and \eqref{eqnnnew1123}, respectively.
	\end{theorem}

	\vskip0.2in
	
	In what follows, let us give several corollaries of Theorem~\ref{thm0001} to make it to be easier to understand.
	
\begin{corollary}\label{coro0005}
		Let $N=2,3$, $d=2$ and suppose that the assumptions $(V_1)$--$(V_2)$ hold.  If $\min_{j=1,2}\nu_j>1$ and $\lambda_{2}<4\lambda_{1}$ then
		\begin{enumerate}
        \item[$(1)$]\quad for $N=2$, \eqref{eq0001} has infinitely many nonradial positive solutions, provided
        		\begin{enumerate}
            \item[$(a)$]\quad $\sum_{j\in\mathfrak{m}_{*}}B_j\delta_j>0$ and $0<\beta_{1,2}<\beta_{1,2,*}$,
			\item[$(b)$]\quad $\sum_{j\in\mathfrak{m}_{*}}B_j\delta_j>0$ and $\beta_{1,2}<0$ in the case of $\lambda_1\not=\lambda_2$,
			\item[$(c)$]\quad $\sum_{j\in\mathfrak{m}_{*}}B_j\delta_j>0$ and $-2\pi^{-\frac12}D_1^{-1}C_1<\beta_{1,2}<0$ in the case $\lambda_1=\lambda_2$,
			\item[$(d)$]\quad $\sum_{j\in\mathfrak{m}_{*}}B_j\delta_j<0$ and $\beta_{1,2}<-2\pi^{-\frac12}D_1^{-1}C_1<0$ in the case of $\lambda_1=\lambda_2$,
        		\end{enumerate}
        \item[$(2)$]\quad for $N=3$, \eqref{eq0001} has infinitely many nonradial positive solutions, provided $\sum_{j\in\mathfrak{m}_{*}}B_j\delta_j>0$ and $\beta_{1,2}<\beta_{1,2,*}$.
		\end{enumerate}
	\end{corollary}
	
	By $(1)$ of Corollary~\ref{coro0005}, one can see that for the two coupled case $d=2$, multi-species nonlinear Schr\"odinger systems~\eqref{eq0001} has infinitely many nonradial positive solutions in dimension two $N=2$, under some suitable conditions on the potentials $V_j(x)$, if the coupling paramter $\beta_{1,2}$ is not very large.  Moreover, when the crossing happens at infinlty, that is, $\lambda_1=\lambda_2$, then $-2\pi^{\frac12}D_1^{-1}C_1$ is a jumping point of $\beta_{1,2}$ for the existence of infinitely many nonradial positive solutions.  $(2)$ of Corollary~\ref{coro0005} is a generalization of \cite[Theorem~1.2]{PW13} in two fronts.  First of all, we allow $\lambda_{2}<4\lambda_{1}$ in Corollary~\ref{coro0005} while  the symmetric condition $\lambda_1=\lambda_2$ is assumed in \cite[Theorem~1.2]{PW13}.  Secondly, we give a description of $\beta_*$, whose existence is asserted in \cite[Theorem~1.2]{PW13}, by proving that $\beta_*\geq\beta_{1,2,*}$.  In the symmetric condition $\lambda_1=\lambda_2$, by rearranging, it is well known that $\beta_{1,2,*}=\min\{\mu_1,\mu_2\}$.

	Next we discuss the case of $N=2, d=3$. 
	\begin{corollary}\label{coro0002}
		Let $N=2$, $d=3$, $\lambda_1<\lambda_2=\lambda_3$.  Suppose that the assumptions $(V_1)$--$(V_2)$ hold and $\beta_{i,j}$ is not an eigenvalue of $-\Delta+\lambda_j$ in $L^2(\bbr^N; w_i^2)$ for all $i,j=1,2,3$ and $i\not=j$ with $\beta_{j,j+1}<\beta_{j,j+1,*}$ for all $j=1,2$.  If $\min_{j=1,2,3}\nu_j>1$ and $\lambda_{3}<4\lambda_{1}$ then \eqref{eq0001} has infinitely many nonradial positive solutions, provided
		\begin{enumerate}
			\item[$(a)$]\quad $\sum_{j\in\mathfrak{m}_{*}}B_j\delta_j>0$ and $\beta_{1,2}>0$,
			\item[$(b)$]\quad $\sum_{j\in\mathfrak{m}_{*}}B_j\delta_j<0$ and $\beta_{1,2}<0$.
		\end{enumerate}
	\end{corollary}
	
	\begin{corollary}\label{coro0003}
		Let $N=2$, $d=3$, $\lambda_1=\lambda_2<\lambda_3$.  Suppose that the assumptions $(V_1)$--$(V_2)$ hold and $\beta_{i,j}$ is not an eigenvalue of $-\Delta+\lambda_j$ in $L^2(\bbr^N; w_i^2)$ for all $i,j=1,2,3$ and $i\not=j$ with $\beta_{j,j+1}<\beta_{j,j+1,*}$ for all $j=1,2$.  If $\min_{j=1,2,3}\nu_j>1$ and $\lambda_{3}<4\lambda_{1}$ then \eqref{eq0001} has infinitely many nonradial positive solutions, provided
		\begin{enumerate}
			\item[$(a)$]\quad $\sum_{j\in\mathfrak{m}_{*}}B_j\delta_j>0$ and $\beta_{1,2}+\beta_{2,3}>0$, $\beta_{1,2}>0$,
			\item[$(b)$]\quad $\sum_{j\in\mathfrak{m}_{*}}B_j\delta_j<0$ and $\beta_{1,2}+\beta_{2,3}<0$, $\beta_{1,2}<0$.
		\end{enumerate}
	\end{corollary}
	
	\begin{corollary}\label{coro0004}
		Let $N=2$, $d=3$, $\lambda_1<\lambda_2<\lambda_3$.  Suppose that the assumptions $(V_1)$--$(V_2)$ hold and $\beta_{i,j}$ is not an eigenvalue of $-\Delta+\lambda_j$ in $L^2(\bbr^N; w_i^2)$ for all $i,j=1,2,3$ and $i\not=j$ with $\beta_{j,j+1}<\beta_{j,j+1,*}$ for all $j=1,2$.  If $\min_{j=1,2,3}\nu_j>1$ and $\lambda_{3}<4\lambda_{1}$ then \eqref{eq0001} has infinitely many nonradial positive solutions, provided
		\begin{enumerate}
			\item[$(a)$]\quad $\sum_{j\in\mathfrak{m}_{*}}B_j\delta_j>0$ and $\beta_{1,2},\beta_{2,3}>0$,
			\item[$(b)$]\quad $\sum_{j\in\mathfrak{m}_{*}}B_j\delta_j<0$ and $\beta_{1,2},\beta_{2,3}<0$.
		\end{enumerate}
	\end{corollary}
	
	By Corollaries~\ref{coro0002}, \ref{coro0003} and \ref{coro0004}, we can see that for the three-coupled case $d=3$, under some suitable conditions on the potentials $V_j(x)$,  multi-species nonlinear Schr\"odinger systems~\eqref{eq0001} in dimension two $N=2$ has infinitely many nonradial positive solutions with some suitable restrictions only on $\beta_{1,2}$ and $\beta_{2,3}$, which also allows the couplings $\{\beta_{i,j}\}$ to be mixed.
	
Finally in the case of $d\geq 3$ we have the following
	
	\begin{corollary}\label{coro0006}
		Let $N=2,3$, $d\geq3$ and suppose that the assumptions $(V_1)$--$(V_2)$ hold.  If $\lambda_j=\lambda$, $\mu_j=\mu$, $\delta_j=\delta$ and $\nu_j=\nu>1$ for all $j$, $\beta_{i,j}=\beta$ for all $i,j$, then \eqref{eq0001} has infinitely many nonradial positive solutions, provided
		\begin{enumerate}
			\item[$(a)$]\quad $\delta>0$ and $\beta<\mu$,
			\item[$(b)$]\quad $\delta<0$ and $\beta<0$.
		\end{enumerate}
	\end{corollary}
	
	Corollary~\ref{coro0006} is a generalization of \cite[Theorem~1.1]{PV22} in the sense that, we remove the smallness assumption (\ref{PisBeta}) in \cite[Theorem~1]{PV22} by giving a uniformly upper bound of $\beta$ for the existence of infinitely many
    nonradial positive solution of \eqref{eq0001} in the case of $\delta>0$ and by showing that there is no lower bound of $\beta$ for the existence of infinitely many nonradial positive solution of \eqref{eq0001} in the case of $\delta<0$.
	
	\subsection{Further remarks}
	Our strategy is still to use the Ljapunov-Schmidt reduction arguments to construct infinitely many solutions of \eqref{eq0001}.  In this procedure, by the assumption~$(V_2)$, \eqref{eq0003} and \eqref{eqnnnew0008},
	\begin{eqnarray*}
		\mathbf{W}=(\sum_{t=1}^{\vartheta}w_1(x-\eta_{t,1}), \sum_{t=1}^{\vartheta}w_2(x-\eta_{t,2}), \cdots, \sum_{t=1}^{\vartheta}w_d(x-\eta_{t,d}))
	\end{eqnarray*}
	is a natural approximation of \eqref{eq0001} if
	\begin{eqnarray*}
		\min_{(t,j)\not=(s,i)}|\eta_{t,j}-\eta_{s,i}|\to+\infty\quad\text{and}\quad\min_{t,j}|\eta_{t,j}|\to+\infty
	\end{eqnarray*}
	as $\vartheta\to+\infty$.
	To continue, we need to construct a correction $\widetilde{\mathbf{v}}_*$ which is much smaller than the approximation $\mathbf{W}$ in a suitable sense and then, to prove that $\mathbf{W}+\widetilde{\mathbf{v}}_*$ is a genuine solution of \eqref{eq0001} by adjusting the locations of $\{\eta_{t,j}\}$.
	
	\vskip0.12in
	
	As that in \cite{PW13,PV22,WY10}, we want to put $\{\eta_{t,j}\}$ on a circle with a large radius $\rho$ to reduce the number of parameters in adjusting the locations of $\{\eta_{t,j}\}$ and
	using the number of bumps of the solutions, as a parameter in the construction of spiked solutions of \eqref{eq0001}.    In \cite{PW13,PV22,WY10}, the locations of $\{\eta_{t,j}\}$ are constructed to be invariant under the rotation of the angles $\frac{2\pi}{\vartheta}$ or $\frac{\pi}{\vartheta}$, respectively.  Thus, the adjustment of the locations of $\{\eta_{t,j}\}$ will be reduced to the solvability of a one-dimensional equation of $\vartheta$ under these invariance.  {\bf However, we do not assume $\lambda_j\equiv\lambda$ for all $j$ or $\beta_{i,j}\equiv\beta$ for all $i$ and $j$ in Theorem~\ref{thm0001}, thus, the limit system of \eqref{eq0001} at infinity can not be rotationally invariant anymore.}  This requires us to introduce more parameters in adjusting the locations of  $\{\eta_{t,j}\}$.  Note that the parameter $\rho$ used in \cite{PW13,PV22,WY10} is in the normal direction of the circle.  Therefore, we shall introduce another $d$ parameters in the tangential direction of the circle in adjusting the locations of $\{\eta_{t,j}\}$.  The locations of the spikes $\{\eta_{t,j}\}$ is roughly summarized in the Figure~\ref{Picture01} and we refer the readers to \eqref{eq0004} for more details.

\begin{figure}[htbp]
  \centering
  \includegraphics[scale=0.07]{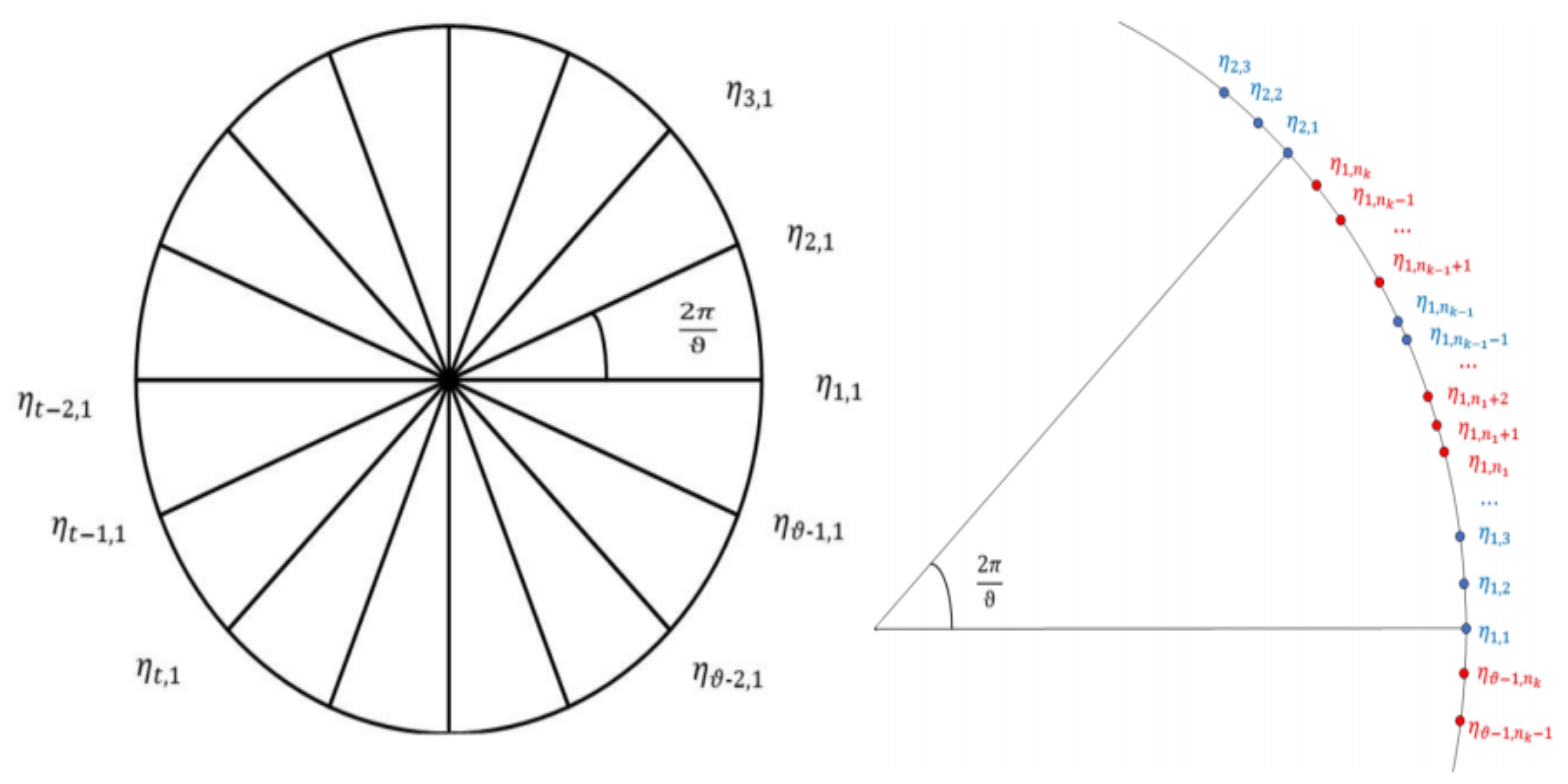}\\
  \caption{The locations of $\{\eta_{t,j}\}$}\label{Picture01}
\end{figure}

	\vskip0.12in
	
	As pointed out in \cite{PV22}, due to the linear coupling term, the correction $\widetilde{\mathbf{v}}_*$ needs to be divided into two parts, say $\mathbf{Q}_*$ (the main term) and $\mathbf{v}_{**}$ (the high order term), even in the symmetric case where the potentials $V_j(x)=V_i(x)$ for all $j,i$ satisfy the symmetric condition~$(V_*)$.
{\bf In this paper, by using the representation formula and making careful and almost sharp estimates, we find out the leading term of the main term of the correction in the reduced problem in the general case~$(V_1)$--$(V_2)$, see Lemma~\ref{lem0006} for more details.}  This gives a precise expansion of the reduced energy functional up to the leading order term in the general case~$(V_1)$--$(V_2)$.  Surprisingly, our estimates show that due to the crossing interaction among the peaks of different components for $i,j\in\mathfrak{n}_{\tau}$ with all $\tau=1,2,\cdots,k$, the correction $\widetilde{\mathbf{v}}_*$ can be negligible in the reduced energy functional for these terms labelled by $j\not=n_{\tau}$ and the main terms of the correction do have contributions to the reduced energy functional for these terms labelled by $j=n_{\tau}$, see \eqref{eqnn9989} and \eqref{eqnnn9989} for more details.  With these two precise expansions of the reduced energy functional up to the leading order term, we observed that, according to our construction of $\{\eta_{t,j}\}$, there is only the interaction among the peaks of different components in the tangential direction.  Moreover, if $\sum_{j=n_{\tau-1}+1}^{n_{\tau}-1}\beta_{j,j+1}>0$ and $\beta_{n_{\tau},n_{\tau}+1}>0$ for all $\tau=1,2,\cdots,k-1$, then the reduced energy functional takes maximum in this direction while, if $\sum_{j=n_{\tau-1}+1}^{n_{\tau}-1}\beta_{j,j+1}<0$ and $\beta_{n_{\tau},n_{\tau}+1}<0$ for all $\tau=1,2,\cdots,k-1$, then the reduced energy functional takes minimum in this direction.  Thus, according to the assumption~\eqref{eq0002}, the balanced conditions in the tangential direction should obey the relations $|\eta_{t,j-1}-\eta_{t,j}|=(1+o(1))|\eta_{t,j}-\eta_{t,j+1}|$ for all $j=n_{\tau-1}+2,\cdots,n_{\tau}-1$ with all $\tau=1,2,\cdots,k$, $|\eta_{t,n_{\tau}-1}-\eta_{t,n_{\tau}}|>|\eta_{t,n_{\tau}}-\eta_{t,n_{\tau}+1}|$ for all $\tau=1,2,\cdots,k-1$ and $|\eta_{t,n_{k}-1}-\eta_{t,n_{k}}|=(1+o(1))|\eta_{t,n_{1}-1}-\eta_{t,n_{1}}|$.  We refer the readers to the Figure~\ref{Picture01} for a better understanding of the balanced conditions in the tangential direction stated above.  On the other hand, as that observed in \cite{PW13,PV22}, the potential effect, the interplay between peaks of the same component and the interaction among the peaks of different components will compete in the normal direction.  For $d\geq3$, thanks to the precise expansion of the reduced energy functional up to the leading order term in the general case~$(V_1)$--$(V_2)$, we observed that the interaction among the peaks of different components always dominates the interplay between peaks of the same component.  Thus, the competition in the normal direction is reduced to the potential effect and the interaction among the peaks of different components for $d\geq3$, {\bf which help us to give a positive answer to the conjecture in \cite{PV22} and remove the smallness assumption (\ref{PisBeta}) in \cite[Theorem~1.1]{PV22} even in the general case~$(V_1)$--$(V_2)$.}  For $d=2$, again, thanks to the precise expansion of the reduced energy functional up to the leading order term in the general case~$(V_1)$--$(V_2)$, we observed that the situation is completed different, that is, the interplay between peaks of the same component always dominates the interaction among the peaks of different components except in the symmetric case $\lambda_1=\lambda_2$ of dimension two $N=2$ while, in this case, according to the crossing phenomenon of the interaction among the peaks of different components, the interplay between peaks of the same component and the interaction among the peaks of different components are the same order term.  Thus, the competition in the normal direction is always reduced to the potential effect and the interaction among the peaks of the same component for $d=2$ except in the symmetric case $\lambda_1=\lambda_2$ of dimension two $N=2$ while, in this case, the second term is generated by the competition of the interplay between peaks of the same component and the interaction among the peaks of different components.  Roughly speaking, in the normal direction, if the interplay between peaks of the same component dominates the interaction among the peaks of different components then the reduced energy functional takes maximum in this direction for $\sum_{j\in\mathfrak{m}_{*}}B_j\delta_j>0$ while, the reduced energy functional takes minimum in this direction for $\sum_{j\in\mathfrak{m}_{*}}B_j\delta_j<0$ if the interaction among the peaks of different components dominates the interplay between peaks of the same component.  Now, summarizing the above observations, we can use variational arguments to find out the balanced conditions in the normal direction and in the tangential direction to solve the reduced problem by taking the maximum of the reduced energy functional over a suitable set in the case~$(a)$ of Theorem~\ref{thm0001} and taking the minimum of the reduced energy functional over a suitable set in the cases~$(b)$ and $(d)$ of Theorem~\ref{thm0001} while, a min-max variational argument to the reduced energy functional over a suitable set is needed in the case~$(c)$ of Theorem~\ref{thm0001}, see Proposition~\ref{prop0002} for more details.  {\bf We remark that the reduced problem is much more delicate than that of \cite{PW13,PV22,WY10} and the competition of the interplay between peaks of the same component and the interaction among the peaks of different components also generates a new phenomenon for $N=2$ and $d=2$, that is, there is a jumping point of the coupling $\beta_{1,2}$ for the existence of infinitely many nonradial positive solutions of \eqref{eq0001} in the symmetric case $\lambda_1=\lambda_2$, which has never been observed in the literature to our best knowledge.}

	\vskip0.12in
	
It is also worth pointing out that the conditions~$(a)$--$(d)$ is {\bf almost optimal} since these conditions completely capture the leading order terms of the reduced energy functional.  Moreover, the conditions $\beta_{j,j+1}<\beta_{j,j+1,*}$ for all $j$ {\bf also seems to be necessary} in constructing infinitely many nonradial positive solutions of \eqref{eq0001}, since the potential solution $\mathbf{U}=\mathbf{W}+\mathbf{Q}_*+\mathbf{v}_{**}$, constructed by us in the proof of Theorem~\ref{thm0001}, can not be positive anymore for $\vartheta$ sufficiently large if $\beta_{j,j+1}>\beta_{j,j+1,*}$ for some $j$, see Remark~\ref{rmk0001} for more details.  Thus, it will be very interesting to construct infinitely many nonradial positive solutions of \eqref{eq0001} if $\beta_{j,j+1}>\beta_{j,j+1,*}$ for some $j$.  On the other hand, the pinching condition~\eqref{pinch} is also crucial in our construction of the correction $\widetilde{\mathbf{v}}_*$ in proving Theorem~\ref{thm0001}, which can be understood as a condition that the natural approximation
$\mathbf{W}$ is very close to a genuine solution of \eqref{eq0001}, since the natural approximation $\mathbf{W}$ dominates every step of our construction under this condition.  We believe this pinching condition~\eqref{pinch} is {\bf optimal} in the constructions of solutions of \eqref{eq0001} which are started from the natural approximation
$\mathbf{W}$.  Thus, it will also be very interesting to construct infinitely many nonradial positive solutions of \eqref{eq0001} without this pinching condition~\eqref{pinch}.

	\vskip0.12in
	
A very challenging problem for \eqref{eq0001} is the existence of infinitely many positive solutions when the potentials $V_j(x)$ are not radially symmetric.  In \cite{dWY15}, infinitely many nonradial positive solutions are constructed to the single scalar equation~(\ref{scalar}) with nonradial potential $V(x)$ satisfying the decaying property as same as that of $(V_0)$. In a future work, we will remove all the radial symmetries of potentials for the full system \eqref{eq0001}.

\medskip

	\noindent{\bf\large Notations.} Throughout this paper, $C$ and $C'$ are indiscriminately used to denote various absolutely positive constants.  $a\sim b$ means that $C'b\leq a\leq Cb$ and $a\lesssim b$ means that $a\leq Cb$.

	\section{The first approximation}
	For every $j\in\mathfrak{n}_\tau$ with $\tau=1,2,\cdots,k$, we define
	\begin{eqnarray}\label{eq0004}
		\xi_{t,j}=\left\{\aligned&\bigg(\cos\bigg(\theta_{j}+\frac{2(t-1)\pi}{\vt}\bigg), \sin\bigg(\theta_{j}+\frac{2(t-1)\pi}{\vt}\bigg)\bigg),\quad N=2,\\
		&\bigg(\cos\bigg(\theta_{j}+\frac{2(t-1)\pi}{\vt}\bigg), \sin\bigg(\theta_{j}+\frac{2(t-1)\pi}{\vt}\bigg), 0\bigg),\quad N=3,
		\endaligned\right.
	\end{eqnarray}
	where $\theta_{j}=\sum_{i=0}^{j-1}\alpha_{i}$
	with $\alpha_0=o(\vartheta^{-1})$, $t=1,2,\cdots, \vt$ for $\vt\in\bbn$ sufficiently large and
	\begin{eqnarray*}
		\alpha_j\sim\frac{2\pi}{\vartheta}-\sum_{i=0}^{d-1}\alpha_i\sim\frac{1}{\vt}
	\end{eqnarray*}
	for all $j=1,2,\cdots,d-1$.  For the sake of simplicity, we denote $\alpha_d=\frac{2\pi}{\vartheta}-\sum_{i=0}^{d-1}\alpha_i$.
	For every $j\in\mathfrak{n}_\tau$ with $\tau=1,2,\cdots,k$, we also define
	\begin{eqnarray*}\label{eq0006}
		W_j=\sum_{t=1}^{\vt}\widetilde{w}_{t,j}
	\end{eqnarray*}
	where $\widetilde{w}_{t,j}=w_{j}(r_{t,j})$ with $r_{t,j}=|x-\rho_j\xi_{t,j}|$ and $\rho_\tau>0$ sufficiently large satisfying
	\begin{eqnarray*}
		\rho_{n_{\tau-1}+1}=\rho_{n_{\tau-1}+2}=\cdots=\rho_{n_{\tau}}=\rho+O(1)\quad\text{for all }\tau=1,2,\cdots,k
	\end{eqnarray*}
	with $\rho\sim \vt\log \vt$ being chosen later.
	By the definition of $\xi_{t,j}$ and the radial symmetry of $w_j$, we know that $W_j$ is invariant under the rotation of the angle $\frac{2\pi}{\vt}$ for every $j\in \mathfrak{n}_{\tau}$ with all $\tau=1,2,\cdots,k$.  We denote
	\begin{eqnarray*}\label{eq0007}
		\eta_{t,j}=\rho_j\xi_{t,j}.
	\end{eqnarray*}
	Then for $\vartheta>0$ sufficiently large, by the geometry of the constructions of $\{\xi_{t,j}\}$,
	\begin{eqnarray*}\label{eqn0008}
		\eta_{t,j}-\eta_{s,i}=2\rho\sin\bigg(\bigg|\frac{\theta_{j}-\theta_{i}}{2}+\frac{(t-s)\pi}{\vt}\bigg|\bigg)+h.o.t.
	\end{eqnarray*}
	for all $(t,j)\not=(s,i)$.  In particular,
	\begin{eqnarray}\label{eq0009}
		\widetilde{\eta}_{j}=\min_{m\not=n}|\eta_{m,j}-\eta_{n,j}|=|\eta_{1,j}-\eta_{2,j}|=\frac{2\pi\rho}{\vt}+h.o.t..
	\end{eqnarray}
	If $n_\tau-n_{\tau-1}>1$, then we can define
	\begin{eqnarray}\label{eq0010}
		\widehat{\eta}_{\tau}=\min_{i\not=j;i,j\in\mathfrak{n}_\tau;m,n\in\bbn}|\eta_{m,i}-\eta_{n,j}|=\rho\widehat{\alpha}_\tau+h,o,t.,
	\end{eqnarray}
	where
	\begin{equation*}
		\widehat{\alpha}_\tau=\min\{\min\limits_{\mathfrak{n}_{\tau-1}+1\le j<\mathfrak{n}_{\tau}}{\alpha_j},\frac{2\pi}{\vt}-\sum_{j=\mathfrak{n}_{\tau-1}+1}^{\mathfrak{n}_{\tau}-1}\alpha_j\}.
	\end{equation*}
	Clearly, $\widehat{\eta}_{\tau}<\widetilde{\eta}_{j}$ for all $j\in\mathfrak{n}_\tau$ and all $\tau=1,2,\cdots,k$.
	Finally, we denote
	\begin{eqnarray}\label{eq0011}
		\widehat{\eta}=\min_{(m,i)\not=(n,j)}|\eta_{m,i}-\eta_{n,j}|\quad\text{and}\quad\widehat{\alpha}=\min_{1\le j\le d}{\alpha}_j.
	\end{eqnarray}
	
	\vskip0.12in
	
	For every $j=1,2,\cdots,d$, we define
	\begin{eqnarray*}
		U_j=W_j+\varphi_j,
	\end{eqnarray*}
	then by \eqref{eq0002}, $\mathbf{U}=(U_1, U_2, \cdots, U_d)$ is a solution of \eqref{eq0001} if and only if $\varphi=(\varphi_1,\varphi_2,\cdots,\varphi_d)$ is a solution of the following system:
	\begin{equation}\label{eq0012}
		\left\{\aligned&\mathcal{L}_j(\varphi)=\sum_{l=1}^3E_{j,l}+N_j,\quad\text{in }\bbr^N,\\
		&\varphi_j(x)\to0\quad\text{as }|x|\to+\infty,\quad j=1,2,\cdots,d,\endaligned\right.
	\end{equation}
	where
	\begin{eqnarray*}
		\mathcal{L}_j(\varphi)=-\Delta \varphi_j+V_j(x)\varphi_j-3\mu_jW_j^2\varphi_j-\sum_{i=1;i\not=j}^{d}\beta_{i,j}(W_i^2\varphi_j+2W_iW_j\varphi_i)
	\end{eqnarray*}	
	is the linear operator,
	\begin{eqnarray}\label{eq0013}
		E_{j,1}=(\lambda_j-V_j(x))W_j,\ \ E_{j,2}=\mu_j(W_j^3-\sum_{t=1}^{\vartheta}w_{t,j}^3),\ \ E_{j,3}=\sum_{i=1;i\not=j}^{d}\beta_{i,j}W_i^2W_j
	\end{eqnarray}
	are the errors and
	\begin{eqnarray*}
		N_j=\sum_{i=1;i\not=j}^{d}\beta_{i,j}(2W_i\varphi_i\varphi_j+W_j\varphi_i^2+\varphi_i^2\varphi_j)+3\mu_j W_j\varphi_j^2+\mu_j\varphi_j^3
	\end{eqnarray*}
	is the nonlinear part.
	\begin{lemma}\label{lem0001}
		Suppose the assumptions~$(V_1)$--$(V_2)$ hold.  If $\lambda_{n_k}<4\lambda_{n_1}$ under the condition~\eqref{eq0002} then for $\vartheta$ sufficiently large, we have
		\begin{eqnarray}
			|E_{j,1}|&\lesssim&\rho^{-\nu_j}\bigg(\sum_{t=1}^{\vartheta}(1+r_{t,j})^{-(\frac{N-1}{2}-\sigma)}e^{-\sqrt{\lambda_{n_\tau}}r_{t,j}}\chi_{\{r_{t,j}\leq\frac{\widetilde{\eta}_{j}}{2}\}}\notag\\
			&&+\bigg(\sum_{t=1}^{\vartheta}(1+r_{t,j})^{-(\frac{N-1}{2}-\sigma)}e^{-\sqrt{\lambda_{n_\tau}}r_{t,j}}\bigg)\chi_{\cap_{t=1}^{\vartheta}\{r_{t,j}\geq\frac{\widetilde{\eta}_{j}}{2}\}}\bigg),\label{eqnnnew3001}
		\end{eqnarray}
		\begin{eqnarray}\label{eqnnnew3012}
			|E_{j,2}|&\lesssim& \widetilde{\eta}_{j}^{\frac{1-N}{2}}e^{-\sqrt{\lambda_{n_\tau}}\widetilde{\eta}_{j}}\bigg(\sum_{t=1}^{\vartheta}(1+r_{t,j})^{-(\frac{N-1}{2}-\sigma)}
			e^{-\sqrt{\lambda_{n_\tau}}r_{t,j}}\chi_{\{r_{t,j}
				\leq\frac{\widetilde{\eta}_{j}}{2}\}}\notag\\
			&&+\bigg(\sum_{t=1}^{\vartheta}(1+r_{t,j})^{-(\frac{N-1}{2}-\sigma)}e^{-\sqrt{\lambda_{n_\tau}}r_{t,j}}\bigg)
			\chi_{\cap_{t=1}^{\vartheta}\{r_{t,j}>\frac{\widetilde{\eta}_{j}}{2}\}}\bigg)
		\end{eqnarray}
		and
		\begin{eqnarray}\label{eqnnnew3023}
			|E_{j,3}|&\lesssim&\widehat{\eta}^{\frac{1-N}{2}}e^{-\sqrt{\lambda_{n_1}}\widehat{\eta}}
			\bigg(\sum_{\tau=1}^{k}\sum_{i=n_{\tau-1}+1}^{n_{\tau}}\sum_{t=1}^{\vartheta}(1+r_{t,i})^{-(\frac{N-1}{2}-\sigma)}e^{-\sqrt{\lambda_{n_{\tau}}}r_{t,i}}
			\chi_{\{r_{t,i}\leq\delta\widehat{\eta}\}}\notag\\
			&&+\bigg(\sum_{\tau=1}^{k}\sum_{i=n_{\tau-1}+1}^{n_{\tau}}\sum_{t=1}^{\vartheta}(1+r_{t,i})^{\frac{1-N}{2}}e^{-\varsigma\sqrt{\lambda_{n_{\tau}}}r_{t,i}}\bigg)\chi_{\cap_{i=1}^{d}\cap_{t=1}^{\vartheta}
				\{r_{t,i}>\delta\widehat{\eta}\}}\bigg),
		\end{eqnarray}
		where $\sigma>\max\{\nu_j,1\}$, $\varsigma\in(0, 1)$ is sufficiently small such that
		\begin{eqnarray}\label{eq3017}
			(2-\varsigma)\sqrt{\lambda_{n_1}}\geq\sqrt{\lambda_{n_\tau}}\quad\text{for all }\tau=2,3,\cdots,k.
		\end{eqnarray}
		and
		\begin{eqnarray}\label{eq3018}
			\delta\in\bigg(0, \frac{\sqrt{\lambda_{n_2}}-\sqrt{\lambda_{n_1}}}{\sqrt{\lambda_{n_2}}-\varsigma\sqrt{\lambda_{n_1}}}\bigg).
		\end{eqnarray}
	\end{lemma}
	\begin{proof}
		Let us begin with the proof by estimating $|E_{j,1}|$.
		For $|x|<\frac{\rho}{3}$, by the triangle inequality, we have $r_{t,j}=|x-\eta_{t,j}|\geq\frac{\rho}{2}\gtrsim\widetilde{\eta}_j$ for all $t$, where $\widetilde{\eta}_j$ is given by \eqref{eq0009}.  Thus, by the assumptions~$(V_1)$ and $(V_2)$ and \eqref{eqnnnew0008},
		\begin{eqnarray}\label{eq3001}
			|E_{j,1}|\lesssim \sum_{t=1}^{\vt}r_{t,j}^{\frac{1-N}{2}}e^{-\sqrt{\lambda_{n_\tau}}r_{t,j}}\lesssim\rho^{-\sigma}\sum_{t=1}^{\vt}r_{t,j}^{-(\frac{N-1}{2}-\sigma)}e^{-\sqrt{\lambda_{n_\tau}}r_{t,j}}
		\end{eqnarray}
		for $|x|<\frac{\rho}{3}$.
		For $|x|\geq\frac{\rho}{3}$, by the assumption~$(V_2)$ and \eqref{eqnnnew0008},
		\begin{eqnarray}\label{eq3002}
			|E_{j,1}|\lesssim \rho^{-\nu_j}\sum_{t=1}^{\vt}(1+r_{t,j})^{-\frac{N-1}{2}}e^{-\sqrt{\lambda_{n_\tau}}r_{t,j}}.
		\end{eqnarray}
		Now if $r_{t,j}\leq\frac{\widetilde{\eta}_{j}}{2}$ for some $t$, then $r_{s,j}\geq\frac{\widetilde{\eta}_{j}}{2}$ for all $s\not=t$ and $|x|\geq |\eta_{t,j}|-r_{t,j}\ge \frac{\rho}{3}$.  Thus, by similar arguments as that used in \cite[Lemma~2.1]{WW21},
		\begin{eqnarray}\label{eq3003}
			\sum_{s=1}^\vt(1+r_{s,j})^{\frac{1-N}{2}}e^{-\sqrt{\lambda_{n_\tau}}r_{s,j}}
			\lesssim (1+r_{t,j})^{\frac{1-N}{2}}e^{-\sqrt{\lambda_{n_\tau}}r_{t,j}},
		\end{eqnarray}
		if $r_{t,j}\leq\frac{\widetilde{\eta}_{j}}{2}$ for some $t$,
		which, together with \eqref{eq3002}, implies
		\begin{eqnarray}\label{eq3004}
			|E_{j,1}|\lesssim\rho^{-\nu_j}(1+r_{t,j})^{-(\frac{N-1}{2}-\sigma)}e^{-\sqrt{\lambda_{n_\tau}}r_{t,j}}
		\end{eqnarray}
		for $|x|\geq\frac{\rho}{3}$ and $r_{t,j}\leq\frac{\widetilde{\eta}_{j}}{2}$ for some $t$.
		On the other hand, if $|x|\geq\frac{\rho}{3}$ and $r_{t,j}>\frac{\widetilde{\eta}_{j}}{2}$ for all $t$, then by \eqref{eq3002},
		\begin{equation}\label{eq3005}
			|E_{j,1}|\lesssim \rho^{-\nu_j}\bigg(\sum_{t=1}^{\vt}(1+r_{t,j})^{-(\frac{N-1}{2}-\sigma)}e^{-\sqrt{\lambda_{n_\tau}}r_{t,j}}\bigg).
		\end{equation}
		Thus by \eqref{eq3001}, \eqref{eq3004} and \eqref{eq3005}, we have \eqref{eqnnnew3001}.
		We next estimate $|E_{j,2}|$.  By \eqref{eq3003}, it is easy to see that
		\begin{eqnarray}\label{eq3007}
			|E_{j,2}|\sim\sum_{m\not=n}\widetilde{w}_{m,j}^2\widetilde{w}_{n,j}+\sum_{m\not=n,m\neq l,n\neq l}\widetilde{w}_{m,j}\widetilde{w}_{n,j}\widetilde{w}_{l,j}
			\sim\sum_{m\not=n}\widetilde{w}_{m,j}^2\widetilde{w}_{n,j}
		\end{eqnarray}
		if $r_{t,j}\leq\frac{\widetilde{\eta}_{j}}{2}$ for some $t$. In this case, $r_{s,j}\geq\frac{\widetilde{\eta}_{j}}{2}$ for all $s\not=t$.  Thus, similar as \eqref{eq3003}, by similar arguments as that used in \cite[Lemma~2.1]{WW21},
		\begin{eqnarray}
			\sum_{m\not=n; n,m\not=t}\widetilde{w}_{m,j}^2\widetilde{w}_{n,j}&\lesssim&\widetilde{\eta}_{j}^{\frac{3(1-N)}{2}}\sum_{m\not=n, n,m\not=t}e^{-\sqrt{\lambda_{n_\tau}}(2r_{m,j}+r_{n,j})}\notag\\
			&\lesssim&\widetilde{\eta}_{j}^{\frac{3(1-N)}{2}}(\sum_{m\not=t}e^{-\sqrt{\lambda_{n_\tau}}r_{m,j}})^2\sum_{n\not=t}e^{-\sqrt{\lambda_{n_\tau}}r_{n,j}}\notag\\
			&\lesssim&\widetilde{\eta}_{j}^{\frac{3(1-N)}{2}}e^{-\sqrt{\lambda_{n_\tau}}\widetilde{\eta}_{j}}e^{-\sqrt{\lambda_{n_\tau}}\frac{\widetilde{\eta}_{j}}{2}}\notag\\
			&\lesssim&\widetilde{\eta}_{j}^{1-N}e^{-\sqrt{\lambda_{n_\tau}}\widetilde{\eta}_{j}}(1+r_{t,j})^{-(\frac{N-1}{2}-\sigma)}e^{-\sqrt{\lambda_{n_\tau}}r_{t,j}}.\label{eq3008}
		\end{eqnarray}
		Similarly,
		\begin{eqnarray}
			\sum_{m\not=t}\widetilde{w}_{m,j}^2\widetilde{w}_{t,j}&\lesssim&(1+r_{t,j})^{\frac{1-N}{2}}e^{-\sqrt{\lambda_{n_\tau}}r_{t,j}}\sum_{m\not=t}r_{m,j}^{1-N}e^{-2\sqrt{\lambda_{n_\tau}}r_{m,j}}\notag\\
			&\lesssim&\widetilde{\eta}_{j}^{1-N}(\sum_{m\not=t}e^{-\sqrt{\lambda_{n_\tau}}r_{m,j}})^2(1+r_{t,j})^{\frac{1-N}{2}}e^{-\sqrt{\lambda_{n_\tau}}r_{t,j}}\notag\\
			&\lesssim&\widetilde{\eta}_{j}^{1-N}e^{-\sqrt{\lambda_{n_\tau}}\widetilde{\eta}_{j}}(1+r_{t,j})^{-(\frac{N-1}{2}-\sigma)}e^{-\sqrt{\lambda_{n_\tau}}r_{t,j}}\label{eq3009}
		\end{eqnarray}
		and
		\begin{eqnarray}
			\sum_{m\not=t}\widetilde{w}_{t,j}^2\widetilde{w}_{m,j}&\lesssim&(1+r_{t,j})^{1-N}e^{-\sqrt{\lambda_{n_\tau}}r_{t,j}}\sum_{m\not=t}r_{m,j}^{\frac{1-N}{2}}e^{-\sqrt{\lambda_{n_\tau}}(r_{m,j}+r_{t,j})}\notag\\
			&\lesssim&\widetilde{\eta}_{j}^{\frac{1-N}{2}}(\sum_{m\not=t}e^{-\sqrt{\lambda_{n_\tau}}|\eta_{m,j}-\eta_{t,j}|})(1+r_{t,j})^{1-N}e^{-\sqrt{\lambda_{n_\tau}}r_{t,j}}\notag\\
			&\lesssim&\widetilde{\eta}_{j}^{\frac{1-N}{2}}e^{-\sqrt{\lambda_{n_\tau}}\widetilde{\eta}_{j}}(1+r_{t,j})^{-(\frac{N-1}{2}-\sigma)}e^{-\sqrt{\lambda_{n_\tau}}r_{t,j}}.\label{eq3010}
		\end{eqnarray}
		Thus, by \eqref{eq3007},
		\begin{eqnarray*}
			|E_{j,2}|\lesssim \widetilde{\eta}_{j}^{\frac{1-N}{2}}e^{-\sqrt{\lambda_{n_\tau}}\widetilde{\eta}_{j}}(1+r_{t,j})^{-(\frac{N-1}{2}-\sigma)}e^{-\sqrt{\lambda_{n_\tau}}r_{t,j}}
		\end{eqnarray*}
		if $r_{t,j}\leq\frac{\widetilde{\eta}_{j}}{2}$ for some $t$.  If $r_{t,j}>\frac{\widetilde{\eta}_{j}}{2}$ for all $t$, then
		\begin{eqnarray}
			|E_{j,2}|&\lesssim& \big(\sum_{s=1}^\vt\widetilde{w}_{s,j}\big)^2\sum_{t=1}^\vt\widetilde{w}_{t,j}\notag\\
			&\lesssim&\widetilde{\eta}_{j}^{1-N}e^{-\sqrt{\lambda_{n_\tau}}\widetilde{\eta}_{j}}\sum_{t=1}^{\vartheta}(1+r_{t,j})^{-\frac{N-1}{2}}e^{-\sqrt{\lambda_{n_\tau}}r_{t,j}}\notag\\
			&\lesssim&\widetilde{\eta}_{j}^{1-N}e^{-\sqrt{\lambda_{n_\tau}}\widetilde{\eta}_{j}}
			\sum_{t=1}^{\vartheta}(1+r_{t,j})^{-(\frac{N-1}{2}-\sigma)}e^{-\sqrt{\lambda_{n_\tau}}r_{t,j}}.\label{eq3011}
		\end{eqnarray}
		Thus, by \eqref{eq3010} and \eqref{eq3011}, we have \eqref{eqnnnew3012}.  We finally estimate $|E_{j,3}|$.  Clearly, by \eqref{eq0013},
		\begin{eqnarray*}
			|E_{j,3}|&=&|E_{j,3,1}+E_{j,3,2}|\\
			&\lesssim&\sum_{i\in\mathfrak{n}_\tau,i\not=j}\sum_{n,m}\widetilde{w}_{m,i}^2\widetilde{w}_{n,j}
			+\sum_{i\in\mathfrak{n}_\tau',\tau\not=\tau'}\sum_{n,m}\widetilde{w}_{m,i}^2\widetilde{w}_{n,j},
		\end{eqnarray*}
		where
		\begin{eqnarray*}
			E_{j,3,1}=\sum_{i\in\mathfrak{n}_\tau,i\not=j}\sum_{n,m}\beta_{i,j}\widetilde{w}_{m,i}^2\widetilde{w}_{n,j}\quad\text{and}\quad
			E_{j,3,2}=\sum_{i\in\mathfrak{n}_\tau',\tau\not=\tau'}\sum_{n,m}\beta_{i,j}\widetilde{w}_{m,i}^2\widetilde{w}_{n,j}.
		\end{eqnarray*}
		For $E_{j,3,1}$, the estimate is similar to that of $E_{j,2}$.  The difference is that we shall use $\widehat{\eta}_{\tau}$ in stead of $\widetilde{\eta}_{j}$ according to the construction of $\xi_{t,j}$, where $\widehat{\eta}_{\tau}$ is given by \eqref{eq0010}.  Let us sketch the
		estimates of $E_{j,3,1}$.  If $r_{n,j}\leq\frac{\widehat{\eta}_{\tau}}{2}$ for some $n$, then $r_{m,i}\geq\frac{\widehat{\eta}_{\tau}}{2}$ for all $i\in\mathfrak{n}_\tau$ with $(m,i)\not=(n,j)$.  Now, by similar estimates of \eqref{eq3008},
		\begin{eqnarray}\label{eq3014}
			|E_{j,3,1}|&\lesssim&\sum_{i\in\mathfrak{n}_\tau,i\not=j}\sum_{l,m}\widetilde{w}_{m,i}^2\widetilde{w}_{l,j}\notag\\
			&\lesssim&\widehat{\eta}_{\tau}^{1-N}e^{-\sqrt{\lambda_{n_\tau}}\widehat{\eta}_{\tau}}(1+r_{n,j})^{-(\frac{N-1}{2}-\sigma)}e^{-\sqrt{\lambda_{n_\tau}}r_{n,j}}.
		\end{eqnarray}
		If $r_{m,i}\leq\frac{\widehat{\eta}_{\tau}}{2}$ for some $i$ and $m$ then $r_{n,l}\geq\frac{\widehat{\eta}_{\tau}}{2}$ for all $l\in\mathfrak{n}_\tau$ with $(n,l)\not=(m,i)$.  Then by similar estimates of \eqref{eq3010},
		\begin{eqnarray}\label{eq3015}
			|E_{j,3,1}|&\lesssim&\sum_{l\in\mathfrak{n}_\tau,l\not=j}\sum_{p=1}^\vt\widetilde{w}_{p,l}^2\sum_{n=1}^\vt\widetilde{w}_{n,j}\notag\\
			&\lesssim&\widetilde{w}_{m,i}^2\sum_{n=1}^{\vartheta}\widetilde{w}_{n,j}\notag\\
			&\lesssim&\widehat{\eta}_{\tau}^{\frac{1-N}{2}}e^{-\sqrt{\lambda_{n_\tau}}\widehat{\eta}_{\tau}}(1+r_{m,i})^{-(\frac{N-1}{2}-\sigma)}e^{-\sqrt{\lambda_{n_\tau}}r_{m,i}}.
		\end{eqnarray}
		If $r_{t,l}\geq\frac{\widehat{\eta}_{\tau}}{2}$ for all $t$ and $l$, then by similar estimates of \eqref{eq3011},
		\begin{eqnarray*}
			|E_{j,3,1}|\lesssim\widehat{\eta}_{\tau}^{1-N}e^{-\sqrt{\lambda_{n_\tau}}\widehat{\eta}_{\tau}}
			\sum_{t=1}^{\vartheta}(1+r_{t,j})^{-(\frac{N-1}{2}-\sigma)}e^{-\sqrt{\lambda_{n_\tau}}r_{t,j}},
		\end{eqnarray*}
		which, together with \eqref{eq3014} and \eqref{eq3015}, implies that
		\begin{eqnarray}\label{eq3016}
			|E_{j,3,1}|
			&\lesssim&\widehat{\eta}_{\tau}^{\frac{1-N}{2}}e^{-\sqrt{\lambda_{n_\tau}}\widehat{\eta}_{\tau}}
			\bigg(\sum_{i\in\mathfrak{n}_\tau}\sum_{t=1}^{\vartheta}(1+r_{t,i})^{-(\frac{N-1}{2}-\sigma)}e^{-\sqrt{\lambda_{n_\tau}}r_{t,i}}\chi_{\{r_{t,i}\leq\frac{\widehat{\eta}_{\tau}}{2}\}}\notag\\
			&&+\bigg(\sum_{t=1}^{\vartheta}(1+r_{t,j})^{-(\frac{N-1}{2}-\sigma)}e^{-\sqrt{\lambda_{n_\tau}}r_{t,j}}\bigg)\chi_{\cap_{i\in\mathfrak{n}_\tau}\cap_{t=1}^{\vartheta}\{r_{t,i}>\frac{\widehat{\eta}_{\tau}}{2}\}}\bigg).
		\end{eqnarray}
		For $E_{j,3,2}$, the estimate is slightly different from that of $E_{j,2}$ and $E_{j,3,1}$.  Note that $E_{j,3,2}$ exists only when $\lambda_{n_k}>\lambda_{n_1}$, thus, we always assume $\lambda_{n_k}>\lambda_{n_1}$ in estimating $|E_{j,3,2}|$. By \eqref{eq3017}, it is easy to check that $\delta<\frac{1}{2}$.
		If $r_{n,j}\leq\delta\widehat{\eta}$ for some $n$, then $r_{t,l}\geq(1-\delta)\widehat{\eta}$ for all $(t,l)\not=(n,j)$, where $\widehat{\eta}$ is given by \eqref{eq0011}.  Thus, by similar estimates of \eqref{eq3008} and \eqref{eq3009},
		\begin{eqnarray}
			|E_{j,3,2}|&\lesssim&\sum_{i\in\mathfrak{n}_\tau',\tau\not=\tau'}\sum_{l,m}\widetilde{w}_{m,i}^2\widetilde{w}_{l,j}\notag\\
			&\lesssim&\sum_{i\in\mathfrak{n}_\tau',\tau\not=\tau'}\sum_{m=1}^\vt\widetilde{w}_{m,i}^2\widetilde{w}_{n,j}+\sum_{i\in\mathfrak{n}_\tau',\tau\not=\tau'}
			\sum_{l,m;l\not=n}\widetilde{w}_{m,i}^2\widetilde{w}_{l,j}\notag\\
			&\lesssim&\widehat{\eta}^{1-N}e^{-2(1-\delta)\sqrt{\lambda_{n_1}}\widehat{\eta}}(1+r_{t,j})^{-(\frac{N-1}{2}-\sigma)}e^{-\sqrt{\lambda_{n_\tau}}r_{n,j}}.\label{eq3019}
		\end{eqnarray}
		If $r_{m,i}\leq\delta\widehat{\eta}$ for some $i$ and $m$ then $r_{p,l}\geq(1-\delta)\widehat{\eta}$ for all $(p,l)\not=(m,i)$. Moreover, by \eqref{eq3018},
		\begin{eqnarray*}
			\sqrt{\lambda_{n_\tau}}r_{n,j}+\varsigma\sqrt{\lambda_{n_\tau'}}r_{m,i}&\geq&((1-\delta)\sqrt{\lambda_{n_\tau}}+\varsigma\delta\sqrt{\lambda_{n_\tau'}})\widehat{\eta}\\
			&\geq&((1-\delta)\sqrt{\lambda_{n_2}}+\varsigma\delta\sqrt{\lambda_{n_1}})\widehat{\eta}\\
			&\geq&\sqrt{\lambda_{n_1}}\widehat{\eta}
		\end{eqnarray*}
		for all $n,m$ and $i,j$ with $\tau'<\tau$. Furthermore, if we suppose
		\begin{eqnarray*}
			|x-\eta_{t,j}|=\min\limits_{1\le n\le \vt} |x-\eta_{n,j}|
		\end{eqnarray*}
		without loss of generality, then by \eqref{eq3017} and similar estimates of \eqref{eq3003},
		\begin{eqnarray}
			|E_{j,3,2}|&\lesssim&\sum_{l\in\mathfrak{n}_\tau'',\tau''\not=\tau}\sum_{p=1}^\vt\widetilde{w}_{p,l}^2\widetilde{w}_{t,j}\notag\\
			&\lesssim&\sum_{l\in\mathfrak{n}_\tau'',\tau''<\tau}\sum_{p=1}^\vt\widetilde{w}_{p,l}^2\widetilde{w}_{t,j}+\sum_{l\in\mathfrak{n}_\tau'',\tau''>\tau}\sum_{p=1}^\vt\widetilde{w}_{p,l}^2\widetilde{w}_{t,j}\notag\\
			&\lesssim&\widehat{\eta}^{\frac{1-N}{2}}\sum_{l\in\mathfrak{n}_\tau'',\tau''<\tau}\sum_{p=1}^\vt(1+r_{p,l})^{\frac{1-N}{2}}e^{-(2-\varsigma)\sqrt{\lambda_{n_\tau'}}r_{p,l}}
			e^{-(\sqrt{\lambda_{n_\tau}}r_{t,j}+\varsigma\sqrt{\lambda_{n_\tau'}}r_{p,l})}\notag\\
			&&+\widehat{\eta}^{1-N}e^{-2(1-\delta)\sqrt{\lambda_{n_1}}\widehat{\eta}}(1+r_{m,i})^{\frac{1-N}{2}} e^{-\sqrt{\lambda_{n_\tau}}r_{t,j}}\notag\\
			&\lesssim&\widehat{\eta}^{\frac{1-N}{2}}e^{-\sqrt{\lambda_{n_1}}\widehat{\eta}}(1+r_{m,i})^{-(\frac{N-1}{2}-\sigma)}e^{-\sqrt{\lambda_{n_\tau}}r_{m,i}}.\label{eq3020}
		\end{eqnarray}
		If $r_{p,l}\geq\delta\widehat{\eta}$ for all $p$ and $l$ then by the construction of $\{\eta_{t,j}\}$ and \eqref{eq3017},
		\begin{eqnarray}
			2\sqrt{\lambda_{n_\tau'}}r_{m,i}+\sqrt{\lambda_{n_\tau}}r_{n,j}&\geq&\sqrt{\lambda_{n_\tau}}|\eta_{n,j}-\eta_{m,i}|
			+(2\sqrt{\lambda_{n_\tau'}}-\sqrt{\lambda_{n_\tau}})r_{m,i}\label{eq3021}\\
			&\geq&\sqrt{\lambda_{n_1}}|\eta_{n,j}-\eta_{m,i}|+\varsigma\sqrt{\lambda_{n_\tau'}} r_{m,i}.\notag
		\end{eqnarray}
		It follows that
		\begin{eqnarray}
			|E_{j,3,2}|&\lesssim&\sum_{l\in\mathfrak{n}_\tau'',\tau''\not=\tau}\sum_{n,p}\widetilde{w}_{p,l}^2\widetilde{w}_{n,j}\notag\\
			&\lesssim&\widehat{\eta}^{\frac{1-N}{2}}\sum_{m,i}(1+r_{m,i})^{1-N}e^{-\varsigma\sqrt{\lambda_{n_\tau'}}r_{m,i}}\sum_{n=1}^\vt e^{-\sqrt{\lambda_{n_1}}|\eta_{n,j}-\eta_{m,i}|}\notag\\
			&\lesssim&\widehat{\eta}^{\frac{1-N}{2}}e^{-\sqrt{\lambda_{n_1}}\widehat{\eta}}\sum_{m,i}(1+r_{m,i})^{1-N}e^{-\varsigma\sqrt{\lambda_{n_\tau'}}r_{m,i}}\notag\\
			&\lesssim& \widehat{\eta}^{\frac{1-N}{2}}e^{-\sqrt{\lambda_{n_1}}\widehat{\eta}}\sum_{m,i}(1+r_{m,i})^{1-N}e^{-\varepsilon\sqrt{\lambda_{n_\tau}}r_{m,i}},\label{eq3022}
		\end{eqnarray}
where we choose $\epsilon<\frac{\varsigma}{2}$.
		Thus, by \eqref{eq3019}, \eqref{eq3020} and \eqref{eq3022},
		\begin{eqnarray*}
			|E_{j,3,2}|&\lesssim&\widehat{\eta}^{\frac{1-N}{2}}e^{-\sqrt{\lambda_{n_1}}\widehat{\eta}}
			\bigg(\sum_{\tau=1}^{k}\sum_{i=n_{\tau-1}+1}^{n_{\tau}}\sum_{t=1}^{\vartheta}(1+r_{t,i})^{-(\frac{N-1}{2}-\sigma)}e^{-\sqrt{\lambda_{n_{\tau}}}r_{t,i}}
			\chi_{\{r_{t,i}\leq\delta\widehat{\eta}\}}\notag\\
			&&+\bigg(\sum_{\tau=1}^{k}\sum_{i=n_{\tau-1}+1}^{n_{\tau}}\sum_{t=1}^{\vartheta}(1+r_{t,i})^{\frac{1-N}{2}}e^{-\varepsilon\sqrt{\lambda_{n_{\tau}}}r_{t,i}}\bigg)\chi_{\cap_{i=1}^{d}\cap_{t=1}^{\vartheta}
				\{r_{t,i}\geq\delta\widehat{\eta}\}}\bigg),
		\end{eqnarray*}
		which, together with \eqref{eq3016}, implies that \eqref{eqnnnew3023} holds.
	\end{proof}

	\section{Linear theory}
	We introduce the norms
	\begin{eqnarray*}
		\|\varphi\|_{\natural,j}&=&\sum_{\tau=1}^{k}\sum_{i=n_{\tau-1}+1}^{n_{\tau}}\sum_{t=1}^{\vartheta}\sup|\varphi|(1+r_{t,i})^{(\frac{N-1}{2}-\sigma)}e^{\sqrt{\lambda_{n_\tau}}r_{t,i}}
		\chi_{\{r_{t,i}\leq\delta\widehat{\eta}\}}\\
		&&+\sup\frac{|\varphi|}{\sum_{\tau=1}^{k}\sum_{i=n_{\tau-1}+1}^{n_{\tau}}\sum_{t=1}^{\vartheta}(1+r_{t,i})^{\frac{1-N}{2}}e^{-\ve\sqrt{\lambda_{n_{\tau}}}r_{t,i}}}
		\chi_{\cap_{i=1}^d\cap_{t=1}^{\vartheta}\{r_{t,i}>\delta\widehat{\eta}\}}
	\end{eqnarray*}
	and
	\begin{eqnarray*}
		\|\varphi\|_{\sharp,j}&=&\sum_{\tau=1}^{k}\sum_{i=n_{\tau-1}+1}^{n_{\tau}}\sum_{t=1}^{\vartheta}\sup|\varphi|(1+r_{t,i})^{(\frac{N-1}{2}-\sigma-1)}e^{\sqrt{\lambda_{n_\tau}}r_{t,i}}
		\chi_{\{r_{t,i}\leq\delta\widehat{\eta}-1\}}\\
		&&+\sup\frac{|\varphi|}{\sum_{\tau=1}^{k}\sum_{i=n_{\tau-1}+1}^{n_{\tau}}\sum_{t=1}^{\vartheta}(1+r_{t,i})^{\frac{1-N}{2}}e^{-\ve\sqrt{\lambda_{n_{\tau}}}r_{t,i}}}
		\chi_{\cap_{i=1}^d\cap_{t=1}^{\vartheta}\{r_{t,i}>\delta\widehat{\eta}-1\}}.
	\end{eqnarray*}
	We also introduce the Banach spaces
	\begin{eqnarray*}
		\mathbb{X}^{\perp}=\prod_{j=1}^{d}\mathbb{X}_j^{\perp}\quad\text{and}\quad\mathbb{Y}^{\perp}=\prod_{j=1}^{d}\mathbb{Y}_j^{\perp},
	\end{eqnarray*}
	where
	\begin{eqnarray*}
		\mathbb{X}_j^{\perp}=\{u\in H^2(\bbr^N)\mid \|u\|_{\sharp,j}<+\infty,\int_{\bbr^N}\partial_{x_l}\widetilde{w}_{t,j}udx=0, u(x)=u(\Theta_t^{\pm}x)\text{ for all $l$ and $t$}\}
	\end{eqnarray*}
	and
	\begin{eqnarray*}
		\mathbb{Y}_j^{\perp}=\{u\in L^2(\bbr^N)\mid \|u\|_{\natural,j}<+\infty,\int_{\bbr^N}\partial_{x_l}\widetilde{w}_{t,j}udx=0, u(x)=u(\Theta_t^{\pm}x)\text{ for all $l$ and $t$}\}
	\end{eqnarray*}
	with
	\begin{eqnarray}\label{eq4001}
		\Theta_t^{\pm}=\left(\begin{matrix}\cos\frac{2(t-1)\pi}{\vartheta} & \sin\frac{2(t-1)\pi}{\vartheta} & 0\\
			-\sin\frac{2(t-1)\pi}{\vartheta} & \cos\frac{2(t-1)\pi}{\vartheta} & 0\\
			0& 0 & {\pm}I_{(N-2)\times(N-2)}\end{matrix}\right)
	\end{eqnarray}
	For $\mathbf{u}=(u_1,u_2,\cdots,u_d)\in \mathbb{X}^\perp, \mathbf{v}=(v_1,v_2,\cdots,v_d)\in\mathbb{Y}^\perp$, we define
	\begin{equation*}
		\|\mathbf{u}\|_{\sharp}=\sum_{j=1}^d\|u_j\|_{\sharp,j},\quad \|\mathbf{v}\|_{\natural}=\sum_{j=1}^d\|v_j\|_{\natural,j}.
	\end{equation*}
	Now we can state our main result in this section.
	\begin{proposition}\label{prop0001}
		Suppose the assumptions~$(V_1)$--$(V_2)$ hold and $\beta_{i,j}$ is not a eigenvalue of $-\Delta+\lambda_j$ in $L^2(\bbr^N; w_i^2)$, that is,
		\begin{eqnarray*}
			-\Delta v+\lambda_j v=\beta_{i,j}w_i^2v \quad\text{in }\bbr^N
		\end{eqnarray*}
		has no solutions in $H^2(\bbr^N)$, then
		for $\vartheta$ and $\rho\overline{\alpha}=\rho\min\limits_{1\le j\le d}\alpha_j$ sufficiently large,
		\begin{equation}\label{eq4002}
			\left\{\aligned&\mathcal{L}(\mathbf{v})=\mathbf{h},\quad\text{in }\bbr^N,\\
			&\mathbf{v}\in\mathbb{X}^{\perp},\endaligned\right.
		\end{equation}
		is unique solvable for every $\mathbf{h}=(h_1,h_2,\cdots,h_d)\in \mathbb{Y}^\perp$,
		where
		\begin{eqnarray*}
			\mathcal{L}(\mathbf{v})=(\mathcal{L}_1(\mathbf{v}),\mathcal{L}_2(\mathbf{v}),\cdots\mathcal{L}_d(\mathbf{v}))
		\end{eqnarray*}
		with
		\begin{eqnarray*}
			\mathcal{L}_j(\mathbf{v})=-\Delta v_j+V_j(x)v_j-3\mu_jW_j^2v_j-\sum_{i=1;i\not=j}^{d}\beta_{i,j}(W_i^2v_j+2W_iW_jv_i).
		\end{eqnarray*}
		Moreover, $\|\mathbf{v}\|_\sharp\lesssim\|\mathbf{h}\|_{\natural}$.
	\end{proposition}
	\begin{proof}
		We start the proof by showing that there exists $R_0>0$ large enough which is independent of $\vartheta$ and $\rho\overline{\alpha}$ sufficiently large such that
		\begin{eqnarray}\label{eq4004}
			|v_{j}(x)|\lesssim(\|\mathbf{h}\|_{\natural}+\sup_{t,i,l}\|v_{l}\|_{L^{\infty}(\partial B_{R_0}(\eta_{t,i}))})\bigg(
			\sum_{\tau,t,i}(1+r_{t,i})^{\frac{1-N}{2}}e^{-\ve\sqrt{\lambda_{n_{\tau}}}r_{t,i}}\bigg)
		\end{eqnarray}
		for $x\in\bbr^N\backslash\bigcup_{t,i}B_{R_0}(\eta_{t,i})$.  Indeed, by \eqref{eqnnnew0008}, it is well known that $\sigma(-\Delta+\lambda_j)=\{\sigma_l\}$ in $L^2(\bbr^N; w_i^2)$ with $\sigma_l\to+\infty$ as $l\to\infty$, where $\sigma(-\Delta+\lambda_j)$ is the spectrum of $-\Delta+\lambda_j$ in $L^2(\bbr^N; w_i^2)$.  Thus, $\sharp\{l\mid\sigma_l<\beta_{i,j}\}$ is finite.  Note that it is also well known that the Morse index of $w_j$ is one for all $j$.  Thus, by the construction of $\{\eta_{t,j}\}$, for $\vartheta$ and $\rho\overline{\alpha}$ sufficiently large, it is standard to use the Riesz representation theorem to show that $\mathcal{L}(\mathbf{v})=\mathbf{h}$ is uniquely solvable in $\prod_{j=1}^d(\mathcal{H}_j)^{\perp}$ for every $\mathbf{h}\in \mathbb{Y}^\perp$, where
		\begin{eqnarray*}
			\mathcal{H}_j=\{u\in H^1(\bbr^N)\mid \int_{\bbr^N}\partial_{x_l}\widetilde{w}_{t,j}udx=0, u(x)=u(\Theta_t^{\pm}x)\text{ for all $l$ and $t$}\}.
		\end{eqnarray*}
		Now, by elliptic estimates, we know that $v_j\in L^\infty(\bbr^N)$ for all $j$.  By the assumptions~$(V_1)$ and $(V_2)$, we can choose $c_j<\lambda_j$ such that $V_j(x)\geq 2c_j>0$ in $\bbr^N$ for all $j$.  By a direct computation,
		\begin{eqnarray}\label{eqnnnew0002}
			-\Delta\varphi_{\omega}+c_j\varphi_{\omega}\gtrsim\varphi_{\omega}\quad\text{in }\bbr^N\backslash B_{R_0}(\eta_{t,i})
		\end{eqnarray}
		for $R_0>0$ sufficiently large, where
		\begin{eqnarray*}
			\varphi_{\omega}=\sum_{\tau=1}^{k}\sum_{i=n_{\tau-1}+1}^{n_{\tau}}\sum_{t=1}^\vartheta((1+r_{t,i})^{\frac{1-N}{2}}e^{-\ve\sqrt{\lambda_{n_{\tau}}}r_{t,i}}+\omega e^{\ve\sqrt{\lambda_{n_{\tau}}}r_{t,i}}).
		\end{eqnarray*}
		Note that by \eqref{eqnnnew0008}, for $\vartheta>0$ sufficiently large,
		\begin{align}
			-\Delta |v_j|+c_j |v_j|\lesssim&(e^{-\sqrt{\lambda_1}R_0}\sup_{l}\|v_{l}\|_{L^{\infty}(\bbr^N\backslash\bigcup_{t,i}B_{R_0}(\eta_{t,i}))}+\|h_j\|_{\natural,j})\notag\\
			&\times\bigg(
			\sum_{\tau,t,i}(1+r_{t,i})^{\frac{1-N}{2}}e^{-\ve\sqrt{\lambda_{n_{\tau}}}r_{t,i}}\bigg)\label{eqnnnew0006}
		\end{align}
		in $\bbr^N\backslash\bigcup_{t,i}B_{R_0}(\eta_{t,i})$ for $R_0>0$ sufficiently large.  Thus, by the maximum principle,
		\begin{eqnarray}\label{eq4005}
			|v_j(x)|\lesssim\bigg(e^{-\sqrt{\lambda_1}R_0}\sup_{l}\|v_{l}\|_{L^{\infty}(\bbr^N\backslash\bigcup_{t,i}B_{R_0}(\eta_{t,i}))}+\|h_j\|_{\natural,j}+\sup_{t,i}\|v_{j}\|_{L^{\infty}(\partial B_{R_0}(\eta_{t,i}))}\bigg)\varphi_{\omega},
		\end{eqnarray}
		in $\bbr^N\backslash\bigcup_{t,i}B_{R_0}(\eta_{t,i})$ for every $j$, where we take $\ve>0$ sufficiently small if necessary.  In particular, by letting $\omega\to0$, we know that
		\begin{eqnarray*}
			\sup_{l}\|v_{l}\|_{L^{\infty}(\bbr^N\backslash\bigcup_{t,i}B_{R_0}(\eta_{t,i}))}
			&\lesssim& e^{-\sqrt{\lambda_1}R_0}\sup_{l}\|v_{l}\|_{L^{\infty}(\bbr^N\backslash\bigcup_{t,i}B_{R_0}(\eta_{t,i}))}\\
			&&+\|\mathbf{h}\|_{\natural}+\sup_{t,i,l}\|v_{l}\|_{L^{\infty}(\partial B_{R_0}(\eta_{t,i}))},
		\end{eqnarray*}
		which implies
		\begin{eqnarray*}
			\sup_{l}\|v_{l}\|_{L^{\infty}(\bbr^N\backslash\bigcup_{t,i}B_{R_0}(\eta_{t,i}))}\lesssim\sup_{t,i,l}\|v_{l}\|_{L^{\infty}(\partial B_{R_0}(\eta_{t,i}))}+\|\mathbf{h}\|_{\natural}.
		\end{eqnarray*}
		It follows from letting $\omega\to0$ in \eqref{eq4005} once more that \eqref{eq4004} holds.  Recall that by \eqref{eq3017} and \eqref{eq3018}, $\delta<\frac12$. Thus, we have $|x|\sim\rho$ in every $B_{\sigma\widehat{\eta}}(\eta_{t,i})$.  It follows from the assumption~$(V_2)$ and \eqref{eq4004} that
		\begin{align}
			&-\Delta |v_j|+(\lambda_j-C\rho^{-\nu_j})|v_j|\notag\\
			&\lesssim\Big(|v_j|\sum_{l=1}^d W_l^2+W_j\sum_{l\not=j}W_l|v_l|+|h_j|\Big)\notag\\
			&\lesssim\Big(e^{-\sqrt{\lambda_1}R_0}\sup_{t,i,l}\|v_{l}\|_{L^{\infty}(\partial B_{R_0}(\eta_{t,i}))}+\|\mathbf{h}\|_{\natural}\Big)
			\bigg(\sum_{\tau=1}^{k}\sum_{l=n_{\tau-1}+1}^{n_{\tau}}\sum_{p=1}^\vartheta\frac{e^{-\sqrt{\lambda_{n_{\tau}}}r_{p,l}}}{(1+r_{p,l})^{(\frac{N-1}{2}-\sigma)}}\bigg)\notag\\
			&\lesssim(e^{-\sqrt{\lambda_1}R_0}\sup_{t,i,l}\|v_{l}\|_{L^{\infty}(\partial B_{R_0}(\eta_{t,i}))}+\|\mathbf{h}\|_{\natural})\frac{e^{-\sqrt{\lambda_{n_{\tau}}}r_{t,i}}}{(1+r_{t,i})^{(\frac{N-1}{2}-\sigma)}}\label{eqnnnew0007}
		\end{align}
		in every $B_{\delta\widehat{\eta}}(\eta_{t,i})\backslash B_{R_0}(\eta_{t,i})$ for $R_0>0$ sufficiently large.  Here, we also apply similar arguments of \eqref{eq3003} for the last inequality.
		For every $\tau$ and every $j\in \mathbf{n}_\tau$, we define
		\begin{eqnarray*}
			\phi_{j,i,t}=r_{t,i}^{-(\frac{N-1}{2}-\sigma)+1}e^{{-(\sqrt{\lambda_{n_\tau}}-C\rho^{-\nu_j})r_{t,i}}}
		\end{eqnarray*}
		for all $i=1,2,\cdots,d$ and $t=1,2,\cdots,\vartheta$.  Then by direct computations,
		\begin{eqnarray}\label{eqnnnew0003}
			-\Delta\phi_{j,i,t}+(\lambda_j-C\rho^{-\nu_j})\phi_{j,i,t}\gtrsim\frac{\phi_{j,i,t}}{r_{t,i}}>0\quad\text{in }\bbr^N\backslash B_{R_0}(\eta_{t,i})
		\end{eqnarray}
		for $R_0>0$ sufficiently large.  Let $\psi_1(t)$ and $\psi_2(t)$ be unique solutions of the following ODEs, respectively,
		\begin{eqnarray*}
			\left\{\aligned&-\psi_{1}''-2\ve\psi_{1}'+\ve^2\psi_{1}=1,\\
			&\psi_{1}(0)=1,\quad \psi_{1}'(0)=0
			\endaligned\right.
		\end{eqnarray*}
		and
		\begin{eqnarray*}
			\left\{\aligned&-\psi_{2}''-2\ve\psi_{2}'+\ve^2\psi_{2}=1,\\
			&\psi_{2}(1)=0,\quad \psi_{2}'(1)=0,
			\endaligned\right.
		\end{eqnarray*}
		where $\ve>0$ is sufficiently small.
		Then
		\begin{eqnarray*}
			\psi_1(r)=\ve^{-2}-\frac{\ve^{-2}-1}{2\sqrt{2}}((\sqrt{2}-1)e^{-(\sqrt{2}+1)\ve r}+(\sqrt{2}+1)e^{(\sqrt{2}-1)\ve r})
		\end{eqnarray*}
		and
		\begin{eqnarray*}
			\psi_2(r)=\ve^{-2}-\frac{\ve^{-2}}{2\sqrt{2}}(\frac{\sqrt{2}-1}{e^{-(\sqrt{2}+1)\ve}}e^{-(\sqrt{2}+1)\ve r}+\frac{\sqrt{2}+1}{e^{(\sqrt{2}-1)\ve}}e^{(\sqrt{2}-1)\ve r}).
		\end{eqnarray*}
		It is easy to see that $\psi_1'(r)<0$ and $\psi_2'(r)<0$ for $r\in(0, 1)$ and there exists $r_0\in(0, 1)$ such that $\psi_1'(r_0)=\psi_{2}'(r_0)$.  Now, let
		\begin{eqnarray*}
			\psi(t)=\left\{\aligned&\psi_1(0)-\psi_{1}(r_0)+\psi_2(r_0), \quad t\leq0,\\
			&\psi_1(t)-\psi_1(r_0)+\psi_2(r_0), \quad 0<t\leq r_0,\\
			&\psi_2(t),\quad r_0<t<1,\\
			&0, \quad t\geq1,\endaligned\right.
		\end{eqnarray*}
		then $\psi(t)\in L^\infty(\bbr)\cap C^1(\bbr)\cap W^{2,\infty}(\bbr)$ is a cut-off function.
		Let
		\begin{eqnarray*}
			\widehat{\phi}_{\omega}=\sum_{i,t}\phi_{j,i,t}\psi(r_{t,i}-\delta\widehat{\eta}+1)+(\psi_1(0)-\psi_{1}(r_0)+\psi_2(r_0)-\sum_{i,t}\psi(r_{t,i}-\delta\widehat{\eta}+1))\varphi_{\omega}.
		\end{eqnarray*}
		Clearly, $supp(\psi(r_{t,i}-\delta\widehat{\eta}+1))\cap supp(\psi(r_{s,j}-\delta\widehat{\eta}+1))=\emptyset$ for all $(t,i)\not=(s,j)$.
		Since $\phi_{j,i,t}=o(1)\varphi_{\omega}$ in $supp(\psi(r_{t,i}-\delta\widehat{\eta}+1))$ for every $i$ and $t$, by \eqref{eqnnnew0002} and \eqref{eqnnnew0003},
		\begin{eqnarray}\label{eqnnnew0005}
			\left\{\aligned&-\Delta\widehat{\phi}_{\omega}+c_j\widehat{\phi}_{\omega}\gtrsim\varphi_{\omega},\quad \text{a.e. }r_{t,i}\geq\delta\widehat{\eta}-1\text{ for all }(t,i),\\
			&-\Delta\widehat{\phi}_{\omega}+(\lambda_j-C\rho^{-\nu_j})\widehat{\phi}_{\omega}\gtrsim\frac{\phi_{j,i,t}}{r_{t,i}},\quad \text{a.e. }R_0\leq r_{t,i}\leq\delta\widehat{\eta}-1\text{ for some }(t,i).
			\endaligned\right.
		\end{eqnarray}
		Thus, by \eqref{eq4004}, \eqref{eqnnnew0006}, \eqref{eqnnnew0007} and \eqref{eqnnnew0005}, we can apply the maximum principle to show that
		\begin{eqnarray*}
			|v_{j}(x)|&\lesssim&\Big(\sup_{t,i,l}\|v_{l}\|_{L^{\infty}(\partial B_{R_0}(\eta_{t,i}))}+\|\mathbf{h}\|_{\natural}\Big)\widehat{\phi}_{\omega}
		\end{eqnarray*}
		in $\bbr^N\backslash\bigcup_{t,i}B_{R_0}(\eta_{t,i})$, which, together with letting $\omega\to0$, implies that
		\begin{eqnarray}
			|v_{j}(x)|
			&\lesssim&\Big(\sup_{t,i,l}\|v_{l}\|_{L^{\infty}(\partial B_{R_0}(\eta_{t,i}))}+\|\mathbf{h}\|_{\natural}\Big)\bigg(\sum_{\tau,t,i}(1+r_{t,i})^{-(\frac{N-1}{2}-\sigma-1)}e^{-\sqrt{\lambda_{n_\tau}}r_{t,i}}
			\chi_{\{R_0\leq r_{t,i}\leq\delta\widehat{\eta}-1\}}\notag\\
			&&+\sum_{\tau=1}^{k}\sum_{i=n_{\tau-1}+1}^{n_{\tau}}\sum_{t=1}^{\vartheta}(1+r_{t,i})^{{\frac{1-N}{2}}}e^{-\ve\sqrt{\lambda_{n_{\tau}}}r_{t,i}}\bigg)\chi_{\cap_{i=1}^{d}\cap_{t=1}^{\vartheta}
				\{r_{t,i}\geq\delta\widehat{\eta}-1\}}\bigg)\label{eq4007}
		\end{eqnarray}
		in $\bbr^N\backslash\bigcup_{t,i}B_{R_0}(\eta_{t,i})$.
		Now, we can use the blow-up arguments to establish the a-priori estimate $\|\mathbf{v}\|_\sharp\lesssim\|\mathbf{h}\|_{\natural}$.
		We assume the contrary that there exists $\vartheta_{n}\to+\infty$ and $\rho_n\overline{\alpha}_n\to+\infty$ as $n\to\infty$, and $\mathbf{h}_n\in\mathbb{Y}^{\perp}$ and $\mathbf{v}_n\in\mathbb{X}^{\perp}$ such that
		\begin{equation*}\label{eq0018}
			\mathcal{L}_n(\mathbf{v}_n)=\mathbf{h}_n \quad\text{in }\bbr^N
		\end{equation*}
		with $\|h_{j,n}\|_{\natural,j}\to0$ for all $j$  as $n\to\infty$ and there exists $j_0$ such that
		\begin{eqnarray*}
			\|v_{j_0,n}\|_{\sharp,j_0}=\max\limits_{1\le j\le d} \|v_{j,n}\|_{\sharp,j}=1.
		\end{eqnarray*}
		Since $|\eta_{t,j}^n|=\rho_{j,n}\to+\infty$ for all $t,j$, we have $|x|\geq\frac{\widehat{\eta}^n}{2}\to+\infty$ and $r_{s,l}^n\geq\frac{\widehat{\eta}^n}{2}\to+\infty$ for all $(s,l)\not=(m,i)$ if there exists $1\leq m\leq \vartheta_n$ and $\mathfrak{n}_{\tau-1}<i\leq\mathfrak{n}_\tau$ for some $\tau$ such that $r_{m,i}\leq\frac{\widehat{\eta}^n}{2}$.  Clearly, $\{v_{j_0,n}(\cdot+\eta_{1,i}^n)\}$ is bounded in $L^\infty_{loc}(\bbr^N)$ for all $i$.  Thus, by elliptic regularity estimates and Ascoli-Arzela's theorem, it is standard to show that $\widetilde{v}_{j_0,i,n}\to{v}_{j_0,i,\infty}$ uniformly on every compact set of $\bbr^N$, where
		\begin{eqnarray*}
			\widetilde{v}_{j_0,i,n}(x)=v_{j_0,n}(x+\eta_{1,i}^n).
		\end{eqnarray*}
		Moreover, if $i=j_0$, then ${v}_{j_0,j_0,\infty}$ is a bounded (weak and then strong) solution of the equation
		\begin{eqnarray}\label{eq4010}
			-\Delta u+\lambda_{j_0}u-3\mu_{j_0}w_{j_0}^2u=0 \quad\text{in }\bbr^N,
		\end{eqnarray}
		while if $i\not=j_0$, then ${v}_{j_0,i,\infty}$ is a bounded (weak and then strong) solution of the equation
		\begin{eqnarray}\label{eq4011}
			-\Delta u+\lambda_{j_0}u-\beta_{i,j}w_i^2u=0 \quad\text{in }\bbr^N.
		\end{eqnarray}
		For the sake of simplicity, we denote ${v}_{j_0,j_0,\infty}$ by ${v}_{j_0,\infty}$.
		Since $w_{j_0}$ is non-degenerate in $H^2(\bbr^N)$, by \eqref{eq4010}, ${v}_{j_0,\infty}=\sum\limits_{l=1}^Nc_l\partial_{x_l}w_{j_0}$.
		Note that by passing to the limit in the orthogonal conditions of $v_{j_0,n}$, we have
		\begin{eqnarray*}
			\int_{\bbr^N}(\partial_{x_l}w_{j_0}){v}_{j_0,\infty}dx=0\quad\text{for all }l=1,2,\cdots,N.
		\end{eqnarray*}
		Thus, we must have ${v}_{j_0,\infty}=0$, which implies that $v_{j_0,n}(\cdot+\eta_{1,j_0}^n)\to0$ uniformly on every compact set of $\bbr^N$.  On the other hand, from the assumption on $\beta_{i,j}$ by \eqref{eq4011}, we also must have ${v}_{j_0,i,\infty}=0$ for all $i\not=j_0$, which implies that $v_{j_0,n}(\cdot+\eta_{1,i}^n)\to0$ uniformly in every compact set of $\bbr^N$ and all $i\not=j_0$. In both cases,
		since $v_{j_0,n}$ is invariant under the action of the group $\Theta_t$ for all $t$, where $\Theta_t$ is given by \eqref{eq4001}, we have $v_{j_0,n}(\cdot+\eta_{t,i}^n)\to0$ uniformly on every compact set of $\bbr^N$ for all $t,i$, which, together with \eqref{eq4007}, implies that $\|v_{j_0,n}\|_{\sharp,j_0}=o_n(1)$.  It is impossible since $\|v_{j_0,n}\|_{\sharp,j_0}=1$.  Thus, we have proved the a-priori estimate $\|\mathbf{v}\|_\sharp\lesssim\|\mathbf{h}\|_{\natural}$.
		Finally, the solvability and uniqueness of \eqref{eq4002}  follows from the a-priori estimate $\|\mathbf{v}\|_\sharp\lesssim\|\mathbf{h}\|_{\natural}$ and the Fredholm alternative.
	\end{proof}
	
	\section{The ansatz and the nonlinear problem}
	Let us consider the linear equation:
	\begin{equation}\label{eqnew9988}
		\left\{\aligned&\mathcal{L}(\mathbf{v})=\mathbf{E}_{3}^{\perp},\quad\text{in }\bbr^N,\\
		&\mathbf{v}\in\mathbb{X}^{\perp},\endaligned\right.
	\end{equation}
	where $\mathbf{E}_{3}^{\perp}=(E_{1,3}^{\perp},E_{2,3}^{\perp},\cdots,E_{d,3}^{\perp})$, $E_{j,3}^{\perp}=E_{j,3}-\sum_{l=1}^{N}\sum_{t=1}^{\vartheta}\gamma_{l,t,j}\partial_{x_l}\widetilde{w}_{t,j}$ with $\gamma_{l,t,j}\in\bbr$.
	\begin{lemma}\label{lem0002}
		Suppose the assumptions~$(V_1)$--$(V_2)$ hold and $\beta_{i,j}$ is not a eigenvalue of $-\Delta+\lambda_j$ in $L^2(\bbr^N; w_i^2)$, that is,
		\begin{eqnarray*}
			-\Delta v+\lambda_j v=\beta_{i,j}w_i^2v \quad\text{in }\bbr^N
		\end{eqnarray*}
		has no solutions in $H^2(\bbr^N)$, then
		for $\vartheta$ and $\rho\overline{\alpha}=\rho\min\limits_{1\le j\le d}\alpha_j$ sufficiently large,
		\begin{eqnarray}\label{eqnn0007}
			E_{j,3}^{\perp}=E_{j,3}-\gamma_{\theta_{j}}\partial_{\theta_{j}}W_j-\gamma_{\rho_j}\partial_{\rho_j}W_j
		\end{eqnarray}
		where
		\begin{eqnarray}\label{eqn2038}
			\gamma_{\theta_{j}}=\frac{\|\partial_{\rho_j}W_j\|_{L^2(\bbr^N)}^2\int_{\bbr^N}E_{j,3}\partial_{\theta_{j}}W_jdx-\int_{\bbr^N}\partial_{\theta_{j}}W_j\partial_{\rho_j}W_jdx
				\int_{\bbr^N}E_{j,3}\partial_{\rho_j}W_jdx}{\|\partial_{\theta_{j}}W_j\|_{L^2(\bbr^N)}^2\|\partial_{\rho_j}W_j\|_{L^2(\bbr^N)}^2-
				(\int_{\bbr^N}\partial_{\theta_{j}}W_j\partial_{\rho_j}W_jdx)^2}
		\end{eqnarray}
		and
		\begin{eqnarray}\label{eqnnnew2038}
			\gamma_{\rho_j}=\frac{\|\partial_{\theta_{j}}W_j\|_{L^2(\bbr^N)}^2\int_{\bbr^N}E_{j,3}\partial_{\rho_j}W_jdx-\int_{\bbr^N}\partial_{\theta_{j}}W_j\partial_{\rho_j}W_jdx
				\int_{\bbr^N}E_{j,3}\partial_{\theta_{j}}W_jdx}{\|\partial_{\theta_{j}}W_j\|_{L^2(\bbr^N)}^2\|\partial_{\rho_j}W_j\|_{L^2(\bbr^N)}^2-
				(\int_{\bbr^N}\partial_{\theta_{j}}W_j\partial_{\rho_j}W_jdx)^2}
		\end{eqnarray}
		with
		\begin{eqnarray}\label{eqnn0005}
			\|\partial_{\rho_j}W_j\|_{L^2(\bbr^N)}^2\sim\vartheta\quad\text{and}\quad\|\partial_{\theta_{j}}W_j\|_{L^2(\bbr^N)}^2\sim\rho_j^2\vartheta.
		\end{eqnarray}
		Moreover, \eqref{eqnew9988} has a unique solution $\mathbf{Q}_*$ satisfying
		\begin{eqnarray}\label{eqn2040}
			\|\mathbf{Q}_*\|_\sharp\lesssim\widehat{\eta}^{\frac{1-N}{2}}e^{-\sqrt{\lambda_{n_1}}\widehat{\eta}}.
		\end{eqnarray}
	\end{lemma}
	\begin{proof}
		Since $E_{j,3}=\sum_{i\not=j}\beta_{i,j}W_i^2W_j$ is invariant under the action of the group $\Theta_t$, we have
		\begin{eqnarray*}
			\int_{\bbr^N}E_{j,3}\partial_{x_l}\widetilde{w}_{t,j}dx=\frac{1}{\vartheta}\int_{\bbr^N}E_{j,3}\partial_{x_l}W_jdx,
		\end{eqnarray*}
		where $\Theta_t$ is given by \eqref{eq4001}.
		Moreover, since $\widetilde{w}_{t,j}=w(x-\rho_j\xi_{t,j})$, by \eqref{eq0004},
		\begin{eqnarray*}
			(\partial_{\rho_j}\widetilde{w}_{t,j}, \partial_{\theta_{j}}\widetilde{w}_{t,j})=\mathbb{M}\nabla_{x'}\widetilde{w}_{t,j},
		\end{eqnarray*}
		where $x'=(x_1,x_2)$ and
		\begin{eqnarray*}
			\mathbb{M}=\left(\begin{matrix}-\cos\bigg(\theta_{j}+\frac{2(t-1)\pi}{\vartheta}\bigg) & -\sin\bigg(\theta_{j}+\frac{2(t-1)\pi}{\vartheta}\bigg)\\
				\rho_j\sin\bigg(\theta_{j}+\frac{2(t-1)\pi}{\vartheta}\bigg) & -\rho_j\cos\bigg(\theta_{j}+\frac{2(t-1)\pi}{\vartheta}\bigg)\end{matrix}\right).
		\end{eqnarray*}
		Clearly, $\mathbb{M}$ is invertible and $u$ is even for $x_3$ in the case of $N=3$ if $u\in\mathbb{Y}^{\perp}$.  Thus, $\int_{\bbr^N}\partial_{x_l}\widetilde{w}_{t,j}udx=0$ for all $t$ and $l$, the orthogonal conditions in $\mathbb{X}_j^{\perp}$ and $\mathbb{Y}_j^{\perp}$, is equivalent to
		\begin{eqnarray*}
			\int_{\bbr^N}\partial_{\rho_j}\widetilde{w}_{t,j}udx=0\quad\text{and}\quad\int_{\bbr^N}\partial_{\theta_{j}}\widetilde{w}_{t,j}udx=0\quad\text{for all }t.
		\end{eqnarray*}
		It follows from $|\xi_{t,j}|=1$ and the invariance of $u$ under the action of the group $\Theta_t$ that the orthogonal conditions in $\mathbb{X}_j^{\perp}$ and $\mathbb{Y}_j^{\perp}$ are
		\begin{eqnarray*}
			\int_{\bbr^N}\partial_{\rho_j}W_judx=0\quad\text{and}\quad\int_{\bbr^N}\partial_{\theta_{j}}W_judx=0.
		\end{eqnarray*}
		By the construction of $\{\xi_{t,j}\}$,
		\begin{eqnarray}\label{eqnn0006}
			\|\partial_{\rho_j}W_j\|_{L^2(\bbr^N)}^2=\vartheta\|\partial_{\rho_j}\widetilde{w}_{1,j}\|_{L^2(\bbr^N)}^2
			+2\sum_{s,t;t\not=s}\int_{\bbr^N}\partial_{\rho_j}\widetilde{w}_{t,j}\partial_{\rho_j}\widetilde{w}_{s,j}dx.
		\end{eqnarray}
		Since $|\partial_{\rho_j}\widetilde{w}_{t,j}|\lesssim\widetilde{w}_{t,j}$ for every $\tau$ and every $j\in \mathbf{n}_\tau$, by \cite[Lemma~A.2]{PV22}, \eqref{eq0009}, $\rho\sim\vartheta\log\vartheta$ and similar estimates of \eqref{eq3003},
		\begin{eqnarray}
			|\sum_{s,t;t\not=s}\int_{\bbr^N}\partial_{\rho_j}\widetilde{w}_{t,j}\partial_{\rho_j}\widetilde{w}_{s,j}dx|&\lesssim&
			\sum_{s,t;t\not=s}\int_{\bbr^N}\widetilde{w}_{t,j}\widetilde{w}_{s,j}dx\notag\\
			&\sim&\sum_{t\not=s}|\eta_{t,j}-\eta_{s,j}|^{\frac{3-N}{2}}e^{-\sqrt{\lambda_{n_\tau}}|\eta_{t,j}-\eta_{s,j}|}\notag\\
			&\lesssim&\sum_{t=1}^{\vartheta}\sum_{s\not=t}|\eta_{t,j}-\eta_{s,j}|^{\frac{3-N}{2}}e^{-\sqrt{\lambda_{n_\tau}}|\eta_{t,j}-\eta_{s,j}|}\notag\\
			&\lesssim&\vartheta\widetilde{\eta}_{j}^{\frac{3-N}{2}}e^{-\sqrt{\lambda_{n_\tau}}\widetilde{\eta}_{j}}\notag\\
			&\lesssim&\vartheta^{1-\sigma'},\label{eqnn0011}
		\end{eqnarray}
		where $\sigma'>0$ is a constant.  Thus, by \eqref{eqnn0006}, $\|\partial_{\rho_j}W_j\|_{L^2(\bbr^N)}^2\sim\vartheta$.    Note that by
		\begin{eqnarray*}
			|\partial_{\rho_j}\widetilde{w}_{t,j}|^2+\rho_j^{-2}|\partial_{\theta_{j}}\widetilde{w}_{t,j}|^2=|\nabla_{x'} \widetilde{w}_{t,j}|^2,
		\end{eqnarray*}
		we also have $|\partial_{\theta_{j}}\widetilde{w}_{t,j}|\lesssim\rho_j\widetilde{w}_{t,j}$, thus, the estimates of $\|\partial_{\theta_{j}}W_j\|_{L^2(\bbr^N)}^2\sim\rho_j^2\vartheta$ is similar and we omit it.  On the other hand, by similar estimates of \eqref{eqnn0011} and $\int_{\bbr^N}\partial_{\theta_{j}}\widetilde{w}_{j,t}\partial_{\rho_{j}}\widetilde{w}_{j,t}dx=0$ for all $t$, we have
		\begin{eqnarray}\label{eqnn9996}
			|\int_{\bbr^N}\partial_{\theta_{j}}W_j\partial_{\rho_{j}}W_jdx|\lesssim|\sum_{t,s;t\not=s}\int_{\bbr^N}\partial_{\rho_j}\widetilde{w}_{t,j}\partial_{\theta_{j}}\widetilde{w}_{s,j}dx|
			\lesssim\rho_j\vartheta\widehat{\eta}_{j}^{\frac{3-N}{2}}e^{-\sqrt{\lambda_{n_\tau}}\widehat{\eta}_{j}},
		\end{eqnarray}
		therefore, by \eqref{eqnn0005} and \eqref{eqnn9996}, \eqref{eqnn0007} and \eqref{eqn2038} hold.  Clearly, by \eqref{eqnn0007} and \eqref{eqnn0005},
		\begin{eqnarray*}
			|E_{j,3}^{\perp}|\lesssim|E_{j,3}|+(|\gamma_{\theta_{j}}|\rho_j+|\gamma_{\rho_j}|)\sum_{t=1}^{\vartheta}\widetilde{w}_{t,j}.
		\end{eqnarray*}
		Note that by \eqref{eqnnnew3023} \eqref{eqn2038} and and \eqref{eqnnnew2038},
		\begin{eqnarray}\label{eqnn0008}
			|\gamma_{\theta_{j}}|\lesssim\rho^{-1}\vartheta^{-\frac12}\|E_{j,3}\|_{L^2(\bbr^N)}\lesssim\rho^{-1}\|E_{j,3}\|_{\natural,j}
			\lesssim\rho^{-1}\widehat{\eta}^{\frac{1-N}{2}}e^{-\sqrt{\lambda_{n_1}}\widehat{\eta}}
		\end{eqnarray}
		and
		\begin{eqnarray}\label{eqnn1008}
			|\gamma_{\rho_j}|\lesssim\vartheta^{-\frac12}\|E_{j,3}\|_{L^2(\bbr^N)}\lesssim\|E_{j,3}\|_{\natural,j}
			\lesssim\widehat{\eta}^{\frac{1-N}{2}}e^{-\sqrt{\lambda_{n_1}}\widehat{\eta}},
		\end{eqnarray}
		thus, $\|\mathbf{E}_{3}^{\perp}\|_{\natural}\lesssim\|\mathbf{E}_{3}\|_{\natural}\lesssim\widehat{\eta}^{\frac{1-N}{2}}e^{-\sqrt{\lambda_{n_1}}\widehat{\eta}}$.
		By Proposition~\ref{prop0001}, \eqref{eqnew9988} has a unique solution $\mathbf{Q}_*$ satisfying
		\begin{eqnarray}\label{eqn2040}
			\|\mathbf{Q}_*\|_{\sharp}\lesssim\widehat{\eta}^{\frac{1-N}{2}}e^{-\sqrt{\lambda_{n_1}}\widehat{\eta}}.
		\end{eqnarray}
		It completes the proof.
	\end{proof}
	
	Let $E_{j}^*=E_{j,1}+E_{j,2}$, then by \eqref{eqnnnew3001} and\eqref{eqnnnew3012},
	\begin{eqnarray}
		|E_{j}^*|&\lesssim&\bigg(\rho^{-\nu_j}+\widetilde{\eta}_{j}^{\frac{1-N}{2}}e^{-\sqrt{\lambda_{n_\tau}}\widetilde{\eta}_{j}}\bigg)
		\bigg(\sum_{t=1}^{\vartheta}(1+r_{t,j})^{-(\frac{N-1}{2}-\sigma)}e^{-\sqrt{\lambda_{n_\tau}}r_{t,j}}\chi_{\{r_{t,j}
			\leq\frac{\widetilde{\eta}_{j}}{2}\}}\notag\\
		&&+\bigg(\sum_{t=1}^{\vartheta}(1+r_{t,j})^{-(\frac{N-1}{2}-\sigma)}e^{-\sqrt{\lambda_{n_\tau}}r_{t,j}}\bigg)
		\chi_{\cap_{t=1}^{\vartheta}\{r_{t,j}\geq\frac{\widetilde{\eta}_{j}}{2}\}}\bigg).\label{eq3013}
	\end{eqnarray}
	We use $\mathbf{U}_*=(W_1+Q_{1,*},W_2+Q_{2,*},\cdots,W_d+Q_{d,*})$ as the final ansatz and write
	\begin{eqnarray*}
		U_j=W_j+Q_{j,*}+v_{j,**},
	\end{eqnarray*}
	where $v_{j,**}$ are the correction terms.  By \eqref{eq0012} and \eqref{eqnew9988}, $\mathbf{v}_{**}=(v_{1,**},v_{2,**},\cdots,v_{d,**})$ satisfies the nonlinear problem
	\begin{equation}\label{eq9003}
		\left\{\aligned&\mathcal{L}(\mathbf{v}_{**})=\mathbf{E}_{**}^{\perp}+\mathbf{N}_{**}^{\perp},\quad\text{in }\bbr^N,\\
		&\mathbf{v}\in\mathbb{X}^{\perp},\endaligned\right.
	\end{equation}
	where $\mathbf{E}_{**}=(E_{1}^{**},E_{2}^{**},\cdots,E_{d}^{**})$ with
	\begin{eqnarray}
		E_{j}^{**}&=&E_{j}^{*}+\sum_{i=1;i\not=j}^{d}\beta_{i,j}(2W_iQ_{i,*}Q_{j,*}+W_jQ_{i,*}^2+Q_{i,*}^2Q_{j,*})+3\mu_j W_jQ_{j,*}^2+Q_{j,*}^3\notag\\
		&&+\gamma_{\theta_{j}}\partial_{\theta_{j}}W_j+\gamma_{\rho_j}\partial_{\rho_j}W_j\notag\\
		&=:&E_j^{***}+\gamma_{\theta_{j}}\partial_{\theta_{j}}W_j+\gamma_{\rho_j}\partial_{\rho_j}W_j\label{eqnn9999}
	\end{eqnarray}
	and $\mathbf{N}_{**}=(N_1^{**},N_2^{**},\cdots,N_d^{**})$ with
	\begin{eqnarray}
		N_j^{**}&=&\sum_{i=1;i\not=j}^{d}\beta_{i,j}(2W_i(v_{i,**}Q_{j,*}+Q_{i,*}v_{j,**}+v_{i,**}v_{j,**})+W_j(2Q_{i,*}v_{i,**}+v_{i,**}^2)\notag\\
		&&+(v_{i,**}^2+2Q_{i,*}v_{i,**})(v_{j,**}+Q_{j,*})+Q_{i,*}^2v_{j,**})+3\mu_j W_j(v_{j,**}^2+2v_{j,**}Q_{j,*})\notag\\
		&&+v_{j,**}^3+3v_{j,**}^2Q_{j,*}+3v_{j,**}Q_{j,*}^2.\label{eqnn9998}
	\end{eqnarray}
	\begin{lemma}\label{lem0003}
		Suppose the assumptions~$(V_1)$--$(V_2)$ hold and $\beta_{i,j}$ is not a eigenvalue of $-\Delta+\lambda_j$ in $L^2(\bbr^N; w_i^2)$, that is,
		\begin{eqnarray*}
			-\Delta v+\lambda_j v=\beta_{i,j}w_i^2v \quad\text{in }\bbr^N
		\end{eqnarray*}
		has no solutions in $H^2(\bbr^N)$, then
		for $\vartheta$ and $\rho\overline{\alpha}=\rho\min\limits_{1\le j\le d}\alpha_j$ sufficiently large, \eqref{eq9003} is uniquely solvable in the set
		\begin{eqnarray*}
			\mathbb{B}=\{\mathbf{v}\in\mathbb{X}^{\perp}\mid \|\mathbf{v}_{**}\|_{\natural}\lesssim \rho^{-\nu_j}+\widehat{\eta}^{\frac{1-N}{2}}e^{-2\sqrt{\lambda_{n_1}}\widehat{\eta}}\}.
		\end{eqnarray*}
	\end{lemma}
	\begin{proof}
		Since $d\geq2$, we have $\min_{j}\widehat{\eta}_j\geq2\widehat{\eta}$.  It follows from \eqref{eq3013} and \eqref{eqn2040} and similar estimates of \eqref{eqnn0008} and \eqref{eqnn1008} that
		\begin{eqnarray}\label{eqnn0013}
			\|(E_{j}^{***})^{\perp}\|_{\natural,j}\lesssim\|E_{j}^{***}\|_{\natural,j}\lesssim\rho^{-\nu_j}+\widehat{\eta}^{\frac{1-N}{2}}e^{-2\sqrt{\lambda_{n_1}}\widehat{\eta}}
		\end{eqnarray}
		for all $i$.  Note that $(\gamma_{\alpha_j}\partial_{\alpha_j}W_j+\gamma_{\rho_j}\partial_{\rho_j}W_j)^{\perp}=0$, thus,
		by \eqref{eqnn0013},
		\begin{eqnarray*}
			\|(E_{j}^{**})^{\perp}\|_{\natural,j}\lesssim\rho^{-\nu_j}+\widehat{\eta}^{\frac{1-N}{2}}e^{-2\sqrt{\lambda_{n_1}}\widehat{\eta}}.
		\end{eqnarray*}
		Since
		\begin{eqnarray*}
			\|N_j^{**}\|_{\natural,j}&\lesssim&\sum_{i=1;i\not=j}^{d}(\|v_{i,**}\|_{\natural,i}\|Q_{j,*}\|_{\natural,j}+\|v_{j,**}\|_{\natural,j}\|Q_{i,*}\|_{\natural,i}
			+\|v_{i,**}\|_{\natural,i}\|Q_{i,*}\|_{\natural,i}+\|v_{i,**}\|_{\natural,i}^2\notag\\
			&&+\|v_{i,**}\|_{\natural,i}\|v_{j,**}\|_{\natural,j})+\|v_{j,**}\|_{\natural,j}\|Q_{j,*}\|_{\natural,j}+\|v_{j,**}\|_{\natural,j}^2,\label{eqn2041}
		\end{eqnarray*}
		by \eqref{eqn2040}, we can use a fix-point argument to solve \eqref{eq9003} in the set $\mathbb{B}$ and the unique solution $\mathbf{v}_{**}$ satisfies
		\begin{eqnarray}\label{eqn2042}
			\|\mathbf{v}\|_{\natural}\lesssim\rho^{-\nu_j}+\widehat{\eta}^{\frac{1-N}{2}}e^{-2\sqrt{\lambda_{n_1}}\widehat{\eta}}.
		\end{eqnarray}
		It completes the proof.
	\end{proof}

	\section{Estimates of $\int_{\bbr^N}E_{j,3}\partial_{\theta_{j}}W_jdx$ and $\int_{\bbr^N}E_{j,3}\partial_{\rho_j}W_jdx$}
	Clearly,
	\begin{eqnarray*}
		\vartheta^{-1}\int_{\bbr^N}E_{j,3}\partial_{\theta_{j}}W_jdx&=&\vartheta^{-1}\sum_{t=1}^{\vartheta}\sum_{i=1;i\not=j}^{d}\sum_{s=1}^{\vartheta}
		\int_{\bbr^N}\widetilde{w}_{s,i}^2\widetilde{w}_{t,j}\partial_{\theta_{j}}\widetilde{w}_{t,j}dx\\
		&&+\vartheta^{-1}\sum_{t=1}^{\vartheta}\sum_{i=1;i\not=j}^{d}\sum_{s=1}^{\vartheta}\sum_{l=1;l\not=t}^{\vartheta}
		\int_{\bbr^N}\widetilde{w}_{s,i}^2\widetilde{w}_{t,j}\partial_{\theta_{j}}\widetilde{w}_{l,j}dx
	\end{eqnarray*}
	and
	\begin{eqnarray*}
		\vartheta^{-1}\int_{\bbr^N}E_{j,3}\partial_{\rho_j}W_jdx&=&\vartheta^{-1}\sum_{t=1}^{\vartheta}\sum_{i=1;i\not=j}^{d}\sum_{s=1}^{\vartheta}
		\int_{\bbr^N}\widetilde{w}_{s,i}^2\widetilde{w}_{t,j}\partial_{\rho_j}\widetilde{w}_{t,j}dx\\
		&&+\vartheta^{-1}\sum_{t=1}^{\vartheta}\sum_{i=1;i\not=j}^{d}\sum_{s=1}^{\vartheta}\sum_{l=1;l\not=t}^{\vartheta}
		\int_{\bbr^N}\widetilde{w}_{s,i}^2\widetilde{w}_{t,j}\partial_{\rho_j}\widetilde{w}_{l,j}dx.
	\end{eqnarray*}
	Let
	\begin{eqnarray*}
		\xi_{t,j}^{\perp}=\left\{\aligned&\bigg(-\sin\bigg(\theta_{j}+\frac{2(t-1)\pi}{\vartheta}\bigg), \cos\bigg(\theta_{j}+\frac{2(t-1)\pi}{\vartheta}\bigg)\bigg),\quad N=2,\\
		&\bigg(-\sin\bigg(\theta_{j}+\frac{2(t-1)\pi}{\vartheta}\bigg), \cos\bigg(\theta_{j}+\frac{2(t-1)\pi}{\vartheta}\bigg),0\bigg),\quad N=3.
		\endaligned\right.
	\end{eqnarray*}
	\begin{lemma}\label{lem0004}
		Suppose the assumptions~$(V_1)$--$(V_2)$ hold and $\beta_{i,j}$ is not a eigenvalue of $-\Delta+\lambda_j$ in $L^2(\bbr^N; w_i^2)$, that is,
		\begin{eqnarray*}
			-\Delta v+\lambda_j v=\beta_{i,j}w_i^2v \quad\text{in }\bbr^N
		\end{eqnarray*}
		has no solutions in $H^2(\bbr^N)$.  If $\lambda_{n_k}<4\lambda_{n_1}$ under the condition~\eqref{eq0002}, $\alpha_{n_{\tau+1}}<\alpha_{n_{\tau}}$ for all $\tau=1,2,\cdots,k-2$ and $\alpha_{n_k}=(1+o(1))\alpha_{n_1}$ with
		\begin{eqnarray*}
			\max\{\max\{\sqrt{\lambda_{n_\tau}}\alpha_{n_{\tau}}\}, \sqrt{\lambda_{d}}\alpha_{d-1}\}\leq2\min\{\min\{\sqrt{\lambda_{n_\tau'}}\alpha_{n_{\tau'}}\},\sqrt{\lambda_{d}}\alpha_{d-1}\}
		\end{eqnarray*}
		in the case of $k>1$, then
		for $\vartheta$ and $\rho\overline{\alpha}=\rho\min\limits_{1\le j\le d}\alpha_j$ sufficiently large,
		\begin{enumerate}
			\item[$(1)$] for $j\not=n_{\tau-1}+1$, $j\not=n_{\tau}$ or $k=1$,
			\begin{eqnarray*}
				\vartheta^{-1}\int_{\bbr^N}E_{j,3}\partial_{\theta_{j}}W_jdx=\left\{\aligned&\rho^{\frac{1}{2}}(c_1\alpha_{j}^{-\frac12}e^{-2\sqrt{\lambda_{n_\tau}}\rho\alpha_{j}}
				-c_2\alpha_{j-1}^{-\frac12}e^{-2\sqrt{\lambda_{n_\tau}}\rho\alpha_{j-1}})+h.o.t.,\quad N=2,\\
				&\rho^{-1}(\frac{c_1}{\alpha^2_{j}}e^{-2\sqrt{\lambda_{n_\tau}}\rho\alpha_{j}}\log(\rho\alpha_{j})
				-\frac{c_2}{\alpha^2_{j-1}}e^{-2\sqrt{\lambda_{n_\tau}}\rho\alpha_{j-1}}\log(\rho\alpha_{j-1}))+h.o.t.,\quad N=3
				\endaligned\right.
			\end{eqnarray*}
			and
			\begin{eqnarray*}
				\vartheta^{-1}\int_{\bbr^N}E_{j,3}\partial_{\rho_j}W_jdx=\left\{\aligned&\rho^{-\frac{1}{2}}(c_1'\alpha_{j}^{\frac12}e^{-2\sqrt{\lambda_{n_\tau}}\rho\alpha_{j}}
				-c_2'\alpha_{j-1}^{\frac12}e^{-2\sqrt{\lambda_{n_\tau}}\rho\alpha_{j-1}})+h.o.t.,\quad N=2,\\
				&\rho^{-2}(\frac{c_1'}{\alpha_{j}}e^{-2\sqrt{\lambda_{n_\tau}}\rho\alpha_{j}}\log(\rho\alpha_{j})
				-\frac{c_2'}{\alpha_{j-1}}e^{-2\sqrt{\lambda_{n_\tau}}\rho\alpha_{j-1}}\log(\rho\alpha_{j-1}))+h.o.t.,\quad N=3.
				\endaligned\right.
			\end{eqnarray*}
			\item[$(2)$] For $j=n_{\tau-1}+1$ with $\tau>1$, if $n_{\tau}-n_{\tau-1}>1$ then
			\begin{eqnarray*}
				\vartheta^{-1}\int_{\bbr^N}E_{j,3}\partial_{\theta_{j}}W_jdx=\left\{\aligned&c_1\rho^{\frac{1}{2}}\alpha_{j}^{-\frac12}e^{-2\sqrt{\lambda_{n_\tau}}\rho\alpha_{j}}
				-c_2e^{-2\sqrt{\lambda_{n_{\tau-1}}}\rho\alpha_{j-1}}+h.o.t.,\quad N=2,\\
				&\rho^{-1}(\frac{c_1}{\alpha_{j}^2}e^{-2\sqrt{\lambda_{n_\tau}}\rho\alpha_{j}}\log(\rho\alpha_{j})
				-\frac{c_2}{\alpha_{j-1}^2}e^{-2\sqrt{\lambda_{n_\tau-1}}\rho\alpha_{j-1}})+h.o.t.,\quad N=3
				\endaligned\right.
			\end{eqnarray*}
			and
			\begin{eqnarray*}
				\vartheta^{-1}\int_{\bbr^N}E_{j,3}\partial_{\rho_j}W_jdx=\left\{\aligned&c_1\rho^{-\frac{1}{2}}\alpha_{j}^{\frac12}e^{-2\sqrt{\lambda_{n_\tau}}\rho\alpha_{j}}
				-c_2\rho^{-1}e^{-2\sqrt{\lambda_{n_{\tau-1}}}\rho\alpha_{j-1}}+h.o.t.,\quad N=2,\\
				&\rho^{-2}(\frac{c_1}{\alpha_{j}}e^{-2\sqrt{\lambda_{n_\tau}}\rho\alpha_{j+1}}\log(\rho\alpha_{j})
				-\frac{c_2}{\alpha_{j-1}}e^{-2\sqrt{\lambda_{n_\tau-1}}\rho\alpha_{j-1}})+h.o.t.,\quad N=3,
				\endaligned\right.
			\end{eqnarray*}
			while if $n_{\tau}-n_{\tau-1}=1$ then
			\begin{eqnarray*}
				\vartheta^{-1}\int_{\bbr^N}E_{j,3}\partial_{\theta_{j}}W_jdx=\left\{\aligned&(\frac{c_1}{\alpha_{j}}e^{-2\sqrt{\lambda_{n_\tau}}\rho\alpha_{j}}
				-\frac{c_2}{\alpha_{j-1}}e^{-2\sqrt{\lambda_{n_{\tau-1}}}\rho\alpha_{j-1}})+h.o.t.,\quad N=2,\\
				&\rho^{-1}(\frac{c_1}{\alpha_{j+1}^2}e^{-2\sqrt{\lambda_{n_\tau}}\rho\alpha_{j+1}}
				-\frac{c_2}{\alpha_{j}^2}e^{-2\sqrt{\lambda_{n_\tau-1}}\rho\alpha_{j}})+h.o.t.,\quad N=3
				\endaligned\right.
			\end{eqnarray*}
			and
			\begin{eqnarray*}
				\vartheta^{-1}\int_{\bbr^N}E_{j,3}\partial_{\rho_j}W_jdx=\left\{\aligned&\rho^{-1}(c_1e^{-2\sqrt{\lambda_{n_\tau}}\rho\alpha_{j}}
				-c_2e^{-2\sqrt{\lambda_{n_{\tau-1}}}\rho\alpha_{j-1}})+h.o.t.,\quad N=2,\\
				&\rho^{-2}(\frac{c_1}{\alpha_{j}}e^{-2\sqrt{\lambda_{n_\tau}}\rho\alpha_{j}}
				-\frac{c_2}{\alpha_{j-1}}e^{-2\sqrt{\lambda_{n_\tau-1}}\rho\alpha_{j-1}})+h.o.t.,\quad N=3.
				\endaligned\right.
			\end{eqnarray*}
			\item[$(3)$] For $j=1$, if $n_1>1$ then
			\begin{eqnarray*}
				\vartheta^{-1}\int_{\bbr^N}E_{1,3}\partial_{\theta_{1}}W_1dx=\left\{\aligned&c_1\rho^{\frac{1}{2}}\alpha_{1}^{-\frac12}e^{-2\sqrt{\lambda_{n_1}}\rho\alpha_{1}}
				-\frac{c_2}{\alpha_d}e^{-2\sqrt{\lambda_{n_{1}}}\rho\alpha_{d}}+h.o.t.,\quad N=2,\\
				&\rho^{-1}(\frac{c_1}{\alpha_{1}^2}e^{-2\sqrt{\lambda_{n_1}}\rho\alpha_{1}}\log(\rho\alpha_{1})
				-\frac{c_2}{\alpha_{d}^2}e^{-2\sqrt{\lambda_{n_1}}\rho\alpha_{d}})+h.o.t.,\quad N=3
				\endaligned\right.
			\end{eqnarray*}
			and
			\begin{eqnarray*}
				\vartheta^{-1}\int_{\bbr^N}E_{1,3}\partial_{\rho_1}W_1dx=\left\{\aligned&c_1\rho^{-\frac{1}{2}}\alpha_{1}^{\frac12}e^{-2\sqrt{\lambda_{n_1}}\rho\alpha_{1}}
				-c_2\rho^{-1}e^{-2\sqrt{\lambda_{n_{1}}}\rho\alpha_{d}}+h.o.t.,\quad N=2,\\
				&\rho^{-2}(\frac{c_1}{\alpha_1}e^{-2\sqrt{\lambda_{n_1}}\rho\alpha_{1}}\log(\rho\alpha_{1})
				-\frac{c_2}{\alpha_d}e^{-2\sqrt{\lambda_{n_1}}\rho\alpha_{d}})+h.o.t.,\quad N=3,
				\endaligned\right.
			\end{eqnarray*}
			while if $n_{1}=1$ then
			\begin{eqnarray*}
				\vartheta^{-1}\int_{\bbr^N}E_{1,3}\partial_{\theta_{1}}W_1dx=\left\{\aligned&(\frac{c_1}{\alpha_1}e^{-2\sqrt{\lambda_{n_1}}\rho\alpha_{1}}
				-\frac{c_2}{\alpha_d}e^{-2\sqrt{\lambda_{n_{1}}}\rho\alpha_{d}})+h.o.t.,\quad N=2,\\
				&\rho^{-1}(\frac{c_1}{\alpha_{1}^2}e^{-2\sqrt{\lambda_{n_1}}\rho\alpha_{1}}
				-\frac{c_2}{\alpha_{d}^2}e^{-2\sqrt{\lambda_{n_1}}\rho\alpha_{d}})+h.o.t.,\quad N=3
				\endaligned\right.
			\end{eqnarray*}
			and
			\begin{eqnarray*}
				\vartheta^{-1}\int_{\bbr^N}E_{1,3}\partial_{\rho_1}W_1dx=\left\{\aligned&\rho^{-1}(c_1e^{-2\sqrt{\lambda_{n_1}}\rho\alpha_{1}}
				-c_2e^{-2\sqrt{\lambda_{n_{1}}}\rho\alpha_{d}})+h.o.t.,\quad N=2,\\
				&\rho^{-2}(\frac{c_1}{\alpha_1}e^{-2\sqrt{\lambda_{n_1}}\rho\alpha_{1}}
				-\frac{c_2}{\alpha_d}e^{-2\sqrt{\lambda_{n_1}}\rho\alpha_{d}})+h.o.t.,\quad N=3.
				\endaligned\right.
			\end{eqnarray*}
			\item[$(4)$] For $j=n_{\tau}$ with $\tau<k$, if $n_{\tau}-n_{\tau-1}>1$ then
			\begin{eqnarray*}
				\vartheta^{-1}\int_{\bbr^N}E_{j,3}\partial_{\theta_{j}}W_jdx=\left\{\aligned&\frac{c_1}{\alpha_{j}}e^{-2\sqrt{\lambda_{n_{\tau}}}\rho\alpha_{j}}
				-\rho^{\frac{1}{2}}c_2\alpha_{j-1}^{-\frac12}e^{-2\sqrt{\lambda_{n_\tau}}\rho\alpha_{j-1}}+h.o.t.,\quad N=2,\\
				&\rho^{-1}(\frac{c_1}{\alpha_{j}^2}e^{-2\sqrt{\lambda_{n_\tau}}\rho\alpha_{j}}
				-\frac{c_1}{\alpha_{j-1}^2}e^{-2\sqrt{\lambda_{n_\tau}}\rho\alpha_{j-1}}\log(\rho\alpha_{j-1}))+h.o.t.,\quad N=3
				\endaligned\right.
			\end{eqnarray*}
			and
			\begin{eqnarray*}
				\vartheta^{-1}\int_{\bbr^N}E_{j,3}\partial_{\rho_j}W_jdx=\left\{\aligned&c_1\rho^{-1}e^{-2\sqrt{\lambda_{n_{\tau}}}\rho\alpha_{j}}
				-c_2\rho^{-\frac{1}{2}}\alpha_{j-1}^{\frac12}e^{-2\sqrt{\lambda_{n_\tau}}\rho\alpha_{j-1}}+h.o.t.,\quad N=2,\\
				&\rho^{-2}(\frac{c_1}{\alpha_{j}}e^{-2\sqrt{\lambda_{n_\tau}}\rho\alpha_{j}}
				-\frac{c_2}{\alpha_{j-1}}e^{-2\sqrt{\lambda_{n_\tau}}\rho\alpha_{j-1}}\log(\rho\alpha_{j-1}))+h.o.t.,\quad N=3,
				\endaligned\right.
			\end{eqnarray*}
			while if $n_{\tau}-n_{\tau-1}=1$ then
			\begin{eqnarray*}
				\vartheta^{-1}\int_{\bbr^N}E_{j,3}\partial_{\theta_{j}}W_jdx=\left\{\aligned&(\frac{c_1}{\alpha_{j}}e^{-2\sqrt{\lambda_{n_\tau}}\rho\alpha_{j}}
				-\frac{c_2}{\alpha_{j-1}}e^{-2\sqrt{\lambda_{n_{\tau}}}\rho\alpha_{j-1}})+h.o.t.,\quad N=2,\\
				&\rho^{-1}(\frac{c_1}{\alpha_{j}^2}e^{-2\sqrt{\lambda_{n_\tau}}\rho\alpha_{j}}
				-\frac{c_2}{\alpha_{j-1}^2}e^{-2\sqrt{\lambda_{n_\tau}}\rho\alpha_{j-1}})+h.o.t.,\quad N=3
				\endaligned\right.
			\end{eqnarray*}
			and
			\begin{eqnarray*}
				\vartheta^{-1}\int_{\bbr^N}E_{j,3}\partial_{\rho_j}W_jdx=\left\{\aligned&\rho^{-1}(c_1e^{-2\sqrt{\lambda_{n_\tau}}\rho\alpha_{j}}
				-c_2e^{-2\sqrt{\lambda_{n_{\tau}}}\rho\alpha_{j-1}})+h.o.t.,\quad N=2,\\
				&\rho^{-2}(\frac{c_1}{\alpha_{j}}e^{-2\sqrt{\lambda_{n_\tau}}\rho\alpha_{j}}
				-\frac{c_2}{\alpha_{j-1}}e^{-2\sqrt{\lambda_{n_\tau}}\rho\alpha_{j-1}})+h.o.t.,\quad N=3.
				\endaligned\right.
			\end{eqnarray*}
			\item[$(5)$] For $j=d$, if $n_{k}-n_{k-1}>1$ then
			\begin{eqnarray*}
				\vartheta^{-1}\int_{\bbr^N}E_{d,3}\partial_{\theta_{d}}W_ddx=\left\{\aligned&\frac{c_1}{\alpha_{d}}e^{-2\sqrt{\lambda_{n_{1}}}\rho\alpha_{d}}
				-c_2\rho^{\frac{1}{2}}\alpha_{d-1}^{-\frac12}e^{-2\sqrt{\lambda_{n_k}}\rho\alpha_{d-1}}+h.o.t.,\quad N=2,\\
				&\rho^{-1}(\frac{c_1}{\alpha_{d}^2}e^{-2\sqrt{\lambda_{n_1}}\rho\alpha_{d}}
				-\frac{c_1}{\alpha_{d-1}^2}e^{-2\sqrt{\lambda_{n_k}}\rho\alpha_{d-1}}\log(\rho\alpha_{d-1}))+h.o.t.,\quad N=3
				\endaligned\right.
			\end{eqnarray*}
			and
			\begin{eqnarray*}
				\vartheta^{-1}\int_{\bbr^N}E_{d,3}\partial_{\rho_d}W_ddx=\left\{\aligned&c_1\rho^{-1}e^{-2\sqrt{\lambda_{n_{1}}}\rho\alpha_{d}}
				-c_2\rho^{-\frac{1}{2}}\alpha_{d-1}^{\frac12}e^{-2\sqrt{\lambda_{n_k}}\rho\alpha_{d-1}}+h.o.t.,\quad N=2,\\
				&\rho^{-2}(\frac{c_1}{\alpha_{d}}e^{-2\sqrt{\lambda_{n_1}}\rho\alpha_{d}}
				-\frac{c_2}{\alpha_{d-1}}e^{-2\sqrt{\lambda_{n_k}}\rho\alpha_{d-1}}\log(\rho\alpha_{d-1}))+h.o.t.,\quad N=3,
				\endaligned\right.
			\end{eqnarray*}
			while if $n_{k}-n_{k-1}=1$ then
			\begin{eqnarray*}
				\vartheta^{-1}\int_{\bbr^N}E_{d,3}\partial_{\theta_{d}}W_ddx=\left\{\aligned&(\frac{c_1}{\alpha_{d}}e^{-2\sqrt{\lambda_{n_1}}\rho\alpha_{d}}
				-\frac{c_2}{\alpha_{d-1}}e^{-2\sqrt{\lambda_{n_{k-1}}}\rho\alpha_{d-1}})+h.o.t.,\quad N=2,\\
				&\rho^{-1}(\frac{c_1}{\alpha_{d}^2}e^{-2\sqrt{\lambda_{n_1}}\rho\alpha_{d}}
				-\frac{c_2}{\alpha_{d-1}^2}e^{-2\sqrt{\lambda_{n_{k-1}}}\rho\alpha_{d-1}})+h.o.t.,\quad N=3
				\endaligned\right.
			\end{eqnarray*}
			and
			\begin{eqnarray*}
				\vartheta^{-1}\int_{\bbr^N}E_{d,3}\partial_{\rho_d}W_ddx=\left\{\aligned&\rho^{-1}(c_1e^{-2\sqrt{\lambda_{n_1}}\rho\alpha_{d}}
				-c_2e^{-2\sqrt{\lambda_{n_{k-1}}}\rho\alpha_{d-1}})+h.o.t.,\quad N=2,\\
				&\rho^{-2}(\frac{c_1}{\alpha_{d}}e^{-2\sqrt{\lambda_{n_1}}\rho\alpha_{d}}
				-\frac{c_2}{\alpha_{d-1}}e^{-2\sqrt{\lambda_{n_{k-1}}}\rho\alpha_{d-1}})+h.o.t.,\quad N=3.
				\endaligned\right.
			\end{eqnarray*}
		\end{enumerate}
	\end{lemma}
	\begin{proof}
		By \cite[Lemma~A.2]{PV22} and similar calculations in \cite[Lemma~2.6]{PV22},
		\begin{eqnarray*}
			&&\sum_{t=1}^{\vartheta}\sum_{i=1;i\not=j}^{d}\sum_{s=1}^{\vartheta}
			\int_{\bbr^N}\widetilde{w}_{s,i}^2\widetilde{w}_{t,j}\partial_{\theta_{j}}\widetilde{w}_{t,j}dx\\
			&\sim&\sum_{i\in\mathfrak{n}_{\tau'};\tau'<\tau}\sum_{t=1}^{\vartheta}\sum_{s=1}^{\vartheta}|\eta_{s,i}-\eta_{t,j}|^{-1}e^{-2\sqrt{\lambda_{n_\tau'}}|\eta_{s,i}-\eta_{t,j}|}
			\langle\frac{\xi_{s,i}-\xi_{t,j}}{|\xi_{s,i}-\xi_{t,j}|}, \rho_j\xi_{t,j}^{\perp}\rangle\\
			&&+\sum_{i\in\mathfrak{n}_\tau}\sum_{t=1}^{\vartheta}\sum_{s=1}^{\vartheta}|\eta_{s,i}-\eta_{t,j}|^{-\frac12}e^{-2\sqrt{\lambda_{n_\tau}}|\eta_{s,i}-\eta_{t,j}|}
			\langle\frac{\xi_{s,i}-\xi_{t,j}}{|\xi_{s,i}-\xi_{t,j}|}, \rho_j\xi_{t,j}^{\perp}\rangle\\
			&&+\sum_{i\in\mathfrak{n}_{\tau'};\tau'>\tau}\sum_{t=1}^{\vartheta}\sum_{s=1}^{\vartheta}|\eta_{s,i}-\eta_{t,j}|^{-1}e^{-2\sqrt{\lambda_{n_\tau}}|\eta_{s,i}-\eta_{t,j}|}
			\langle\frac{\xi_{s,i}-\xi_{t,j}}{|\xi_{s,i}-\xi_{t,j}|}, \rho_j\xi_{t,j}^{\perp}\rangle
		\end{eqnarray*}
		for $N=2$ and
		\begin{eqnarray*}
			&&\sum_{t=1}^{\vartheta}\sum_{i=1;i\not=j}^{d}\sum_{s=1}^{\vartheta}
			\int_{\bbr^N}\widetilde{w}_{s,i}^2\widetilde{w}_{t,j}\partial_{\theta_{j}}\widetilde{w}_{t,j}dx\\
			&\sim&\sum_{i\in\mathfrak{n}_{\tau'};\tau'<\tau}\sum_{t=1}^{\vartheta}\sum_{s=1}^{\vartheta}|\eta_{s,i}-\eta_{t,j}|^{-2}e^{-2\sqrt{\lambda_{n_\tau'}}|\eta_{s,i}-\eta_{t,j}|}
			\langle\frac{\xi_{s,i}-\xi_{t,j}}{|\xi_{s,i}-\xi_{t,j}|}, \rho_j\xi_{t,j}^{\perp}\rangle\\
			&&+\sum_{i\in\mathfrak{n}_\tau}\sum_{t=1}^{\vartheta}\sum_{s=1}^{\vartheta}|\eta_{s,i}-\eta_{t,j}|^{-2} e^{-2\sqrt{\lambda_{n_\tau}}|\eta_{s,i}-\eta_{t,j}|}\log|\eta_{s,i}-\eta_{t,j}|
			\langle\frac{\xi_{s,i}-\xi_{t,j}}{|\xi_{s,i}-\xi_{t,j}|}, \rho_j\xi_{t,j}^{\perp}\rangle\\
			&&+\sum_{i\in\mathfrak{n}_{\tau'};\tau'>\tau}\sum_{t=1}^{\vartheta}\sum_{s=1}^{\vartheta}|\eta_{s,i}-\eta_{t,j}|^{-2}e^{-2\sqrt{\lambda_{n_\tau}}|\eta_{s,i}-\eta_{t,j}|}
			\langle\frac{\xi_{s,i}-\xi_{t,j}}{|\xi_{s,i}-\xi_{t,j}|}, \rho_j\xi_{t,j}^{\perp}\rangle
		\end{eqnarray*}
		for $N=3$, while
		\begin{eqnarray*}
			&&\sum_{t=1}^{\vartheta}\sum_{i=1;i\not=j}^{d}\sum_{s=1}^{\vartheta}
			\int_{\bbr^N}\widetilde{w}_{s,i}^2\widetilde{w}_{t,j}\partial_{\rho_j}\widetilde{w}_{t,j}dx\\
			&\sim&\sum_{i\in\mathfrak{n}_{\tau'};\tau'<\tau}\sum_{t=1}^{\vartheta}\sum_{s=1}^{\vartheta}|\eta_{s,i}-\eta_{t,j}|^{-1}e^{-2\sqrt{\lambda_{n_\tau'}}|\eta_{s,i}-\eta_{t,j}|}
			\langle\frac{\xi_{s,i}-\xi_{t,j}}{|\xi_{s,i}-\xi_{t,j}|}, \xi_{t,j}\rangle\\
			&&+\sum_{i\in\mathfrak{n}_\tau}\sum_{t=1}^{\vartheta}\sum_{s=1}^{\vartheta}|\eta_{s,i}-\eta_{t,j}|^{-\frac12}e^{-2\sqrt{\lambda_{n_\tau}}|\eta_{s,i}-\eta_{t,j}|}
			\langle\frac{\xi_{s,i}-\xi_{t,j}}{|\xi_{s,i}-\xi_{t,j}|}, \xi_{t,j}\rangle\\
			&&+\sum_{i\in\mathfrak{n}_{\tau'};\tau'>\tau}\sum_{t=1}^{\vartheta}\sum_{s=1}^{\vartheta}|\eta_{s,i}-\eta_{t,j}|^{-1}e^{-2\sqrt{\lambda_{n_\tau}}|\eta_{s,i}-\eta_{t,j}|}
			\langle\frac{\xi_{s,i}-\xi_{t,j}}{|\xi_{s,i}-\xi_{t,j}|}, \xi_{t,j}\rangle
		\end{eqnarray*}
		for $N=2$ and
		\begin{eqnarray*}
			&&\sum_{t=1}^{\vartheta}\sum_{i=1;i\not=j}^{d}\sum_{s=1}^{\vartheta}
			\int_{\bbr^N}\widetilde{w}_{s,i}^2\widetilde{w}_{t,j}\partial_{\rho_j}\widetilde{w}_{t,j}dx\\
			&\sim&\sum_{i\in\mathfrak{n}_{\tau'};\tau'<\tau}\sum_{t=1}^{\vartheta}\sum_{s=1}^{\vartheta}|\eta_{s,i}-\eta_{t,j}|^{-2}e^{-2\sqrt{\lambda_{n_\tau'}}|\eta_{s,i}-\eta_{t,j}|}
			\langle\frac{\xi_{s,i}-\xi_{t,j}}{|\xi_{s,i}-\xi_{t,j}|}, \xi_{t,j}\rangle\\
			&&+\sum_{i\in\mathfrak{n}_\tau}\sum_{t=1}^{\vartheta}\sum_{s=1}^{\vartheta}|\eta_{s,i}-\eta_{t,j}|^{-2} e^{-2\sqrt{\lambda_{n_\tau}}|\eta_{s,i}-\eta_{t,j}|}\log|\eta_{s,i}-\eta_{t,j}|
			\langle\frac{\xi_{s,i}-\xi_{t,j}}{|\xi_{s,i}-\xi_{t,j}|}, \xi_{t,j}\rangle\\
			&&+\sum_{i\in\mathfrak{n}_{\tau'};\tau'>\tau}\sum_{t=1}^{\vartheta}\sum_{s=1}^{\vartheta}|\eta_{s,i}-\eta_{t,j}|^{-2}e^{-2\sqrt{\lambda_{n_\tau}}|\eta_{s,i}-\eta_{t,j}|}
			\langle\frac{\xi_{s,i}-\xi_{t,j}}{|\xi_{s,i}-\xi_{t,j}|}, \xi_{t,j}\rangle
		\end{eqnarray*}
		for $N=3$.  Note that
		\begin{eqnarray*}
			&\langle\frac{\xi_{1,i}-\xi_{1,j}}{|\xi_{1,i}-\xi_{1,j}|}, \xi_{1,j}\rangle=\frac{\cos(\theta_{j+1}-\theta_{i+1})-1}{|\xi_{1,i}-\xi_{1,j}|}\sim-(\theta_{j+1}-\theta_{i+1}),\\
			&\langle\frac{\xi_{1,i}-\xi_{1,j}}{|\xi_{1,i}-\xi_{1,j}|}, \xi_{1,j}^{\perp}\rangle=\frac{\sin(\theta_{j+1}-\theta_{i+1})}{|\xi_{1,i}-\xi_{1,j}|}\sim1,
		\end{eqnarray*}
		and
		\begin{eqnarray*}
			|\eta_{1,i}-\eta_{1,j}|=\rho_j|\theta_{j+1}-\theta_{i+1}|+h.o.t.,
		\end{eqnarray*}
		Thus,
		\begin{enumerate}
			\item[$(1)$] for $j\not=n_{\tau-1}+1$, $j\not=n_{\tau}$ or $k=1$,
			\begin{eqnarray*}
				&&\vartheta^{-1}\sum_{t,i,s;i\not=j}
				\int_{\bbr^N}\widetilde{w}_{s,i}^2\widetilde{w}_{t,j}\partial_{\theta_{j}}\widetilde{w}_{t,j}dx\\
				&=&\left\{\aligned&\rho^{\frac{1}{2}}(c_1\alpha_{j}^{-\frac12}e^{-2\sqrt{\lambda_{n_\tau}}\rho\alpha_{j}}
				-c_2\alpha_{j-1}^{-\frac12}e^{-2\sqrt{\lambda_{n_\tau}}\rho\alpha_{j-1}})+h.o.t.,\quad N=2,\\
				&\rho^{-1}(\frac{c_1}{\alpha^2_{j}}e^{-2\sqrt{\lambda_{n_\tau}}\rho\alpha_{j}}\log(\rho\alpha_{j})
				-\frac{c_2}{\alpha^2_{j-1}}e^{-2\sqrt{\lambda_{n_\tau}}\rho\alpha_{j-1}}\log(\rho\alpha_{j-1}))+h.o.t.,\quad N=3
				\endaligned\right.
			\end{eqnarray*}
			and
			\begin{eqnarray*}
				&&\vartheta^{-1}\sum_{t,i,s;i\not=j}\int_{\bbr^N}\widetilde{w}_{s,i}^2\widetilde{w}_{t,j}\partial_{\rho_j}\widetilde{w}_{t,j}dx\\
				&=&\left\{\aligned&\rho^{-\frac{1}{2}}(c_1'\alpha_{j}^{\frac12}e^{-2\sqrt{\lambda_{n_\tau}}\rho\alpha_{j}}
				-c_2'\alpha_{j-1}^{\frac12}e^{-2\sqrt{\lambda_{n_\tau}}\rho\alpha_{j-1}})+h.o.t.,\quad N=2,\\
				&\rho^{-2}(\frac{c_1'}{\alpha_{j}}e^{-2\sqrt{\lambda_{n_\tau}}\rho\alpha_{j}}\log(\rho\alpha_{j})
				-\frac{c_2'}{\alpha_{j-1}}e^{-2\sqrt{\lambda_{n_\tau}}\rho\alpha_{j-1}}\log(\rho\alpha_{j-1}))+h.o.t.,\quad N=3.
				\endaligned\right.
			\end{eqnarray*}
			\item[$(2)$] For $j=n_{\tau-1}+1$ with $\tau>1$, if $n_{\tau}-n_{\tau-1}>1$ then
			\begin{eqnarray*}
				&&\vartheta^{-1}\sum_{t,i,s;i\not=j}
				\int_{\bbr^N}\widetilde{w}_{s,i}^2\widetilde{w}_{t,j}\partial_{\theta_{j}}\widetilde{w}_{t,j}dx\\
				&=&\left\{\aligned&c_1\rho^{\frac{1}{2}}\alpha_{j}^{-\frac12}e^{-2\sqrt{\lambda_{n_\tau}}\rho\alpha_{j}}
				-c_2e^{-2\sqrt{\lambda_{n_{\tau-1}}}\rho\alpha_{j-1}}+h.o.t.,\quad N=2,\\
				&\rho^{-1}(\frac{c_1}{\alpha_{j}^2}e^{-2\sqrt{\lambda_{n_\tau}}\rho\alpha_{j}}\log(\rho\alpha_{j})
				-\frac{c_2}{\alpha_{j-1}^2}e^{-2\sqrt{\lambda_{n_\tau-1}}\rho\alpha_{j-1}})+h.o.t.,\quad N=3
				\endaligned\right.
			\end{eqnarray*}
			and
			\begin{eqnarray*}
				&&\vartheta^{-1}\sum_{t,i,s;i\not=j}\int_{\bbr^N}\widetilde{w}_{s,i}^2\widetilde{w}_{t,j}\partial_{\rho_j}\widetilde{w}_{t,j}dx\\
				&=&\left\{\aligned&c_1\rho^{-\frac{1}{2}}\alpha_{j}^{\frac12}e^{-2\sqrt{\lambda_{n_\tau}}\rho\alpha_{j}}
				-c_2\rho^{-1}e^{-2\sqrt{\lambda_{n_{\tau-1}}}\rho\alpha_{j-1}}+h.o.t.,\quad N=2,\\
				&\rho^{-2}(\frac{c_1}{\alpha_{j}}e^{-2\sqrt{\lambda_{n_\tau}}\rho\alpha_{j+1}}\log(\rho\alpha_{j})
				-\frac{c_2}{\alpha_{j-1}}e^{-2\sqrt{\lambda_{n_\tau-1}}\rho\alpha_{j-1}})+h.o.t.,\quad N=3,
				\endaligned\right.
			\end{eqnarray*}
			while if $n_{\tau}-n_{\tau-1}=1$ then
			\begin{eqnarray*}
				&&\vartheta^{-1}\sum_{t,i,s;i\not=j}
				\int_{\bbr^N}\widetilde{w}_{s,i}^2\widetilde{w}_{t,j}\partial_{\theta_{j}}\widetilde{w}_{t,j}dx\\
				&=&\left\{\aligned&(\frac{c_1}{\alpha_{j}}e^{-2\sqrt{\lambda_{n_\tau}}\rho\alpha_{j}}
				-\frac{c_2}{\alpha_{j-1}}e^{-2\sqrt{\lambda_{n_{\tau-1}}}\rho\alpha_{j-1}})+h.o.t.,\quad N=2,\\
				&\rho^{-1}(\frac{c_1}{\alpha_{j+1}^2}e^{-2\sqrt{\lambda_{n_\tau}}\rho\alpha_{j+1}}
				-\frac{c_2}{\alpha_{j}^2}e^{-2\sqrt{\lambda_{n_\tau-1}}\rho\alpha_{j}})+h.o.t.,\quad N=3
				\endaligned\right.
			\end{eqnarray*}
			and
			\begin{eqnarray*}
				&&\vartheta^{-1}\sum_{t,i,s;i\not=j}\int_{\bbr^N}\widetilde{w}_{s,i}^2\widetilde{w}_{t,j}\partial_{\rho_j}\widetilde{w}_{t,j}dx\\
				&=&\left\{\aligned&\rho^{-1}(c_1e^{-2\sqrt{\lambda_{n_\tau}}\rho\alpha_{j}}
				-c_2e^{-2\sqrt{\lambda_{n_{\tau-1}}}\rho\alpha_{j-1}})+h.o.t.,\quad N=2,\\
				&\rho^{-2}(\frac{c_1}{\alpha_{j}}e^{-2\sqrt{\lambda_{n_\tau}}\rho\alpha_{j}}
				-\frac{c_2}{\alpha_{j-1}}e^{-2\sqrt{\lambda_{n_\tau-1}}\rho\alpha_{j-1}})+h.o.t.,\quad N=3.
				\endaligned\right.
			\end{eqnarray*}
			\item[$(3)$] For $j=1$, if $n_1>1$ then
			\begin{eqnarray*}
				&&\vartheta^{-1}\sum_{t,i,s;i\not=1}
				\int_{\bbr^N}\widetilde{w}_{s,i}^2\widetilde{w}_{t,1}\partial_{\theta_{1}}\widetilde{w}_{t,1}dx\\
				&=&\left\{\aligned&c_1\rho^{\frac{1}{2}}\alpha_{1}^{-\frac12}e^{-2\sqrt{\lambda_{n_1}}\rho\alpha_{1}}
				-\frac{c_2}{\alpha_d}e^{-2\sqrt{\lambda_{n_{1}}}\rho\alpha_{d}}+h.o.t.,\quad N=2,\\
				&\rho^{-1}(\frac{c_1}{\alpha_{1}^2}e^{-2\sqrt{\lambda_{n_1}}\rho\alpha_{1}}\log(\rho\alpha_{1})
				-\frac{c_2}{\alpha_{d}^2}e^{-2\sqrt{\lambda_{n_1}}\rho\alpha_{d}})+h.o.t.,\quad N=3
				\endaligned\right.
			\end{eqnarray*}
			and
			\begin{eqnarray*}
				&&\vartheta^{-1}\sum_{t,i,s;i\not=1}\int_{\bbr^N}\widetilde{w}_{s,i}^2\widetilde{w}_{t,1}\partial_{\rho_1}\widetilde{w}_{t,1}dx\\
				&=&\left\{\aligned&c_1\rho^{-\frac{1}{2}}\alpha_{1}^{\frac12}e^{-2\sqrt{\lambda_{n_1}}\rho\alpha_{1}}
				-c_2\rho^{-1}e^{-2\sqrt{\lambda_{n_{1}}}\rho\alpha_{d}}+h.o.t.,\quad N=2,\\
				&\rho^{-2}(\frac{c_1}{\alpha_1}e^{-2\sqrt{\lambda_{n_1}}\rho\alpha_{1}}\log(\rho\alpha_{1})
				-\frac{c_2}{\alpha_d}e^{-2\sqrt{\lambda_{n_1}}\rho\alpha_{d}})+h.o.t.,\quad N=3,
				\endaligned\right.
			\end{eqnarray*}
			while if $n_{1}=1$ then
			\begin{eqnarray*}
				&&\vartheta^{-1}\sum_{t,i,s;i\not=1}
				\int_{\bbr^N}\widetilde{w}_{s,i}^2\widetilde{w}_{t,1}\partial_{\theta_{1}}\widetilde{w}_{t,1}dx\\
				&=&\left\{\aligned&(\frac{c_1}{\alpha_1}e^{-2\sqrt{\lambda_{n_1}}\rho\alpha_{1}}
				-\frac{c_2}{\alpha_d}e^{-2\sqrt{\lambda_{n_{1}}}\rho\alpha_{d}})+h.o.t.,\quad N=2,\\
				&\rho^{-1}(\frac{c_1}{\alpha_{1}^2}e^{-2\sqrt{\lambda_{n_1}}\rho\alpha_{1}}
				-\frac{c_2}{\alpha_{d}^2}e^{-2\sqrt{\lambda_{n_1}}\rho\alpha_{d}})+h.o.t.,\quad N=3
				\endaligned\right.
			\end{eqnarray*}
			and
			\begin{eqnarray*}
				&&\vartheta^{-1}\sum_{t,i,s;i\not=1}\int_{\bbr^N}\widetilde{w}_{s,i}^2\widetilde{w}_{t,1}\partial_{\rho_1}\widetilde{w}_{t,1}dx\\
				&=&\left\{\aligned&\rho^{-1}(c_1e^{-2\sqrt{\lambda_{n_1}}\rho\alpha_{1}}
				-c_2e^{-2\sqrt{\lambda_{n_{1}}}\rho\alpha_{d}})+h.o.t.,\quad N=2,\\
				&\rho^{-2}(\frac{c_1}{\alpha_1}e^{-2\sqrt{\lambda_{n_1}}\rho\alpha_{1}}
				-\frac{c_2}{\alpha_d}e^{-2\sqrt{\lambda_{n_1}}\rho\alpha_{d}})+h.o.t.,\quad N=3.
				\endaligned\right.
			\end{eqnarray*}
			\item[$(4)$] For $j=n_{\tau}$ with $\tau<k$, if $n_{\tau}-n_{\tau-1}>1$ then
			\begin{eqnarray*}
				&&\vartheta^{-1}\sum_{t,i,s;i\not=j}
				\int_{\bbr^N}\widetilde{w}_{s,i}^2\widetilde{w}_{t,j}\partial_{\theta_{j}}\widetilde{w}_{t,j}dx\\
				&=&\left\{\aligned&\frac{c_1}{\alpha_{j}}e^{-2\sqrt{\lambda_{n_{\tau}}}\rho\alpha_{j}}
				-\rho^{\frac{1}{2}}c_2\alpha_{j-1}^{-\frac12}e^{-2\sqrt{\lambda_{n_\tau}}\rho\alpha_{j-1}}+h.o.t.,\quad N=2,\\
				&\rho^{-1}(\frac{c_1}{\alpha_{j}^2}e^{-2\sqrt{\lambda_{n_\tau}}\rho\alpha_{j}}
				-\frac{c_1}{\alpha_{j-1}^2}e^{-2\sqrt{\lambda_{n_\tau}}\rho\alpha_{j-1}}\log(\rho\alpha_{j-1}))+h.o.t.,\quad N=3
				\endaligned\right.
			\end{eqnarray*}
			and
			\begin{eqnarray*}
				&&\vartheta^{-1}\sum_{t,i,s;i\not=j}\int_{\bbr^N}\widetilde{w}_{s,i}^2\widetilde{w}_{t,j}\partial_{\rho_j}\widetilde{w}_{t,j}dx\\
				&=&\left\{\aligned&c_1\rho^{-1}e^{-2\sqrt{\lambda_{n_{\tau}}}\rho\alpha_{j}}
				-c_2\rho^{-\frac{1}{2}}\alpha_{j-1}^{\frac12}e^{-2\sqrt{\lambda_{n_\tau}}\rho\alpha_{j-1}}+h.o.t.,\quad N=2,\\
				&\rho^{-2}(\frac{c_1}{\alpha_{j}}e^{-2\sqrt{\lambda_{n_\tau}}\rho\alpha_{j}}
				-\frac{c_2}{\alpha_{j-1}}e^{-2\sqrt{\lambda_{n_\tau}}\rho\alpha_{j-1}}\log(\rho\alpha_{j-1}))+h.o.t.,\quad N=3,
				\endaligned\right.
			\end{eqnarray*}
			while if $n_{\tau}-n_{\tau-1}=1$ then
			\begin{eqnarray*}
				&&\vartheta^{-1}\sum_{t,i,s;i\not=j}
				\int_{\bbr^N}\widetilde{w}_{s,i}^2\widetilde{w}_{t,j}\partial_{\theta_{j}}\widetilde{w}_{t,j}dx\\
				&=&\left\{\aligned&(\frac{c_1}{\alpha_{j}}e^{-2\sqrt{\lambda_{n_\tau}}\rho\alpha_{j}}
				-\frac{c_2}{\alpha_{j-1}}e^{-2\sqrt{\lambda_{n_{\tau}}}\rho\alpha_{j-1}})+h.o.t.,\quad N=2,\\
				&\rho^{-1}(\frac{c_1}{\alpha_{j}^2}e^{-2\sqrt{\lambda_{n_\tau}}\rho\alpha_{j}}
				-\frac{c_2}{\alpha_{j-1}^2}e^{-2\sqrt{\lambda_{n_\tau}}\rho\alpha_{j-1}})+h.o.t.,\quad N=3
				\endaligned\right.
			\end{eqnarray*}
			and
			\begin{eqnarray*}
				&&\vartheta^{-1}\sum_{t,i,s;i\not=j}\int_{\bbr^N}\widetilde{w}_{s,i}^2\widetilde{w}_{t,j}\partial_{\rho_j}\widetilde{w}_{t,j}dx\\
				&=&\left\{\aligned&\rho^{-1}(c_1e^{-2\sqrt{\lambda_{n_\tau}}\rho\alpha_{j}}
				-c_2e^{-2\sqrt{\lambda_{n_{\tau}}}\rho\alpha_{j-1}})+h.o.t.,\quad N=2,\\
				&\rho^{-2}(\frac{c_1}{\alpha_{j}}e^{-2\sqrt{\lambda_{n_\tau}}\rho\alpha_{j}}
				-\frac{c_2}{\alpha_{j-1}}e^{-2\sqrt{\lambda_{n_\tau}}\rho\alpha_{j-1}})+h.o.t.,\quad N=3.
				\endaligned\right.
			\end{eqnarray*}
			\item[$(5)$] For $j=d$, if $n_{k}-n_{k-1}>1$ then
			\begin{eqnarray*}
				&&\vartheta^{-1}\sum_{t,i,s;i\not=d}
				\int_{\bbr^N}\widetilde{w}_{s,i}^2\widetilde{w}_{t,d}\partial_{\theta_{d}}\widetilde{w}_{t,d}dx\\
				&=&\left\{\aligned&\frac{c_1}{\alpha_{d}}e^{-2\sqrt{\lambda_{n_{1}}}\rho\alpha_{d}}
				-c_2\rho^{\frac{1}{2}}\alpha_{d-1}^{-\frac12}e^{-2\sqrt{\lambda_{n_k}}\rho\alpha_{d-1}}+h.o.t.,\quad N=2,\\
				&\rho^{-1}(\frac{c_1}{\alpha_{d}^2}e^{-2\sqrt{\lambda_{n_1}}\rho\alpha_{d}}
				-\frac{c_1}{\alpha_{d-1}^2}e^{-2\sqrt{\lambda_{n_k}}\rho\alpha_{d-1}}\log(\rho\alpha_{d-1}))+h.o.t.,\quad N=3
				\endaligned\right.
			\end{eqnarray*}
			and
			\begin{eqnarray*}
				&&\vartheta^{-1}\sum_{t,i,s;i\not=d}\int_{\bbr^N}\widetilde{w}_{s,i}^2\widetilde{w}_{t,d}\partial_{\rho_{d}}\widetilde{w}_{t,d}dx\\
				&=&\left\{\aligned&c_1\rho^{-1}e^{-2\sqrt{\lambda_{n_{1}}}\rho\alpha_{d}}
				-c_2\rho^{-\frac{1}{2}}\alpha_{d-1}^{\frac12}e^{-2\sqrt{\lambda_{n_k}}\rho\alpha_{d-1}}+h.o.t.,\quad N=2,\\
				&\rho^{-2}(\frac{c_1}{\alpha_{d}}e^{-2\sqrt{\lambda_{n_1}}\rho\alpha_{d}}
				-\frac{c_2}{\alpha_{d-1}}e^{-2\sqrt{\lambda_{n_k}}\rho\alpha_{d-1}}\log(\rho\alpha_{d-1}))+h.o.t.,\quad N=3,
				\endaligned\right.
			\end{eqnarray*}
			while if $n_{k}-n_{k-1}=1$ then
			\begin{eqnarray*}
				&&\vartheta^{-1}\sum_{t,i,s;i\not=d}
				\int_{\bbr^N}\widetilde{w}_{s,i}^2\widetilde{w}_{t,d}\partial_{\theta_{d}}\widetilde{w}_{t,d}dx\\
				&=&\left\{\aligned&(\frac{c_1}{\alpha_{d}}e^{-2\sqrt{\lambda_{n_1}}\rho\alpha_{d}}
				-\frac{c_2}{\alpha_{d-1}}e^{-2\sqrt{\lambda_{n_{k-1}}}\rho\alpha_{d-1}})+h.o.t.,\quad N=2,\\
				&\rho^{-1}(\frac{c_1}{\alpha_{d}^2}e^{-2\sqrt{\lambda_{n_1}}\rho\alpha_{d}}
				-\frac{c_2}{\alpha_{d-1}^2}e^{-2\sqrt{\lambda_{n_{k-1}}}\rho\alpha_{d-1}})+h.o.t.,\quad N=3
				\endaligned\right.
			\end{eqnarray*}
			and
			\begin{eqnarray*}
				&&\vartheta^{-1}\sum_{t,i,s;i\not=d}\int_{\bbr^N}\widetilde{w}_{s,i}^2\widetilde{w}_{t,d}\partial_{\rho_{d}}\widetilde{w}_{t,d}dx\\
				&=&\left\{\aligned&\rho^{-1}(c_1e^{-2\sqrt{\lambda_{n_1}}\rho\alpha_{d}}
				-c_2e^{-2\sqrt{\lambda_{n_{k-1}}}\rho\alpha_{d-1}})+h.o.t.,\quad N=2,\\
				&\rho^{-2}(\frac{c_1}{\alpha_{d}}e^{-2\sqrt{\lambda_{n_1}}\rho\alpha_{d}}
				-\frac{c_2}{\alpha_{d-1}}e^{-2\sqrt{\lambda_{n_{k-1}}}\rho\alpha_{d-1}})+h.o.t.,\quad N=3.
				\endaligned\right.
			\end{eqnarray*}
		\end{enumerate}
		On the other hand, for every $s$ and $i$, by symmetry,
		\begin{eqnarray*}
			\sum_{t=1}^{\vartheta}\sum_{l=1;l\not=t}^{\vartheta}
			\int_{\bbr^N}\widetilde{w}_{s,i}^2\widetilde{w}_{t,j}\partial_{\rho_j}\widetilde{w}_{l,j}dx&=&\sum_{t=1}^{\vartheta}\sum_{i=1;i\not=j}^{d}\sum_{l=1;l>t}^{\vartheta}
			\int_{\bbr^N}\widetilde{w}_{s,i}^2\partial_{\rho_j}(\widetilde{w}_{t,j}\widetilde{w}_{l,j})dx\\
			&=&\sum_{t,l;|s-l|\leq|s-t|}
			\int_{\bbr^N}\widetilde{w}_{s,i}^2\partial_{\rho_j}(\widetilde{w}_{t,j}\widetilde{w}_{l,j})dx\\
			&&+\sum_{t,l;|s-t|\leq|s-l|}
			\int_{\bbr^N}\widetilde{w}_{s,i}^2\partial_{\rho_j}(\widetilde{w}_{t,j}\widetilde{w}_{l,j})dx\\
			&=&2\sum_{t,l;|s-t|\leq|s-l|}
			\int_{\bbr^N}\widetilde{w}_{s,i}^2\partial_{\rho_j}(\widetilde{w}_{t,j}\widetilde{w}_{l,j})dx
		\end{eqnarray*}
		and
		\begin{eqnarray*}
			\sum_{t=1}^{\vartheta}\sum_{l=1;l\not=t}^{\vartheta}
			\int_{\bbr^N}\widetilde{w}_{s,i}^2\widetilde{w}_{t,j}\partial_{\theta_{j}}\widetilde{w}_{l,j}dx&=&\sum_{t=1}^{\vartheta}\sum_{i=1;i\not=j}^{d}\sum_{l=1;l>t}^{\vartheta}
			\int_{\bbr^N}\widetilde{w}_{s,i}^2\partial_{\theta_{j}}(\widetilde{w}_{t,j}\widetilde{w}_{l,j})dx\\
			&=&\sum_{t,l;|s-l|\leq|s-t|}
			\int_{\bbr^N}\widetilde{w}_{s,i}^2\partial_{\theta_{j}}(\widetilde{w}_{t,j}\widetilde{w}_{l,j})dx\\
			&&+\sum_{t,l;|s-t|\leq|s-l|}
			\int_{\bbr^N}\widetilde{w}_{s,i}^2\partial_{\theta_{j}}(\widetilde{w}_{t,j}\widetilde{w}_{l,j})dx\\
			&=&2\sum_{t,l;|s-t|\leq|s-l|}
			\int_{\bbr^N}\widetilde{w}_{s,i}^2\partial_{\theta_{j}}(\widetilde{w}_{t,j}\widetilde{w}_{l,j})dx.
		\end{eqnarray*}
		By \cite[Lemma~5.1]{ACR07} and the assumption $\lambda_{n_k}<4\lambda_{n_1}$, for every $j\in \mathbf{n}_\tau$ with $\tau=1,2,\cdots,k$ and $l\not=t$,
		\begin{eqnarray*}
			\int_{\bbr^N}\widetilde{w}_{s,i}^2\widetilde{w}_{t,j}\widetilde{w}_{l,j}dx=o(e^{-\sqrt{\lambda_{n_\tau}}\eta_{**}}).
		\end{eqnarray*}
		where $\eta_{**}<\widetilde{\eta}$ with $\widetilde{\eta}=\min_{y\in\bbr^N}\{|y-\eta_{s,i}|+|y-\eta_{t,j}|+|y-\eta_{l,j}|\}$.
		Thus, by similar arguments as that used for \cite[Lemma~A.2]{PV22},
		\begin{eqnarray}\label{eqnnnew0030}
			\rho_j^{-1}\int_{\bbr^N}\widetilde{w}_{s,i}^2\partial_{\theta_{j}}(\widetilde{w}_{t,j}\widetilde{w}_{l,j})dx=o( e^{-\sqrt{\lambda_{n_\tau}}\eta_{**}})
		\end{eqnarray}
		and
		\begin{eqnarray}\label{eqnnnew0031}
			\int_{\bbr^N}\widetilde{w}_{s,i}^2\partial_{\rho_j}(\widetilde{w}_{t,j}\widetilde{w}_{l,j})dx
			=o(e^{-\sqrt{\lambda_{n_\tau}}\eta_{**}}).
		\end{eqnarray}
		Since $\{\eta_{s,i}\}\subset\bbr^2$, it is well known that $\widetilde{\eta}$ is attained by the Fermat point.  Thus, in the case of $|s-t|\leq|s-l|$ for every $s$, either
		\begin{eqnarray*}
			\widetilde{\eta}=\left\{\aligned&|\eta_{s,i}-\eta_{t,j}|+|\eta_{s,i}-\eta_{l,j}|,\quad\text{$\eta_{s,i}$ is the middle point},\\
			&|\eta_{s,i}-\eta_{t,j}|+|\eta_{t,j}-\eta_{l,j}|,\quad\text{$\eta_{t,j}$ is the middle point}\endaligned\right.
		\end{eqnarray*}
		or $\widetilde{\eta}>\frac12(|\eta_{s,i}-\eta_{t,j}|+|\eta_{s,i}-\eta_{l,j}|+|\eta_{t,j}-\eta_{l,j}|)\gtrsim\rho$.
		For these points $\eta_{t,j}$ and $\eta_{l,j}$ which satisfy $|s-t|\leq|s-l|$ and $\widetilde{\eta}>\frac12(|\eta_{s,i}-\eta_{t,j}|+|\eta_{s,i}-\eta_{l,j}|+|\eta_{t,j}-\eta_{l,j}|)\gtrsim\rho$.  We call them ``good points'' for the sake of simplicity.  Then, by \eqref{eqnnnew0030} and \eqref{eqnnnew0031},
		\begin{eqnarray*}
			\sum_{t,l;\text{good points}}
			\int_{\bbr^N}\widetilde{w}_{s,i}^2\partial_{\theta_{j}}(\widetilde{w}_{t,j}\widetilde{w}_{l,j})dx=o(\sum_{t,l}
			\int_{\bbr^N}\widetilde{w}_{s,i}^2\widetilde{w}_{t,j}\partial_{\theta_{j}}\widetilde{w}_{t,j}dx)
		\end{eqnarray*}
		and
		\begin{eqnarray*}
			\sum_{t,l;\text{good points}}
			\int_{\bbr^N}\widetilde{w}_{s,i}^2\partial_{\rho_{j}}(\widetilde{w}_{t,j}\widetilde{w}_{l,j})dx=o(\sum_{t,l}
			\int_{\bbr^N}\widetilde{w}_{s,i}^2\widetilde{w}_{t,1}\partial_{\rho_{j}}\widetilde{w}_{t,j}dx)
		\end{eqnarray*}
		for all $s$ and $i\not=j$.  For these points $\eta_{t,j}$ and $\eta_{l,j}$ which satisfy $|s-t|\leq|s-l|$ and $\widetilde{\eta}=|\eta_{s,i}-\eta_{t,j}|+|\eta_{s,i}-\eta_{l,j}|$.  We call them ``bad points'' for the sake of simplicity.  For these points, we observe that since $\alpha_j\sim\frac{1}{\vartheta}$, then for $d\geq3$,
		\begin{eqnarray*}
			\widetilde{\eta}\geq\widehat{\eta}+2\min\{|\eta_{s,j+1}-\eta_{s,j}|,|\eta_{s,j-1}-\eta_{s,j}|\}+O(\frac{|l-s|\rho}{\vartheta}).
		\end{eqnarray*}
		Thus, by similar arguments of \eqref{eq3003},
		\begin{eqnarray*}
			\sum_{t,l;\text{bad points}}
			\int_{\bbr^N}\widetilde{w}_{s,i}^2\partial_{\theta_{j}}(\widetilde{w}_{t,j}\widetilde{w}_{l,j})dx=o(\sum_{t,l}
			\int_{\bbr^N}\widetilde{w}_{s,i}^2\widetilde{w}_{t,j}\partial_{\theta_{j}}\widetilde{w}_{t,j}dx)
		\end{eqnarray*}
		and
		\begin{eqnarray*}
			\sum_{t,l;\text{bad points}}
			\int_{\bbr^N}\widetilde{w}_{s,i}^2\partial_{\rho_{j}}(\widetilde{w}_{t,j}\widetilde{w}_{l,j})dx=o(\sum_{t,l}
			\int_{\bbr^N}\widetilde{w}_{s,i}^2\widetilde{w}_{t,1}\partial_{\rho_{j}}\widetilde{w}_{t,j}dx)
		\end{eqnarray*}
		for $d\geq3$.  For $d=2$,
		except the terms
		\begin{eqnarray*}
			\int_{\bbr^N}\widetilde{w}_{t,1}^2\partial_{\theta_{2}}(\widetilde{w}_{t,2}\widetilde{w}_{t-1,2})dx,\quad
			\int_{\bbr^N}\widetilde{w}_{t,1}^2\partial_{\rho_{2}}(\widetilde{w}_{t,2}\widetilde{w}_{t-1,2})dx,
		\end{eqnarray*}
		and
		\begin{eqnarray*}
			\int_{\bbr^N}\widetilde{w}_{t-1,2}^2\partial_{\theta_{1}}(\widetilde{w}_{t,1}\widetilde{w}_{t-1,1})dx,\quad
			\int_{\bbr^N}\widetilde{w}_{t-1,2}^2\partial_{\rho_{1}}(\widetilde{w}_{t,1}\widetilde{w}_{t-1,1})dx,
		\end{eqnarray*}
		the other terms in $\sum_{t,l}
		\int_{\bbr^N}\widetilde{w}_{s,i}^2\partial_{\theta_{j}}(\widetilde{w}_{t,j}\widetilde{w}_{l,j})dx$ and $\sum_{t,l}
		\int_{\bbr^N}\widetilde{w}_{s,i}^2\partial_{\rho_{j}}(\widetilde{w}_{t,j}\widetilde{w}_{l,j})dx$ are $o(\sum_{t,l}
		\int_{\bbr^N}\widetilde{w}_{s,i}^2\widetilde{w}_{t,j}\partial_{\theta_{j}}\widetilde{w}_{t,j}dx)$ and $o(\sum_{t,l}
		\int_{\bbr^N}\widetilde{w}_{s,i}^2\widetilde{w}_{t,1}\partial_{\rho_{j}}\widetilde{w}_{t,j}dx)$, respectively, even if we sum them up in terms of $l$ and $t$ by similar arguments for $d\geq3$.  For the term
		\begin{eqnarray*}
			\int_{\bbr^N}\widetilde{w}_{t,1}^2\widetilde{w}_{t,2}\widetilde{w}_{t-1,2}dx,
		\end{eqnarray*}
		Without loss of generality, we may assume that $|\eta_{t,2}-\eta_{t,1}|\leq|\eta_{t-1,2}-\eta_{t,1}|$ and denote $\widehat{\eta}=|\eta_{t,2}-\eta_{t,1}|=\alpha\widetilde{\eta}_{2}$ with $\alpha\leq\frac12$.  Moreover, by translations and rotations if necessary, we may assume that $\eta_{t,2}=0$ and denote $\eta_{t-1,2}=(\eta_{t-1,2}^0,0,0,\cdots,0)$.  Note that $\eta_{t,2}$, $\eta_{t,1}$ and $\eta_{t-1,2}$ are almost on the same line by $\rho_{j}=\rho+O(1)$ for all $j$.  Thus, $\eta_{t,1}=(\eta_{t,1}^0,o(1),o(1),\cdots,o(1))$.  Now, as that in the proof of \cite[Lemma~3.7]{ACR07}, we rewrite $x=(x_1,x')$.  Then,
		\begin{eqnarray*}
			\int_{\bbr^N}\widetilde{w}_{t,1}^2\widetilde{w}_{t,2}\widetilde{w}_{t-1,2}dx&\sim&\int_{x_1\leq1}\widetilde{w}_{t,1}^2\widetilde{w}_{t,2}\widetilde{w}_{t-1,2}dx_1dx'\\
			&&+\int_{1<x_1\leq\frac{\eta_{t-1,2}^0}{2};|x'|\leq x_1}\widetilde{w}_{t,1}^2\widetilde{w}_{t,2}\widetilde{w}_{t-1,2}dx_1dx'\\
			&&+\int_{\frac{\eta_{t-1,2}^0}{2}<x_1\leq\eta_{t-1,2}^0-1;|x'|\leq x_1}\widetilde{w}_{t,1}^2\widetilde{w}_{t,2}\widetilde{w}_{t-1,2}dx_1dx'\\
			&&+\int_{x_1\geq\eta_{t-1,2}^0-1}\widetilde{w}_{t,1}^2\widetilde{w}_{t,2}\widetilde{w}_{t-1,2}dx_1dx'\\
			&&+\int_{1<x_1\leq\frac{\eta_{t-1,2}^0}{2};|x'|> x_1}\widetilde{w}_{t,1}^2\widetilde{w}_{t,2}\widetilde{w}_{t-1,2}dx_1dx'\\
			&&+\int_{\frac{\eta_{t-1,2}^0}{2}<x_1\leq\eta_{t-1,2}^0-1;|x'|> x_1}\widetilde{w}_{t,1}^2\widetilde{w}_{t,2}\widetilde{w}_{t-1,2}dx_1dx'.
		\end{eqnarray*}
		By \eqref{eqnnnew0008},
		\begin{eqnarray*}
			\int_{x_1\leq1}\widetilde{w}_{t,1}^2\widetilde{w}_{t,2}\widetilde{w}_{t-1,2}dx_1dx'
			&\lesssim&(\eta_{t,1}^0)^{\frac{3(1-N)}{2}}e^{-(2\sqrt{\lambda_1}\eta_{t,1}^0+\sqrt{\lambda_2}\eta_{t-1,2}^0)}\\
			&\sim&\widehat{\eta}^{-\frac{3(1-N)}{2}}e^{-(2\sqrt{\lambda_1}\widehat{\eta}+\sqrt{\lambda_2}\widetilde{\eta}_{2})}.
		\end{eqnarray*}
		Since $|x-\eta_{t,1}^0|=\sqrt{|x_1-\eta_{t,1}^0|^2+|x'|^2}\sim|x_1-\eta_{t,1}^0|+|x'|$, by \cite[(5.6)]{ACR07},
		\begin{eqnarray*}
			&&\frac{\int_{1<x_1\leq\eta_{t,1}^0;|x'|\leq x_1}\widetilde{w}_{t,1}^2\widetilde{w}_{t,2}\widetilde{w}_{t-1,2}dx_1dx'}{(\eta_{t,1}^0)^{1-N}e^{-\sqrt{\lambda_2}\eta_{t-1,2}^0}}\\
			&\sim&\int_{1}^{\eta_{t,1}^0}(|x_1-\eta_{t,1}^0|+|x'|+1)^{1-N}e^{-c|x_1-\eta_{t,1}^0|}\int_{|x'|\leq x_1}e^{-(c|x'|+\frac{c'|x'|^2}{x_1})}dx'.
		\end{eqnarray*}
		Since
		\begin{eqnarray*}
			&&\int_{1}^{\eta_{t,1}^0}(|x_1-\eta_{t,1}^0|+|x'|+1)^{1-N}e^{-c|x_1-\eta_{t,1}^0|}\int_{|x'|\leq x_1}e^{-(c|x'|+\frac{c'|x'|^2}{x_1})}dx'\\
			&\lesssim&
			\int_{1}^{\eta_{t,1}^0}(|x_1-\eta_{t,1}^0|+1)^{1-N}e^{-c|x_1-\eta_{t,1}^0|}\int_{|x'|\leq x_1}e^{-c|x'|}dx'\\
			&\sim&1
		\end{eqnarray*}
		and
		\begin{eqnarray*}
			&&\int_{1}^{\eta_{t,1}^0}(|x_1-\eta_{t,1}^0|+|x'|+1)^{1-N}e^{-c|x_1-\eta_{t,1}^0|}\int_{|x'|\leq x_1}e^{-(c|x'|+\frac{c'|x'|^2}{x_1})}dx'\\
			&\gtrsim&
			\int_{\eta_{t,1}^0-1}^{\eta_{t,1}^0}(|x_1-\eta_{t,1}^0|+|x'|+1)^{1-N}e^{-c|x_1-\eta_{t,1}^0|}\int_{|x'|\leq 1}e^{-c''|x'|}dx'\\
			&\sim&1,
		\end{eqnarray*}
		we have
		\begin{eqnarray*}
			\int_{1<x_1\leq\eta_{t,1}^0;|x'|\leq x_1}\widetilde{w}_{t,1}^2\widetilde{w}_{t,2}\widetilde{w}_{t-1,2}dx_1dx'&\sim&(\eta_{t,1}^0)^{1-N}e^{-\sqrt{\lambda_2}\eta_{t-1,2}^0}\\
			&\sim&\widehat{\eta}^{1-N}e^{-\frac{\sqrt{\lambda_2}}{\alpha}|\eta_{t,2}-\eta_{t,1}|}.
		\end{eqnarray*}
		Note that $\widetilde{w}_{t,1}^2\lesssim(\eta_{t,1}^0)^{1-N}e^{-\sqrt{2\lambda_1}\eta_{t,1}^0}$, by similar arguments as that used for \cite[Lemma~3.7]{ACR07},
		\begin{eqnarray*}
			\int_{1<x_1\leq\frac{\eta_{t-1,2}^0}{2};|x'|> x_1}\widetilde{w}_{t,1}^2\widetilde{w}_{t,2}\widetilde{w}_{t-1,2}dx_1dx'&\lesssim&(\eta_{t,1}^0)^{\frac{3(1-N)}{2}}e^{-(\sqrt{2\lambda_1}\eta_{t,1}^0+\sqrt{\lambda_2}\eta_{t-1,2}^0)}\\
			&\sim&\widehat{\eta}^{-\frac{3(1-N)}{2}}e^{-(\sqrt{2\lambda_1}\widehat{\eta}+\sqrt{\lambda_2}\widetilde{\eta}_{2})}.
		\end{eqnarray*}
		The other terms can be estimated similarly and thus, we have
		\begin{eqnarray*}
			\int_{\bbr^N}\widetilde{w}_{t,1}^2\widetilde{w}_{t,2}\widetilde{w}_{t-1,2}dx\sim\widehat{\eta}^{1-N}e^{-\frac{\sqrt{\lambda_2}}{\alpha}|\eta_{t,2}-\eta_{t,1}|}.
		\end{eqnarray*}
		Similarly, if we assume that $|\eta_{t-1,2}-\eta_{t-1,1}|\leq|\eta_{t-1,2}-\eta_{t,1}|$ and denote $\widehat{\eta}=|\eta_{t-1,2}-\eta_{t-1,1}|=\alpha\widetilde{\eta}_{1}$ with $\alpha\leq\frac12$, we also have
		\begin{eqnarray*}
			\int_{\bbr^N}\widetilde{w}_{t-1,2}^2\widetilde{w}_{t,1}\widetilde{w}_{t-1,1}dx\sim
			|\eta_{t-1,2}-\eta_{t-1,1}|^{1-N}e^{-\frac{\sqrt{\lambda_1}}{\alpha}|\eta_{t-1,2}-\eta_{t-1,1}|}.
		\end{eqnarray*}
		Now, by similar arguments as that used for \cite[Lemma~A.2]{PV22}, similar calculations in \cite[Lemma~2.6]{PV22} and the symmetry of the construction of $\{\eta_{t,j}\}$,
		\begin{eqnarray*}
			\int_{\bbr^N}\widetilde{w}_{t,1}^2\partial_{\theta_{2}}(\widetilde{w}_{t,2}\widetilde{w}_{t-1,2})dx
			\sim\widehat{\eta}^{1-N}e^{-\frac{\sqrt{\lambda_2}}{\alpha}\widehat{\eta}}
			\langle\frac{\xi_{t,2}-\xi_{t,1}}{|\xi_{t,2}-\xi_{t,1}|}, \rho_2\xi_{t,2}^{\perp}\rangle
		\end{eqnarray*}
		and
		\begin{eqnarray*}
			\int_{\bbr^N}\widetilde{w}_{t,1}^2\partial_{\rho_2}(\widetilde{w}_{t,2}\widetilde{w}_{t-1,2})dx
			\sim\widehat{\eta}^{1-N}e^{-\frac{\sqrt{\lambda_2}}{\alpha}\widehat{\eta}}
			\langle\frac{\xi_{t,2}-\xi_{t,1}}{|\xi_{t,2}-\xi_{t,1}|}, \xi_{t,2}\rangle,
		\end{eqnarray*}
		while
		\begin{eqnarray*}
			\int_{\bbr^N}\widetilde{w}_{t-1,2}^2\partial_{\theta_{1}}(\widetilde{w}_{t,1}\widetilde{w}_{t-1,1})dx
			\sim\widehat{\eta}^{1-N}e^{-\frac{\sqrt{\lambda_1}}{\alpha}\widehat{\eta}}
			\langle\frac{\xi_{t-1,2}-\xi_{t-1,1}}{|\xi_{t-1,2}-\xi_{t-1,1}|}, \rho_1\xi_{t-1,1}^{\perp}\rangle
		\end{eqnarray*}
		and
		\begin{eqnarray*}
			\int_{\bbr^N}\widetilde{w}_{t-1,2}^2\partial_{\rho_1}(\widetilde{w}_{t,1}\widetilde{w}_{t-1,1})dx
			\sim\widehat{\eta}^{1-N}e^{-\frac{\sqrt{\lambda_1}}{\alpha}\widehat{\eta}}
			\langle\frac{\xi_{t-1,2}-\xi_{t-1,1}}{|\xi_{t-1,2}-\xi_{t-1,1}|}, \xi_{t-1,1}\rangle.
		\end{eqnarray*}
		It follows that in the case of $d=2$,
		\begin{eqnarray*}
			\int_{\bbr^N}\widetilde{w}_{t,1}^2\partial_{\theta_{2}}(\widetilde{w}_{t,2}\widetilde{w}_{t-1,2})dx,\quad
			\int_{\bbr^N}\widetilde{w}_{t,1}^2\partial_{\rho_{2}}(\widetilde{w}_{t,2}\widetilde{w}_{t-1,2})dx,
		\end{eqnarray*}
		and
		\begin{eqnarray*}
			\int_{\bbr^N}\widetilde{w}_{t-1,2}^2\partial_{\theta_{1}}(\widetilde{w}_{t,1}\widetilde{w}_{t-1,1})dx,\quad
			\int_{\bbr^N}\widetilde{w}_{t-1,2}^2\partial_{\rho_{1}}(\widetilde{w}_{t,1}\widetilde{w}_{t-1,1})dx
		\end{eqnarray*}
		are also h.o.t. for $\lambda_1=\lambda_2$, while
		\begin{eqnarray*}
			\int_{\bbr^N}\widetilde{w}_{t,1}^2\partial_{\theta_{2}}(\widetilde{w}_{t,2}\widetilde{w}_{t-1,2})dx
			\sim\int_{\bbr^N}\widetilde{w}_{t,1}^2\widetilde{w}_{t,2}\partial_{\theta_{2}}(\widetilde{w}_{t,2})dx,\\
			\int_{\bbr^N}\widetilde{w}_{t,1}^2\partial_{\rho_{2}}(\widetilde{w}_{t,2}\widetilde{w}_{t-1,2})dx
			\sim\int_{\bbr^N}\widetilde{w}_{t,1}^2\widetilde{w}_{t,2}\partial_{\rho_{2}}(\widetilde{w}_{t,2})dx,
		\end{eqnarray*}
		and
		\begin{eqnarray*}
			\int_{\bbr^N}\widetilde{w}_{t-1,2}^2\partial_{\theta_{1}}(\widetilde{w}_{t,1}\widetilde{w}_{t-1,1})dx
			\sim\int_{\bbr^N}\widetilde{w}_{t-1,2}^2\widetilde{w}_{t,1}\partial_{\theta_{1}}(\widetilde{w}_{t,1})dx,\\
			\int_{\bbr^N}\widetilde{w}_{t-1,2}^2\partial_{\rho_{1}}(\widetilde{w}_{t,1}\widetilde{w}_{t-1,1})dx
			\sim\int_{\bbr^N}\widetilde{w}_{t-1,2}^2\widetilde{w}_{t,1}\partial_{\rho_{1}}(\widetilde{w}_{t,1})dx
		\end{eqnarray*}
		for $\lambda_1<\lambda_2$.  The conclusion then follows from the above computations.
	\end{proof}

	\section{The reduced problem}
	So far, by \eqref{eqnew9988} and \eqref{eq9003}, we have proved that $\mathbf{U}=\mathbf{W}+\mathbf{Q}_*+\mathbf{v}_{**}$
	solves the following equation:
	\begin{equation}\label{eqnn9997}
		\left\{\aligned&-\Delta U_j+V_j(x)U_j=\mu_jU_j^3+\sum_{i=1,i\not=j}^d\beta_{i,j} U_i^2U_j-\gamma_{\theta_{j}}^*\partial_{\theta_{j}}W_j-\gamma_{\rho_j}^*\partial_{\rho_j}W_j\quad\text{in }\bbr^N,\\
		&\int_{\bbr^N}\partial_{x_l}\widetilde{w}_{t,j}Q_{j,*}dx=\int_{\bbr^N}\partial_{x_l}\widetilde{w}_{t,j}v_{j,**}dx=0, \quad j=1,2,\cdots,d; l=1,2,\cdots,N; t=1,2,\cdots,\vartheta,\endaligned\right.
	\end{equation}
	where $\mathbf{U}=(U_1,U_2,\cdots,U_d)$ with $U_j=W_j+Q_{j,*}+v_{j,**}$, and
	\begin{eqnarray*}
		\gamma_{\theta_{j}}^*=\frac{\|\partial_{\rho_j}W_j\|_{L^2(\bbr^N)}^2\int_{\bbr^N}(E_{j}^{***}+N_j^{**})\partial_{\theta_{j}}W_jdx
			-\int_{\bbr^N}\partial_{\theta_{j}}W_j\partial_{\rho_j}W_jdx
			\int_{\bbr^N}(E_{j}^{***}+N_j^{**})\partial_{\rho_j}W_jdx}{\|\partial_{\theta_{j}}W_j\|_{L^2(\bbr^N)}^2\|\partial_{\rho_j}W_j\|_{L^2(\bbr^N)}^2-
			(\int_{\bbr^N}\partial_{\theta_{j}}W_j\partial_{\rho_j}W_jdx)^2}
	\end{eqnarray*}
	and
	\begin{eqnarray*}
		\gamma_{\rho_j}^*=\frac{\|\partial_{\theta_{j}}W_j\|_{L^2(\bbr^N)}^2\int_{\bbr^N}(E_{j}^{***}+N_j^{**})\partial_{\rho_j}W_jdx
			-\int_{\bbr^N}\partial_{\theta_{j}}W_j\partial_{\rho_j}W_jdx
			\int_{\bbr^N}(E_{j}^{***}+N_j^{**})\partial_{\theta_{j}}W_jdx}{\|\partial_{\theta_{j}}W_j\|_{L^2(\bbr^N)}^2\|\partial_{\rho_j}W_j\|_{L^2(\bbr^N)}^2-
			(\int_{\bbr^N}\partial_{\theta_{j}}W_j\partial_{\rho_j}W_jdx)^2}
	\end{eqnarray*}
	with $E_{j}^{***}$ and $N_j^{**}$ being given by \eqref{eqnn9999} and \eqref{eqnn9998}, respectively.
	\begin{lemma}\label{lem0005}
		Suppose the assumptions~$(V_1)$--$(V_2)$ hold and $\beta_{i,j}$ is not an eigenvalue of $-\Delta+\lambda_j$ in $L^2(\bbr^N; w_i^2)$, that is,
		\begin{eqnarray*}
			-\Delta v+\lambda_j v=\beta_{i,j}w_i^2v \quad\text{in }\bbr^N
		\end{eqnarray*}
		has no solutions in $H^2(\bbr^N)$, then
		for $\vartheta$ and $\rho\overline{\alpha}=\rho\min\limits_{1\le j\le d}\alpha_j$ sufficiently large, $\gamma_{\rho_l}^*=0$ and $\gamma_{\alpha_{l-1}}^*=0$ for all $l=1,2,\cdots,d$ if and only if $\nabla\mathcal{J}(\overrightarrow{\rho},\overrightarrow{\alpha})=\bf{0}$, where
		\begin{eqnarray*}
			\mathcal{J}(\overrightarrow{\rho},\overrightarrow{\alpha})=\mathcal{E}(\mathbf{U}),
		\end{eqnarray*}
		with $\overrightarrow{\rho}=(\rho_1,\rho_2,\cdots,\rho_d)$, $\overrightarrow{\alpha}=(\alpha_1,\alpha_2,\cdots,\alpha_{d})$ and
		\begin{eqnarray*}
			\mathcal{E}(\mathbf{u})=\sum_{j=1}^d\frac{1}{2}\int_{\bbr^N}|\nabla u_j|^2+V_j(x)u_j^2dx-\frac{\mu_j}{4}\int_{\bbr^N}u_j^4dx-\sum_{i=1;i\not=j}^{d}\frac{\beta_{i,j}}{4}\int_{\bbr^N}u_i^2u_j^2dx
		\end{eqnarray*}
		being the energy functional of \eqref{eq0001}.
	\end{lemma}
	\begin{proof}
		Since $\mathbf{Q}_*$ and $\mathbf{v}_{**}$ are obtained by applying Proposition~\ref{prop0001} (the linear theory) through the fix-point argument, it is standard to show that $\mathbf{Q}_*$ and $\mathbf{v}_{**}$ are Lipschitz in terms of the parameters $\overrightarrow{\rho}$ and $\overrightarrow{\alpha}$.  Thus, $\mathcal{J}(\overrightarrow{\rho},\overrightarrow{\alpha})$ is of class $C^1$.  Moreover, by \eqref{eqnn9997},
		\begin{eqnarray*}
			\partial_{\rho_l}\mathcal{J}(\overrightarrow{\rho},\overrightarrow{\alpha})&=&\sum_{j=1}^{d}(\gamma_{\theta_{j}}^*\int_{\bbr^N}\partial_{\theta_{j}}W_j\partial_{\rho_l}U_jdx+
			\gamma_{\rho_j}^*\int_{\bbr^N}\partial_{\rho_{j}}W_j\partial_{\rho_l}U_jdx)\\
			&=&\sum_{j=1}^{d}(\gamma_{\theta_{j}}^*\int_{\bbr^N}\partial_{\theta_{j}}W_j\partial_{\rho_l}W_jdx+
			\gamma_{\rho_j}^*\int_{\bbr^N}\partial_{\rho_{j}}W_j\partial_{\rho_l}W_jdx)\\
			&&+\sum_{j=1}^{d}(\gamma_{\theta_{j}}^*\int_{\bbr^N}\partial_{\theta_{j}}W_j\partial_{\rho_l}(Q_{j,*}+v_{j,**})dx+
			\gamma_{\rho_j}^*\int_{\bbr^N}\partial_{\rho_{j}}W_j\partial_{\rho_l}(Q_{j,*}+v_{j,**})dx)
		\end{eqnarray*}
		and
		\begin{eqnarray*}
			\partial_{\theta_{l}}\mathcal{J}(\overrightarrow{\rho},\overrightarrow{\alpha})&=&\sum_{j=1}^{d}(\gamma_{\theta_{j}}^*\int_{\bbr^N}\partial_{\theta_{j}}W_j
			\partial_{\theta_{l}}U_jdx+
			\gamma_{\rho_j}^*\int_{\bbr^N}\partial_{\rho_{j}}W_j\partial_{\theta_{l}}U_jdx)\\
			&=&\sum_{j=1}^{d}(\gamma_{\theta_{j}}^*\int_{\bbr^N}\partial_{\theta_{j}}W_j\partial_{\theta_{l}}W_jdx+
			\gamma_{\rho_j}^*\int_{\bbr^N}\partial_{\rho_{j}}W_j\partial_{\theta_{l}}W_jdx)\\
			&&+\sum_{j=1}^{d}(\gamma_{\theta_{j}}^*\int_{\bbr^N}\partial_{\theta_{j}}W_j\partial_{\theta_{l}}(Q_{j,*}+v_{j,**})dx+
			\gamma_{\rho_j}^*\int_{\bbr^N}\partial_{\rho_{j}}W_j\partial_{\theta_{l}}(Q_{j,*}+v_{j,**})dx)
		\end{eqnarray*}
		for all $l=1,2,\cdots,d$.  By \eqref{eqnn0005} and \eqref{eqnn9996},
		\begin{eqnarray*}
			\sum_{j=1}^{d}(\gamma_{\theta_{j}}^*\int_{\bbr^N}\partial_{\theta_{j}}W_j\partial_{\rho_l}W_jdx+
			\gamma_{\rho_j}^*\int_{\bbr^N}\partial_{\rho_{j}}W_j\partial_{\rho_l}W_jdx)\sim\vartheta\gamma_{\rho_l}^*+o(\sum_{j\not=l}\vartheta\gamma_{\rho_j}^*)
			+(o(\sum_{j=1}^d\rho_j\vartheta\gamma_{\theta_{j}}^*))
		\end{eqnarray*}
		and
		\begin{eqnarray*}
			\sum_{j=1}^{d}(\gamma_{\theta_{j}}^*\int_{\bbr^N}\partial_{\theta_{j}}W_j\partial_{\theta_{l}}W_jdx+
			\gamma_{\rho_j}^*\int_{\bbr^N}\partial_{\rho_{j}}W_j\partial_{\theta_{l}}W_jdx)\sim\rho_l^2\vartheta(\gamma_{\theta_{l}}^*+o(\sum_{j=1;j\not=l}^{d}\gamma_{\theta_{j}}^*))
			+o(\vartheta\sum_{j=1;j\not=l}^d\rho_j\gamma_{\rho_j}^*)
		\end{eqnarray*}
		for all $l=1,2,\cdots,d$.  Moreover, by the orthogonal conditions of $Q_{j,*}$ and $v_{j,**}$, and \eqref{eqn2040} and \eqref{eqn2042},
		\begin{eqnarray*}
			&\int_{\bbr^N}\partial_{\theta_{j}}W_j\partial_{\theta_{l}}(Q_{j,*}+v_{j,**})dx
			=-\int_{\bbr^N}\partial_{\theta_{l}}\partial_{\theta_{j}}W_j(Q_{j,*}+v_{j,**})dx=o(\vartheta\rho^2),\\
			&\int_{\bbr^N}\partial_{\rho_{j}}W_j\partial_{\theta_{l}}(Q_{j,*}+v_{j,**})dx
			=-\int_{\bbr^N}\partial_{\theta_{l}}\partial_{\rho_{j}}W_j(Q_{j,*}+v_{j,**})dx=o(\vartheta\rho),\\
			&\int_{\bbr^N}\partial_{\theta_{j}}W_j\partial_{\rho_l}(Q_{j,*}+v_{j,**})dx
			=-\int_{\bbr^N}\partial_{\rho_l}\partial_{\theta_{j}}W_j(Q_{j,*}+v_{j,**})dx=o(\vartheta\rho),\\
			&\int_{\bbr^N}\partial_{\rho_{j}}W_j\partial_{\rho_l}(Q_{j,*}+v_{j,**})dx
			=-\int_{\bbr^N}\partial_{\rho_l}\partial_{\rho_{j}}W_j(Q_{j,*}+v_{j,**})dx=o(\vartheta)
		\end{eqnarray*}
		for all $j,l=1,2,\cdots,d$.  Thus, by $\rho_j=\rho+o(1)$ for all $j=1,2,\cdots,d$,
		\begin{eqnarray*}
			\rho_l\partial_{\rho_l}\mathcal{J}(\overrightarrow{\rho},\overrightarrow{\alpha})
			\sim\vartheta\rho_l\gamma_{\rho_l}^*+o(\sum_{j=1;j\not=l}^{d}\vartheta\rho_j\gamma_{\rho_j}^*)
			+(o(\sum_{j=1}^d\rho_j^2\vartheta\gamma_{\theta_{j}}^*))
		\end{eqnarray*}
		and
		\begin{eqnarray*}
			\partial_{\theta_{l}}\mathcal{J}(\overrightarrow{\rho},\overrightarrow{\alpha})
			\sim\vartheta\rho_l^2\gamma_{\theta_{l}}^*+o(\sum_{j=1;j\not=l}^{d}\vartheta\rho_j^2\gamma_{\theta_{j}}^*)
			+o(\vartheta\sum_{j=1}^d\rho_j\gamma_{\rho_j}^*)
		\end{eqnarray*}
		for all $l=1,2,\cdots,d$, which implies that $\gamma_{\rho_l}^*=0$ and $\gamma_{\theta_{l}}^*=0$ for all $l=1,2,\cdots,d$ if and only if $\nabla\mathcal{J}(\overrightarrow{\rho},\overrightarrow{\alpha})=\bf{0}$.
	\end{proof}
	
	Let $K_j(x)$ be the Green function of the operator $-\Delta+\lambda_j$, that is, $K_j$ solves the following equation:
	\begin{eqnarray*}
		\left\{\aligned&-\Delta K_j+\lambda K_j=\delta_0\quad\text{in }\bbr^N,\\
		&K_j(x)\to0\quad\text{as }|x|\to+\infty,\endaligned\right.
	\end{eqnarray*}
	where $\delta_0$ is the Dirac mass supported at $x=0$.  Then it is well known that $K_j(x)=\Gamma(x)-H_j(x)$, where $\Gamma(x)$ is the fundamental solution of $-\Delta$ and $H_j(x)$ is the regular part which is of class $C^1$ since $N=2,3$.  Moreover,
	\begin{eqnarray*}\label{eqn2052}
		K_j(x)\sim |x|^{-\frac{N-1}{2}}e^{-\sqrt{\lambda_j}|x|}\quad\text{as }|x|\to+\infty.
	\end{eqnarray*}
	By the representation formula,
	\begin{eqnarray*}
		Q_{j,*}(x)=\int_{\bbr^N}K_j(x-y)\widetilde{Q}_{j,*}(y)dy
	\end{eqnarray*}
	where
	\begin{eqnarray*}
		\widetilde{Q}_{j,*}(y)&=&(\lambda_j-V_j(y))Q_{j,*}(y)+3\mu_jW_j(y)^2Q_{j,*}(y)+\sum_{i=1;i\not=j}^{d}\beta_{i,j}W_i(y)^2Q_{j,*}(y)\\
		&&+2\sum_{i=1;i\not=j}^{d}\beta_{i,j}W_i(y)W_j(y)Q_{i,*}(y)+E_{j,3}(y)
		-\gamma_{\theta_{j}}\partial_{\theta_{j}}W_j(y)-\gamma_{\rho_j}\partial_{\rho_j}W_j(y).
	\end{eqnarray*}
	\begin{lemma}\label{lem0006}
		Suppose the assumptions~$(V_1)$--$(V_2)$ hold and $\beta_{i,j}$ is not an eigenvalue of $-\Delta+\lambda_j$ in $L^2(\bbr^N; w_i^2)$, that is,
		\begin{eqnarray*}
			-\Delta v+\lambda_j v=\beta_{i,j}w_i^2v \quad\text{in }\bbr^N
		\end{eqnarray*}
		has no solutions in $H^2(\bbr^N)$.  If $\lambda_{n_k}<4\lambda_{n_1}$ under the condition~\eqref{eq0002}, $\alpha_{n_{\tau+1}}<\alpha_{n_{\tau}}$ for all $\tau=1,2,\cdots,k-2$ and $\alpha_{n_k}=(1+o(1))\alpha_{n_1}$ with
		\begin{eqnarray*}
			\max\{\max\{\sqrt{\lambda_{n_\tau}}\alpha_{n_{\tau}}\}, \sqrt{\lambda_{d}}\alpha_{d-1}\}\leq2\min\{\min\{\sqrt{\lambda_{n_\tau'}}\alpha_{n_{\tau'}}\},\sqrt{\lambda_{d}}\alpha_{d-1}\}
		\end{eqnarray*}
		in the case of $k>1$
		and $\alpha_{j}=(1+o(1))\alpha_{n_{\tau}}$ for all $j=n_{\tau-1}+1,n_{\tau-1}+2,\cdots,n_{\tau}$,
		then
		\begin{eqnarray*}
			|Q_{j,*}|\lesssim
			(\rho_j\alpha_j)^{\frac{1-N}{2}}e^{-\sqrt{\lambda_j}\rho_j\alpha_j}\sum_{\tau,i,m}(1+r_{m,i})^{\frac{1-N}{2}}e^{-\textcolor{red}{\varepsilon}\sqrt{\lambda_{n_\tau}} r_{m,i}}
		\end{eqnarray*}
		in $\bbr^N$ with $\|Q_{j,*}\|_{L^\infty(\bbr^N)}\sim(\rho_j\alpha_j)^{\frac{1-N}{2}}e^{-\sqrt{\lambda_j}\rho_j\alpha_j}$ for $j=1,2,\cdots,d-1$, and
		\begin{eqnarray*}
			|Q_{d,*}|\lesssim
			(\rho_d\alpha_{d-1})^{\frac{1-N}{2}}e^{-\sqrt{\lambda_d}\rho_d\alpha_{d-1}}\sum_{\tau,i,m}(1+r_{m,i})^{\frac{1-N}{2}}e^{-\textcolor{red}{\varepsilon}\sqrt{\lambda_{n_\tau}} r_{m,i}}
		\end{eqnarray*}
		in $\bbr^N$ with $\|Q_{d,*}\|_{L^\infty(\bbr^N)}\sim(\rho_d\alpha_{d-1})^{\frac{1-N}{2}}e^{-\sqrt{\lambda_d}\rho_d\alpha_{d-1}}$.
	\end{lemma}
	\begin{proof}
		By \eqref{eq3021} and \eqref{eq0010}, $\alpha_{n_{\tau}+1}<\alpha_{n_{\tau}}$ for all $\tau=1,2,\cdots,k-2$, $\alpha_{n_k}=(1+o(1))\alpha_{n_1}$ and $\alpha_{j}=(1+o(1))\alpha_{n_{\tau}}$ for all $j=n_{\tau-1}+1,n_{\tau-1}+2,\cdots,n_{\tau}$. Thus by similar estimates of \eqref{eq3003},
		\begin{eqnarray*}
			|E_{j,3}|\lesssim
			(\rho_j\alpha_j)^{\frac{1-N}{2}}e^{-\sqrt{\lambda_j}\rho_j\alpha_j}\sum_{\tau,i,m}(1+r_{m,i})^{\frac{1-N}{2}}e^{-\varepsilon\sqrt{\lambda_{n_\tau}} r_{m,i}}
		\end{eqnarray*}
		and
		\begin{eqnarray}\label{eqnnnew0010}
			\|E_{j,3}\|_{L^\infty(\bbr^N)}\sim
			(\rho_j\alpha_j)^{\frac{1-N}{2}}e^{-\sqrt{\lambda_j}\alpha_j\rho_j}
		\end{eqnarray}
		for all $j=1,2,\cdots,d-1$, and
		\begin{eqnarray*}
			|E_{d,3}|\lesssim
			(\rho_d\alpha_{d-1})^{\frac{1-N}{2}}e^{-\sqrt{\lambda_d}\rho_d\alpha_{d-1}}\sum_{i,m}(1+r_{m,i})^{\frac{1-N}{2}}e^{-\varepsilon\sqrt{\lambda_i} r_{m,i}}
		\end{eqnarray*}
		and
		\begin{eqnarray}\label{eqnnnew0011}
			\|E_{d,3}\|_{L^\infty(\bbr^N)}\sim
			(\rho_d\alpha_{d-1})^{\frac{1-N}{2}}e^{-\sqrt{\lambda_d}\rho_d\alpha_{d-1}}.
		\end{eqnarray}
		Note that by \eqref{eq0010} and similar estimates of \eqref{eq3003},
		\begin{eqnarray*}
			W_iW_j\lesssim\left\{\aligned&(\rho_i\alpha_i)^{\frac{1-N}{2}}e^{-\sqrt{\lambda_i}\rho_i\alpha_i}\quad \text{for }i<j,\\
			&(\rho_j\alpha_j)^{\frac{1-N}{2}}e^{-\sqrt{\lambda_j}\rho_j\alpha_j}\quad\text{for }i>j,
			\endaligned\right.
		\end{eqnarray*}
		thus,
		by $\lambda_{n_k}<4\lambda_{n_1}$, $\alpha_{n_{\tau+1}}<\alpha_{n_{\tau}}$ for all $\tau=1,2,\cdots,k-2$ and $\alpha_{n_k}=(1+o(1))\alpha_{n_1}$ with
		\begin{eqnarray*}
			\max\{\max\{\sqrt{\lambda_{n_\tau}}\alpha_{n_{\tau}}\}, \sqrt{\lambda_{d}}\alpha_{d-1}\}\leq2\min\{\min\{\sqrt{\lambda_{n_\tau'}}\alpha_{n_{\tau'}}\},\sqrt{\lambda_{d}}\alpha_{d-1}\}
		\end{eqnarray*}
		in the case of $k>1$, we can apply the maximum principle and blow-up arguments similarly as that in the proof of Proposition~\ref{prop0001}, we have
		\begin{eqnarray}\label{eqnn9990}
			|Q_{j,*}|\lesssim(\rho_j\alpha_j)^{\frac{1-N}{2}}e^{-\sqrt{\lambda_j}\alpha_j}\sum_{\tau,i,m}(1+r_{m,i})^{\frac{1-N}{2}}e^{-\varepsilon\sqrt{\lambda_{n_\tau}} r_{m,i}}
		\end{eqnarray}
		in $\bbr^N$ for $j=1,2,\cdots,d-1$, and
		\begin{eqnarray}\label{eqnnnew9990}
			|Q_{d,*}|\lesssim(\rho_d\alpha_{d-1})^{\frac{1-N}{2}}e^{-\sqrt{\lambda_d}\rho_d\alpha_{d-1}}\sum_{\tau,i,m}(1+r_{m,i})^{\frac{1-N}{2}}e^{-\varepsilon\sqrt{\lambda_{n_\tau}} r_{m,i}}
		\end{eqnarray}
		in $\bbr^N$.  Since by \eqref{eqn2038}, \eqref{eqnnnew2038}, \eqref{eqnn0005}, \eqref{eqnn9996} and Lemma~\ref{lem0004},
		\begin{eqnarray}\label{eqnnnew1111}
			\gamma_{\theta_{j}}\sim\frac{\int_{\bbr^N}E_{j,3}\partial_{\theta_{j}}W_jdx}{\rho^2\vartheta}\quad\text{and}\quad
			\gamma_{\rho_{j}}\sim\frac{\int_{\bbr^N}E_{j,3}\partial_{\rho_{j}}W_jdx}{\vartheta}
		\end{eqnarray}
		for all $j$, by Lemma~\ref{lem0004} once more,
		\begin{eqnarray*}
			\int_{\bbr^N\backslash(\cup_{t,l}B_{\delta\widehat{\eta}}(\eta_{t,l}))}K_j(x-y)\widetilde{Q}_{j,*}(y)dy=o(\|E_{j,3}\|_{L^\infty(\bbr^N)}).
		\end{eqnarray*}
		For $\int_{\cup_{t,l}B_{\delta\widehat{\eta}}(\eta_{t,l})}K(x-y)\widetilde{Q}_{j,*}(y)dy$,
		if $x\in\bbr^N\backslash(\cup_{t,l}B_{\delta\widehat{\eta}}(\eta_{t,l}))$, then by Lemma~\ref{lem0004}, $\lambda_{n_k}<4\lambda_{n_1}$ and \eqref{eqnnnew1111} once more,
		\begin{eqnarray*}
			&&\int_{\cup_{t,l}B_{\delta\widehat{\eta}}(\eta_{t,l})}K(x-y)\widetilde{Q}_{j,**}(y)dy\\
			&=&o\bigg(\|E_{j,3}\|_{L^\infty(\bbr^N)}(\int_{\cup_{t,l}B_{\frac{\delta}{2}\widehat{\eta}}(\eta_{t,l})}K_j^{1-\sigma'}(x-y)dy
			+\int_{\cup_{t,l}(B_{\delta\widehat{\eta}}(\eta_{t,l})\backslash B_{\frac{\delta}{2}\widehat{\eta}}(\eta_{t,l}))}K_j(x-y)dy)\bigg)\\
			&=&o(\|E_{j,3}\|_{L^\infty(\bbr^N)}).
		\end{eqnarray*}
		It remains to consider $\int_{\cup_{t,l}B_{\delta\widehat{\eta}}(\eta_{t,l})}K_j(x-y)\widetilde{Q}_{j,*}(y)dy$ in the cases of $x\in B_{\delta\widehat{\eta}}(\eta_{s,p})$ for some $s$ and $p$.  In these cases, by similar arguments as used above,
		\begin{eqnarray*}\label{eqnnnew1112}
			&&\int_{\cup_{t,l}B_{\delta\widehat{\eta}}(\eta_{t,l})}K(x-y)\widetilde{Q}_{j,*}(y)dy\\
			&=&\left\{\aligned&\int_{B_{\delta\widehat{\eta}}(\eta_{t,j})}K_j(x-y)(3\mu_j\widetilde{w}_{t,j}(y)^2Q_{j,*}(y)+E_{j,3}(y))dy
			+o(\|E_{j,3}\|_{L^\infty(\bbr^N)}),\quad p=j,\notag\\
			&\int_{B_{\delta\widehat{\eta}}(\eta_{t,i})}K_j(x-y)(\beta_{i,j}\widetilde{w}_{t,i}(y)^2Q_{j,*}(y)+E_{j,3}(y))dy+o(\|E_{j,3}\|_{L^\infty(\bbr^N)}), \quad p=i.\endaligned\right.
		\end{eqnarray*}
		Moreover, it is easy to see that
		\begin{eqnarray*}
			E_{j,3}(y)\sim((\rho_j\alpha_j)^{1-N}e^{-2\sqrt{\lambda_{j+1}}\rho_j\alpha_j}+
			(\rho_j\alpha_{j-1})^{1-N}e^{-2\sqrt{\lambda_{j-1}}\rho_j\alpha_{j-1}})\widetilde{w}_{t,j}(y)
		\end{eqnarray*}
		in $B_{R}(\eta_{t,j})$ for all $t$ and
		\begin{eqnarray*}
			E_{j,3}(y)\sim|\eta_{t,j}-\eta_{t,i}|^{\frac{1-N}{2}}e^{-\sqrt{\lambda_{j}}|\eta_{t,j}-\eta_{t,i}|}\widetilde{w}^2_{t,i}(y)
		\end{eqnarray*}
		in $B_{R}(\eta_{t,i})$ for all $i\not=j$ and all $t$ with any $R>0$ sufficiently large.  Let
		\begin{eqnarray*}
			Q_{j,**,t,l}=\frac{Q_{j,*}(\cdot+\eta_{t,l})}{(\rho_j\alpha_j)^{\frac{1-N}{2}}e^{-\sqrt{\lambda_{j}}\rho_j\alpha_j}}
		\end{eqnarray*}
		for all $l$, $t$ and $j=1,2,\cdots,d-1$, and
		\begin{eqnarray*}
			Q_{d,**,t,l}=\frac{Q_{d,*}(\cdot+\eta_{t,l})}{(\rho_d\alpha_d)^{\frac{1-N}{2}}e^{-\sqrt{\lambda_{d}}\rho_d\alpha_{d-1}}}
		\end{eqnarray*}
		for all $l,t$.
		We recall that $\mathbf{Q}_*$ is the solution of \eqref{eqnew9988}.  Then by Lemma~\ref{lem0004}, $\lambda_{n_k}<4\lambda_{n_1}$ and \eqref{eqnnnew1111},
		it is easy to show that $Q_{j,**,t,l}\to 0$ uniformly on every compact subset of $\bbr^N$ for all $l\not=j-1$ and $j+1$.  For $Q_{j,**,t,j\pm1}$, either $Q_{j,**,t,j\pm1}\to 0$ or $Q_{j,**,t,j\pm1}\to \phi_{j,t,j\pm1}$ uniformly on every compact set of $\bbr^N$,
		where
		\begin{equation}\label{eqnnnew1122}
			-\Delta\phi_{j,t,j\pm1}+\lambda_j\phi_{j,t,j\pm1}-\beta_{i,j}w_i^2\phi_{j,t,j\pm1}=\beta_{i,j}w_i^2,\quad\text{in }\bbr^N
		\end{equation}
		Note that $\beta_{i,j}$ is not an eigenvalue of $-\Delta+\lambda_j$ in $L^2(\bbr^N; w_i^2)$, thus, $\phi_{j,t,j\pm1}$ are unique.  It follows that
		\begin{eqnarray*}
			&&\int_{\cup_{t,l}B_{\delta\widehat{\eta}}(\eta_{t,l})}K(x-y)\widetilde{Q}_{j,*}(y)dy\\
			&\sim&\left\{\aligned&O(e^{-\delta'R}\|E_{j,3}\|_{L^\infty(\bbr^N)}),\quad l\not=j\pm1,\\
			&\|E_{j,3}\|_{L^\infty(\bbr^N)}\int_{B_{R}(\eta_{t,j+1})}K_j(x-y)\widetilde{w}_{t,j+1}(y)^2(\phi_{j,t,j+1}+1)dy
			+O(e^{-\delta'R}\|E_{j,3}\|_{L^\infty(\bbr^N)}), \quad l=j+1,\\
			&\|E_{j,3}\|_{L^\infty(\bbr^N)}\int_{B_{R}(\eta_{t,j-1})}K_j(x-y)\widetilde{w}_{t,j-1}(y)^2(\phi_{j,t,j-1}+1)dy
			+O(e^{-\delta'R}\|E_{j,3}\|_{L^\infty(\bbr^N)}), \quad l=j-1,
			\endaligned\right.
		\end{eqnarray*}
		where $\delta'>0$ is sufficiently small and $R>0$ is sufficiently large.
		Now, summarizing the above estimates,
		\begin{eqnarray*}
			Q_{j,*}(x)&=&\left\{\aligned&o(\|E_{j,3}\|_{L^\infty(\bbr^N)}),\quad x\in\bbr^N\backslash(\cup_{t,l}B_{\delta\widehat{\eta}}(\eta_{t,l})),\\
			&o(\|E_{j,3}\|_{L^\infty(\bbr^N)}),\quad x\in B_{\delta\widehat{\eta}}(\eta_{t,l})\text{ with }l\not=j\pm1,\\
			&D_j(x)\|E_{j,3}\|_{L^\infty(\bbr^N)}+o(\|E_{j,3}\|_{L^\infty(\bbr^N)}),\quad x\in B_{\delta\widehat{\eta}}(\eta_{t,l})\text{ with }l=j\pm1,
			\endaligned\right.
		\end{eqnarray*}
		where
		\begin{eqnarray*}
			D_j(x)&\sim&\int_{B_{R}(\eta_{t,j+1})}K_j(x-y)\widetilde{w}_{t,j+1}(y)^2(\phi_{j,t,j+1}+1)dy\\
			&&+\int_{B_{R}(\eta_{t,j-1})}K_j(x-y)\widetilde{w}_{t,j-1}(y)^2(\phi_{j,t,j-1}+1)dy+O(e^{-\delta' R}),
		\end{eqnarray*}
		which, together with \eqref{eqnnnew0010}, \eqref{eqnnnew0011}, \eqref{eqnn9990} and \eqref{eqnnnew9990}, completes the proof.
	\end{proof}
	
	\begin{remark}\label{rmk0001}
		By Lemmas~\ref{lem0001}, \ref{lem0002}--\ref{lem0003}, \ref{lem0004} and \ref{lem0006}, we know that
		\begin{eqnarray*}
			\gamma_{\theta_{j}}^*\lesssim(1+o(1))\rho^{-1}(\rho^{-\nu_j}+\widehat{\eta}^{\frac{1-N}{2}}e^{-2\sqrt{\lambda_{n_1}}\widehat{\eta}})
		\end{eqnarray*}
		and
		\begin{eqnarray*}
			\gamma_{\rho_{j}}^*\lesssim(1+o(1))(\rho^{-\nu_j}+\widehat{\eta}^{\frac{1-N}{2}}e^{-2\sqrt{\lambda_{n_1}}\widehat{\eta}}).
		\end{eqnarray*}
		Thus, by \eqref{eqnn9997} and , we can go through the arguments in the proof of Lemma~\ref{lem0006} to show that
		\begin{eqnarray*}
			U_{j,**,t,l}\to\left\{\aligned&0,\quad l\not=j\pm1,\\
			&\phi_{j,t,j\pm1},\quad l=j\pm1 \endaligned\right.
		\end{eqnarray*}
		on every compact set of $\bbr^N$, where
		\begin{eqnarray*}
			U_{j,**,t,l}=\frac{U_{j}(\cdot+\eta_{t,l})}{(\rho_j\alpha_j)^{\frac{1-N}{2}}e^{-\sqrt{\lambda_{j}}\rho_j\alpha_j}}
		\end{eqnarray*}
		for all $l$, $t$ and $j=1,2,\cdots,d-1$,
		\begin{eqnarray*}
			U_{d,**,t,l}=\frac{U_{d}(\cdot+\eta_{t,l})}{(\rho_d\alpha_d)^{\frac{1-N}{2}}e^{-\sqrt{\lambda_{d}}\rho_d\alpha_{d-1}}}
		\end{eqnarray*}
		for all $l,t$ and $\phi_{j,t,j\pm1}$ is a solution of \eqref{eqnnnew1122}.  Moreover, it is easy to see that $\phi_{j,t,j\pm1}$ are positive if $\beta_{j,j\pm1}\leq\beta_{j,j\pm1,*}$ and $\phi_{j,t,j\pm1}$ are sign-changing if $\beta_{j,j\pm1}>\beta_{j,j\pm1,*}$, where $\beta_{j,j\pm1,*}$ are the first eigenvalue of $-\Delta+\lambda_j$ in $L^2(\bbr^N; w_{j\pm1}^2)$.  It follows that $\mathbf{U}=\mathbf{W}+\mathbf{Q}_*+\mathbf{v}_{**}$ can not be a solution of \eqref{eq0001} if $\beta_{j,j\pm1}>\beta_{j,j\pm1,*}$ for some $j$.
	\end{remark}

	Let $m_j$ be the energy of $w_j$.  Then by \eqref{eqnew9988}, \eqref{eq9003}, the symmetry of the construction of $\{\eta_{t,j}\}$ and Lemmas~\ref{lem0002}--\ref{lem0003} and \ref{lem0006},
	\begin{eqnarray*}
		\mathcal{J}(\overrightarrow{\rho},\overrightarrow{\alpha})&=&\sum_{j=1}^{d}(\vartheta m_j+\frac{1}{2}\int_{\bbr^N}(V_j(x)-\lambda_j)W_j^2dx-\mu_j(\sum_{s<t}\int_{\bbr^N}\widetilde{w}_{t,j}^3\widetilde{w}_{s,j}dx\notag\\
		&&-\frac{3}{2}\sum_{t,l,s;s\not=t,l\not=t}\int_{\bbr^N}\widetilde{w}_{t,j}^2\widetilde{w}_{s,j}\widetilde{w}_{l,j}dx
		-\sum_{t,l,s,p;s\not=t,l\not=t,p\not=t,s\not=l,l\not=p,s\not=p}\int_{\bbr^N}\widetilde{w}_{t,j}\widetilde{w}_{s,j}\widetilde{w}_{l,j}\widetilde{w}_{p,j}dx)\notag\\
		&&-\frac12\sum_{i,j;i\not=j}\beta_{i,j}\int_{\bbr^N}W_i^2W_j^2dx-\frac{1}{2}\sum_{j=1}^{d}(\int_{\bbr^N}E_{j}(Q_{j,*}+v_{j,**})dx\notag\\
		&&+\gamma_{\theta_{j}}^*\int_{\bbr^N}\partial_{\theta_{j}}W_j(Q_{j,*}+v_{j,**})dx+\gamma_{\rho_{j}^*}\int_{\bbr^N}\partial_{\rho_{j}}W_j(Q_{j,*}+v_{j,**})dx)\notag\\
		&&+O(\vartheta(\sum_{j=1}^{d}(\|Q_{j,*}\|_{L^{\infty}(\bbr^N)}+\|v_{j,**}\|_{L^{\infty}(\bbr^N)}))(\sum_{i=1}^{d}(\|Q_{i,*}\|_{L^{\infty}(\bbr^N)}^2+\|v_{i,**}\|_{L^{\infty}(\bbr^N)}^2)))\notag.
	\end{eqnarray*}
	\begin{proposition}\label{prop0002}
		Suppose the assumptions~$(V_1)$--$(V_2)$ hold and $\beta_{i,j}$ is not an eigenvalue of $-\Delta+\lambda_j$ in $L^2(\bbr^N; w_i^2)$, that is,
		\begin{eqnarray*}
			-\Delta v+\lambda_j v=\beta_{i,j}w_i^2v \quad\text{in }\bbr^N
		\end{eqnarray*}
		has no solutions in $H^2(\bbr^N)$ with $\beta_{j,j+1}<\beta_{j,j+1,*}$ for all $j$ where $\beta_{j,j+1,*}$ are the first eigenvalue of $-\Delta+\lambda_j$ in $L^2(\bbr^N; w_{j+1}^2)$.  If $N=2,3$, $\lambda_{n_k}<4\lambda_{n_1}$ under the condition~\eqref{eq0002}, $\alpha_{n_{\tau+1}}<\alpha_{n_{\tau}}$ for all $\tau=1,2,\cdots,k-2$ and $\alpha_{n_k}=(1+o(1))\alpha_{n_1}$ with
		\begin{eqnarray*}
			\max\{\max\{\sqrt{\lambda_{n_\tau}}\alpha_{n_{\tau}}\}, \sqrt{\lambda_{d}}\alpha_{d-1}\}\leq2\min\{\min\{\sqrt{\lambda_{n_\tau'}}\alpha_{n_{\tau'}}\},\sqrt{\lambda_{d}}\alpha_{d-1}\}
		\end{eqnarray*}
		in the case of $k>1$
		and $\alpha_{j}=(1+o(1))\alpha_{n_{\tau}}$ for all $j=n_{\tau-1}+1,n_{\tau-1}+2,\cdots,n_{\tau}$,
		then $\mathcal{J}(\overrightarrow{\rho},\overrightarrow{\alpha})$ has a critical point for $\min_{j}\nu_j>1$, provided
		\begin{enumerate}
			\item[$(a)$]\quad $\sum_{j\in\mathfrak{m}_{*}}B_j\delta_j>0$ and $\sum_{j=n_{\tau-1}+1}^{n_{\tau}-1}\beta_{j,j+1}>0$, $\beta_{n_{\tau},n_{\tau}+1}>0$ for all $\tau=1,2,\cdots,k-1$,
			\item[$(b)$]\quad $\sum_{j\in\mathfrak{m}_{*}}B_j\delta_j<0$ and $\sum_{j=n_{\tau-1}+1}^{n_{\tau}-1}\beta_{j,j+1}<0$, $\beta_{n_{\tau},n_{\tau}+1}<0$ for all $\tau=1,2,\cdots,k-1$ in the case of $d\geq3$,
			\item[$(c)$]\quad $\sum_{j\in\mathfrak{m}_{*}}B_j\delta_j>0$, and $-2\pi^{-\frac12}D_1^{-1}C_1<\beta_{1,2}<0$ and $\lambda_1=\lambda_2$ in the case of $N=2$ and $d=2$ while, $\beta_{1,2}<0$ in the cases of $N=3$ and $d=2$ or $N=2$, $d=2$ with $\lambda_1\not=\lambda_2$,
			\item[$(d)$]\quad $\sum_{j\in\mathfrak{m}_{*}}B_j\delta_j<0$, $\beta_{1,2}<-2\pi^{-\frac12}D_1^{-1}C_1<0$ and $\lambda_1=\lambda_2$ in the case of $N=2$ and $d=2$,
		\end{enumerate}
		where $B_j,C_j,D_{\tau}>0$ are given by \eqref{eqnnnew1127}, \eqref{eqnnnew1126}, and \eqref{eqnnnew1123}, respectively.
	\end{proposition}
	\begin{proof}
		Let
		\begin{eqnarray*}
			\mathcal{J}_*(\overrightarrow{\rho},\overrightarrow{\alpha})&=&\sum_{j=1}^{d}(\vartheta m_j+\frac{1}{2}\int_{\bbr^N}(V_j(x)-\lambda_j)W_j^2dx-\mu_j\sum_{s<t}\int_{\bbr^N}\widetilde{w}_{t,j}^3\widetilde{w}_{t,j}dx\\
			&&-\frac12\sum_{i=1;i\not=j}^{d}\beta_{i,j}\int_{\bbr^N}W_i^2W_j^2dx).
		\end{eqnarray*}
		Then by the assumption~$(V_2)$, the symmetry of the construction of $\{\eta_{t,j}\}$, the assumptions on $\{\alpha_j\}$, \cite[Proposition~A.2]{WY10} and \cite[Lemma~3.7]{ACR07},
		\begin{eqnarray*}
			\vartheta^{-1}\mathcal{J}_*(\overrightarrow{\rho},\overrightarrow{\alpha})&=&\sum_{j=1}^{d}(m_j+\frac{B_j\delta_j+o(1)}{\rho_j^{\nu_j}}-(C_j+o(1))\widehat{\eta}_{j}^{\frac{1-N}{2}}e^{-\sqrt{\lambda_{j}}\widehat{\eta}_{j}})\notag\\
			&&-\sum_{\tau=1}^{k}\sum_{j=n_{\tau-1}+1}^{n_{\tau}-1}D_{\tau}\beta_{j,j+1}(\rho_j\alpha_j)^{-\frac12}e^{-2\sqrt{\lambda_{n_{\tau}}}\rho_j\alpha_j}\\
			&&-\sum_{\tau=1}^{k-1}D_{\tau}'\beta_{n_{\tau},n_{\tau}+1}(\rho_{n_{\tau}}\alpha_{n_{\tau}})^{-1}e^{-2\sqrt{\lambda_{n_{\tau}}}\rho_{n_{\tau}}\alpha_{n_{\tau}}}\\
			&&-D_{k}'\beta_{d,1}(\rho_{n_{k}}\alpha_{n_{k}})^{-1}e^{-2\sqrt{\lambda_{n_{1}}}\rho_{n_{k}}\alpha_{n_{k}}}+h.o.t.
		\end{eqnarray*}
		for $N=2$ and
		\begin{eqnarray*}
			\vartheta^{-1}\mathcal{J}_*(\overrightarrow{\rho},\overrightarrow{\alpha})&=&\sum_{j=1}^{d}(m_j+\frac{B_j\delta_j+o(1)}{\rho_j^{\nu_j}}-(C_j+o(1))\widehat{\eta}_{j}^{\frac{1-N}{2}}e^{-\sqrt{\lambda_{j}}\widehat{\eta}_{j}})\notag\\
			&&-\sum_{\tau=1}^{k}\sum_{j=n_{\tau-1}+1}^{n_{\tau}-1}D_{\tau}\beta_{j,j+1}(\rho_j\alpha_j)^{-2}e^{-2\sqrt{\lambda_{n_{\tau}}}\rho_j\alpha_j}\log(\rho_j\alpha_j)\\
			&&-\sum_{\tau=1}^{k-1}D_{\tau}'\beta_{n_{\tau},n_{\tau}+1}(\rho_{n_{\tau}}\alpha_{n_{\tau}})^{-2}e^{-2\sqrt{\lambda_{n_{\tau}}}\rho_{n_{\tau}}\alpha_{n_{\tau}}}\\
			&&-D_{k}'\beta_{d,1}(\rho_{n_{k}}\alpha_{n_{k}})^{-2}e^{-2\sqrt{\lambda_{n_{1}}}\rho_{n_{k}}\alpha_{n_{k}}}+h.o.t.
		\end{eqnarray*}
		for $N=3$, where $B_j,C_j,D_{\tau},D_{\tau}'>0$ are given by \eqref{eqnnnew1127}, \eqref{eqnnnew1126}, \eqref{eqnnnew1123}, \eqref{eqnnnew1124} and \eqref{eqnnnew1125}, respectively.
		On the other hand, by \cite[Lemma~3.7]{ACR07},
		\begin{eqnarray*}
			\sum_{t,l,s;s\not=t,l\not=t}\int_{\bbr^N}\widetilde{w}_{t,j}^2\widetilde{w}_{s,j}\widetilde{w}_{l,j}dx=o(\widehat{\eta}_{j}^{\frac{1-N}{2}}e^{-\sqrt{\lambda_{j}}\widehat{\eta}_{j}})
		\end{eqnarray*}
		and
		\begin{eqnarray*}
			\sum_{t,l,s,p;s\not=t,l\not=t,p\not=t,s\not=l,l\not=p,s\not=p}\int_{\bbr^N}\widetilde{w}_{t,j}\widetilde{w}_{s,j}\widetilde{w}_{l,j}\widetilde{w}_{p,j}dx=o(\widehat{\eta}_{j}^{\frac{1-N}{2}}e^{-\sqrt{\lambda_{j}}\widehat{\eta}_{j}}).
		\end{eqnarray*}
		By $\lambda_{n_k}<4\lambda_{n_1}$, the symmetry of the construction of $\{\eta_{t,j}\}$, the assumptions on $\{\alpha_j\}$, Lemmas~\ref{lem0001}, \ref{lem0002}--\ref{lem0003}, \ref{lem0004} and \ref{lem0006}, we know that the terms including $Q_{j,*}$ and $v_{j,**}$ are all h.o.t. of $\mathcal{J}_*(\overrightarrow{\rho},\overrightarrow{\alpha})$, expect the terms $\int_{\bbr^N}E_{j,3}Q_{j,*}dx$, where by Lemma~\ref{lem0006},
		\begin{eqnarray*}
			\int_{\bbr^N}E_{j,3}Q_{j,*}dx\sim\beta_{j,j+1}(\rho_j\alpha_j)^{-2}e^{-2\sqrt{\lambda_{j}}\rho_j\alpha_j}\quad\text{for all }j.
		\end{eqnarray*}
		Now, by Remark~\ref{rmk0001}, if $\beta_{j,j+1}<\beta_{j,j+1,*}$ for all $j$, then
		\begin{eqnarray}
			\vartheta^{-1}\mathcal{J}(\overrightarrow{\rho},\overrightarrow{\alpha})&=&\sum_{j=1}^{d}(m_j+\frac{B_j\delta_j+o(1)}{\rho_j^{\nu_j}}-(C_j+o(1))
			\widehat{\eta}_{j}^{\frac{1-N}{2}}e^{-\sqrt{\lambda_{j}}\widehat{\eta}_{j}})\notag\\
			&&-\sum_{\tau=1}^{k}\sum_{j=n_{\tau-1}+1}^{n_{\tau}-1}D_{\tau}\beta_{j,j+1}(\rho_j\alpha_j)^{-\frac12}e^{-2\sqrt{\lambda_{n_{\tau}}}\rho_j\alpha_j}\notag\\
			&&-\sum_{\tau=1}^{k-1}D_{\tau}''\beta_{n_{\tau},n_{\tau}+1}(\rho_{n_{\tau}}\alpha_{n_{\tau}})^{-1}e^{-2\sqrt{\lambda_{n_{\tau}}}\rho_{n_{\tau}}\alpha_{n_{\tau}}}\notag\\
			&&-D_{k}''\beta_{d,1}(\rho_{n_{k}}\alpha_{n_{k}})^{-1}e^{-2\sqrt{\lambda_{n_{1}}}\rho_{n_{k}}\alpha_{n_{k}}}+h.o.t.
			\label{eqnn9989}
		\end{eqnarray}
		for $N=2$ and
		\begin{eqnarray}
			\vartheta^{-1}\mathcal{J}(\overrightarrow{\rho},\overrightarrow{\alpha})&=&\sum_{j=1}^{d}(m_j+\frac{B_j\delta_j+o(1)}{\rho_j^{\nu_j}}-(C_j+o(1))
			\widehat{\eta}_{j}^{\frac{1-N}{2}}e^{-\sqrt{\lambda_{j}}\widehat{\eta}_{j}})\notag\\
			&&-\sum_{\tau=1}^{k}\sum_{j=n_{\tau-1}+1}^{n_{\tau}-1}D_{\tau}\beta_{j,j+1}(\rho_j\alpha_j)^{-2}e^{-2\sqrt{\lambda_{n_{\tau}}}\rho_j\alpha_j}\log(\rho_j\alpha_j)\notag\\
			&&-\sum_{\tau=1}^{k-1}D_{\tau}''\beta_{n_{\tau},n_{\tau}+1}(\rho_{n_{\tau}}\alpha_{n_{\tau}})^{-2}e^{-2\sqrt{\lambda_{n_{\tau}}}\rho_{n_{\tau}}\alpha_{n_{\tau}}}\notag\\
			&&-D_{k}''\beta_{d,1}(\rho_{n_{k}}\alpha_{n_{k}})^{-2}e^{-2\sqrt{\lambda_{n_{1}}}\rho_{n_{k}}\alpha_{n_{k}}}+h.o.t.
			\label{eqnnn9989}
		\end{eqnarray}
		for $N=3$, where $D_{n_\tau}''>0$ for all $\tau$.  Recall that $\rho_j\sim\vartheta\log\vartheta$ and $\alpha_j\sim\frac{1}{\vartheta}$ with $\sum_{j=1}^d\alpha_j=\frac{2\pi}{\vartheta}$ for all $j$.  Thus, we assume that $\rho_j=\rho_j^*\vartheta\log\vartheta$ and $\alpha_j=\frac{\alpha_j^*}{\vartheta}$
		where $\rho_j^*>0$ and $\alpha_j^*\in(0, 2\pi)$ with $\sum_{j=1}^d\alpha_j^*=2\pi$.  By the assumptions of $\{\alpha_j\}$, $\alpha_{j}^*=\alpha_{n_{\tau}}^*+o(1)$ for all $j=n_{\tau-1}+1,\cdots, n_{\tau}$ and
		\begin{eqnarray}\label{eqnnnew1116}
			\sum_{\tau=2}^{k-1}(n_{\tau}-n_{\tau-1})\alpha_{n_\tau}^*+(n_{1}+n_{k}-n_{k-1})\alpha_{n_1}^*=2\pi+o(1).
		\end{eqnarray}
		Now, intersecting these into \eqref{eqnn9989} and \eqref{eqnnn9989} and noting that $\rho_j=\rho+O(1)$, we have
		\begin{eqnarray*}
			\mathcal{J}(\overrightarrow{\rho},\overrightarrow{\alpha})
			&=&\vartheta\sum_{j=1}^{d}\bigg( m_j+\frac{B_j\delta_j+o(1)}{(\rho^*\vartheta\log\vartheta)^{\nu_j}}-\frac{C_j+o(1)}{(\rho^*\log\vartheta)^{\frac{1}{2}}\vartheta^{2\pi\sqrt{\lambda_{j}}\rho^*}}\bigg)\\
			&&-\sum_{\tau=2}^{k-1}\bigg(\frac{D_{\tau}\sum_{j=n_{\tau-1}+1}^{n_{\tau}-1}\beta_{j,j+1}\vartheta}
			{(\rho^*\alpha_{n_{\tau}}^*\log\vartheta)^{\frac12}\vartheta^{2\sqrt{\lambda_{n_{\tau}}}\rho^*\alpha_{n_{\tau}}^*}}+\frac{D_{\tau}''\beta_{n_{\tau},n_{\tau}+1}\vartheta}
			{(\rho^*\alpha_{n_{\tau}}^*\log\vartheta)\vartheta^{2\sqrt{\lambda_{n_{\tau}}}\rho^*\alpha_{n_{\tau}}^*}}\bigg)\\
			&&-\bigg(\frac{D_{1}\sum_{j=1}^{n_{1}-1}\beta_{j,j+1}\vartheta(n_{1}+n_{k}-n_{k-1})^{\frac12}}
			{(\rho^*(2\pi-\sum_{\tau=2}^{k-1}\alpha_{n_{\tau}}^*)\log\vartheta)^{\frac12}\vartheta^{2\sqrt{\lambda_{n_{1}}}\rho^*\frac{
						2\pi-\sum_{\tau=2}^{k-1}(n_{\tau}-n_{\tau-1})\alpha_{n_\tau}^*}{n_{1}+n_{k}-n_{k-1}}}}\\
			&&+\frac{D_{1}''\beta_{n_{1},n_{1}+1}\vartheta(n_{1}+n_{k}-n_{k-1})}
			{(\rho^*(2\pi-\sum_{\tau=2}^{k-1}\alpha_{n_{\tau}}^*)\log\vartheta)\vartheta^{2\sqrt{\lambda_{n_{1}}}\rho^*\frac{
						2\pi-\sum_{\tau=2}^{k-1}(n_{\tau}-n_{\tau-1})\alpha_{n_1}^*}{n_{1}+n_{k}-n_{k-1}}}}\bigg)+h.o.t.\\
			&=:&\mathcal{J}(\rho_*,\overrightarrow{\alpha}_*)+h.o.t.
		\end{eqnarray*}
		for $N=2$ and
		\begin{eqnarray*}
			\mathcal{J}(\overrightarrow{\rho},\overrightarrow{\alpha})
			&=&\vartheta\sum_{j=1}^{d}\bigg( m_j+\frac{B_j\delta_j+o(1)}{(\rho^*\vartheta\log\vartheta)^{\nu_j}}-\frac{C_j+o(1)}{(\rho^*\log\vartheta)\vartheta^{2\pi\sqrt{\lambda_{j}}\rho^*}}\bigg)\\
			&&-\sum_{\tau=2}^{k-1}\bigg(\frac{D_{\tau}\sum_{j=n_{\tau-1}+1}^{n_{\tau}-1}\beta_{j,j+1}\vartheta\log(\rho^*\alpha_{n_\tau}^*\log\vartheta)}
			{(\rho^*\alpha_{n_{\tau}}^*\log\vartheta)^{2}\vartheta^{2\sqrt{\lambda_{n_{\tau}}}\rho^*\alpha_{n_{\tau}}^*}}
			+\frac{D_{\tau}''\beta_{n_{\tau},n_{\tau}+1}\vartheta}
			{(\rho^*\alpha_{n_{\tau}}^*\log\vartheta)^2\vartheta^{2\sqrt{\lambda_{n_{\tau}}}\rho^*\alpha_{n_{\tau}}^*}}\bigg)\\
			&&-\bigg(\frac{D_{1}\sum_{j=1}^{n_{1}-1}\beta_{j,j+1}\vartheta(n_{1}+n_{k}-n_{k-1})^{2}\log(\rho^*\alpha_{n_1}^*\log\vartheta)}
			{(\rho^*(2\pi-\sum_{\tau=2}^{k-1}\alpha_{n_{\tau}}^*)\log\vartheta)^{2}\vartheta^{2\sqrt{\lambda_{n_{1}}}\rho^*\frac{
						2\pi-\sum_{\tau=2}^{k-1}(n_{\tau}-n_{\tau-1})\alpha_{n_\tau}^*}{n_{1}+n_{k}-n_{k-1}}}}\\
			&&+\frac{D_{1}''\beta_{n_{1},n_{1}+1}\vartheta(n_{1}+n_{k}-n_{k-1})^2}
			{(\rho^*(2\pi-\sum_{\tau=2}^{k-1}\alpha_{n_{\tau}}^*)\log\vartheta)^2\vartheta^{2\sqrt{\lambda_{n_{1}}}\rho^*\frac{
						2\pi-\sum_{\tau=2}^{k-1}(n_{\tau}-n_{\tau-1})\alpha_{n_1}^*}{n_{1}+n_{k}-n_{k-1}}}}\bigg)+h.o.t.\\
			&=:&\mathcal{J}(\rho_*,\overrightarrow{\alpha}_*)+h.o.t.
		\end{eqnarray*}
		for $N=3$, where $\overrightarrow{\alpha}_*=(\alpha_{n_2}^*,\cdots,\alpha_{n_{k-1}}^*)$.
		Let
		\begin{eqnarray*}
			\mathcal{M}=\{\alpha_{n_\tau}^*\in(0, 2\pi)\mid \alpha_{n_{\tau}+1}\leq\alpha_{n_{\tau}}\text{ for all }\tau=2,\cdots,k-1\}.
		\end{eqnarray*}
		Since $\overline{\mathcal{M}}$ is compact, the restriction of $\mathcal{J}(\rho_*,\overrightarrow{\alpha}_*)$ on $\overline{\mathcal{M}}$ attains its maximum and minimum at some points, say $\overrightarrow{\alpha}_{*,\flat}$ and $\overrightarrow{\alpha}_{*,\ddag}$, respectively.
		We claim that $\overrightarrow{\alpha}_{*,\flat}\in\text{int}(\mathcal{M})$ in the case of $(a)$ and $\overrightarrow{\alpha}_{*,\ddag}\in\text{int}(\mathcal{M})$ in the cases of $(b)$, $(c)$ and $(d)$.  Indeed, assume that $\overrightarrow{\alpha}_{*,\flat}\in\partial\mathcal{M}$, then one of the three cases must happen:
		\begin{enumerate}
			\item[$(i)$]\quad  $\alpha_{n_\tau}^{*,\flat}=0$ for some $\tau$,
			\item[$(ii)$]\quad  $\alpha_{n_{\tau'}}^{*,\flat}=2\pi$ for some $\tau'$,
			\item[$(iii)$]\quad  $\alpha_{n_\tau}^{*,\flat}=\alpha_{n_\tau'}^{*,\flat}$ for some $\tau'>\tau$.
		\end{enumerate}
		Clearly, $\mathcal{J}(\rho,\overrightarrow{\alpha}_*)\to-\infty$ as $\alpha_{\tau}^*\to0$ or $\alpha_{\tau}^*\to2\pi$ for some $\tau$ in the case of $(a)$.  Thus, we must have the case~$(iii)$.
		Note that
		\begin{eqnarray*}
			\partial_{\alpha_{n_{\tau}}^*}\bigg(\frac{D_\tau^*}
			{(\rho^*\alpha_{n_{\tau}}^*\log\vartheta)^{a}\vartheta^{2\sqrt{\lambda_{n_{\tau}}}\rho^*\alpha_{n_{\tau}}^*}}+\frac{D_{1}^*}
			{(\rho^*(2\pi-\sum_{\tau=2}^{k-1}\alpha_{n_{\tau}}^*)\log\vartheta)^{b}\vartheta^{2\sqrt{\lambda_{1}}\rho^*\frac{
						2\pi-\sum_{\tau=2}^{k-1}(n_{\tau}-n_{\tau-1})\alpha_{n_\tau}^*}{n_{1}+n_{k}-n_{k-1}}}}\bigg)\leq0
		\end{eqnarray*}
		is equivalent to
		\begin{eqnarray}\label{eqnn9986}
			2\sqrt{\lambda_{n_{\tau}}}\rho^*\alpha_{n_{\tau}}^*- 2\sqrt{\lambda_{1}}\rho^*\frac{2\pi-\sum_{\tau=2}^{k-1}(n_{\tau}-n_{\tau-1})\alpha_{n_\tau}^*}{n_{1}+n_{k}-n_{k-1}}+o(1)\leq0
		\end{eqnarray}
		for all $a,b>0$ and $D_\tau^*,D_1^*>0$.
		Thus, $(\overrightarrow{\rho}_{*,\flat},\overrightarrow{\alpha}_{*,\flat})$ must satisfies \eqref{eqnn9986} for $\tau'$ and satisfies
		\begin{eqnarray}\label{eqnn9985}
			2\sqrt{\lambda_{n_{\tau}}}\rho^*\alpha_{n_{\tau}}^*-2\sqrt{\lambda_{1}}\rho^*\frac{
				2\pi-\sum_{\tau=2}^{k-1}(n_{\tau}-n_{\tau-1})\alpha_{n_\tau}^*}{n_{1}+n_{k}-n_{k-1}}+o(1)\geq0
		\end{eqnarray}
		for $\tau$ since $\tau'>\tau$.  It follows from \eqref{eqnn9986} and \eqref{eqnn9985} that $\sqrt{\lambda_{\tau'}}\rho_{\tau'}^*\alpha_{\tau'}^{*,\flat}\leq \sqrt{\lambda_{\tau}}\rho_{\tau}^*\alpha_{\tau}^{*,\flat}$,
		which is impossible since $\tau'>\tau$ and $\alpha_{\tau}^{*,\flat}=\alpha_{\tau'}^{*,\flat}$.  Thus, $\max_{\overrightarrow{\alpha}_*\in\mathcal{M}}\mathcal{J}(\rho_*,\overrightarrow{\alpha}_*)$
		is attained at some point of int$(\mathcal{M})$ for $N=2$.  For $N=3$, the argument is similar so we omit it.  In the cases~$(b)$, $(c)$ and $(d)$, we know that $\mathcal{J}(\rho,\overrightarrow{\alpha}_*)\to+\infty$ as $\alpha_{\tau}^*\to0$ or $\alpha_{\tau}^*\to2\pi$ for some $\tau$.  Thus, by similar arguments, we have $\overrightarrow{\alpha}_{*,\ddag}\in\text{int}(\mathcal{M})$.  We denote this point (the maximum point in the case~$(a)$ and the minimum point in the cases~$(b)$, $(c)$ and $(d)$) by $\overrightarrow{\alpha}_{*,0}(\rho_*)$ for every $\rho_*>0$.  By \eqref{eqnn9986}, we know that $\overrightarrow{\alpha}_{*,0}(\rho_*)$ satisfies
		\begin{eqnarray}\label{eqnnnew1113}
			\sqrt{\lambda_{n_\tau}}\alpha_{n_\tau}^{*,0}(\rho_*)&=&\sqrt{\lambda_{n_\tau'}}\alpha_{n_\tau'}^{*,0}(\rho_*)+o(1)\notag\\
			&=&\sqrt{\lambda_{1}}\frac{
				2\pi-\sum_{\tau=2}^{k-1}(n_{\tau}-n_{\tau-1})\alpha_{n_\tau}^*}{n_{1}+n_{k}-n_{k-1}}+o(1)
		\end{eqnarray}
		for all $\tau=2,\cdots,k-1$, which, together with $\alpha_{n_k}=(1+o(1))\alpha_{n_1}$ and \eqref{eqnnnew1116}, implies that
		\begin{eqnarray*}
			\alpha_{n_\tau}^{*,0}(\rho_*)=\alpha_{n_\tau}^{**,0}+o(1)\quad\text{for all }\tau=1,2,\cdots,k,
		\end{eqnarray*}
		where
		\begin{eqnarray*}
			\alpha_{n_\tau}^{**,0}=\frac{2\pi}{\sqrt{\lambda_{n_\tau}}(\frac{n_{k}-n_{k-1}}{\sqrt{\lambda_{n_1}}}
				+\sum_{s=1}^{k-1}\frac{n_{s}-n_{s-1}}{\sqrt{\lambda_{n_s}}})}\quad\text{for all }\tau=1,2,\cdots,k-1\text{ and }\alpha_{n_k}^{**,0}=\alpha_{n_1}^{**,0}.
		\end{eqnarray*}
		Now, suppose that $\rho_*\geq\ve$ for a sufficiently small $\ve>0$ which is independent of $\vartheta$. Recall that $\nu_*=\min\{\nu_j\}$ and $\mathfrak{m}_{*}=\{j=1,2,\cdots,d\mid \nu_j=\nu_*\}$, then by \eqref{eqnnnew1113},
		\begin{eqnarray*}
			\mathcal{J}(\rho_*,\overrightarrow{\alpha}_{**,0})=\vartheta(\sum_{j=1}^dm_j+\frac{\sum_{j\in\mathfrak{m}_{*}}B_j\delta_j}{(\rho_*\vartheta\log\vartheta)^{\nu_*}}
			-\frac{2C_1+\pi^{-\frac12}D_1\beta_{1,2}}{(\rho^*\log\vartheta)^{\frac12}\vartheta^{\lambda_*\rho^*}})+h.o.t.
		\end{eqnarray*}
		in the case $d=2$ and $N=2$ with $\lambda_1=\lambda_2$,
		\begin{eqnarray*}
			\mathcal{J}(\rho_*,\overrightarrow{\alpha}_{**,0})=\vartheta(\sum_{j=1}^dm_j+\frac{\sum_{j\in\mathfrak{m}_{*}}B_j\delta_j}{(\rho_*\vartheta\log\vartheta)^{\nu_*}}
			-\frac{C_1}{(\rho^*\log\vartheta)^{\frac12}\vartheta^{\lambda_*\rho^*}})+h.o.t.
		\end{eqnarray*}
		in the case $d=2$ and $N=2$ with $\lambda_1\not=\lambda_2$,
		\begin{eqnarray*}
			\mathcal{J}(\rho_*,\overrightarrow{\alpha}_{**,0})&=&\vartheta(\sum_{j=1}^dm_j+\frac{\sum_{j\in\mathfrak{m}_{*}}B_j\delta_j}{(\rho_*\vartheta\log\vartheta)^{\nu_*}}
			-\frac{\sum_{j=n_{k-1}+1}^{n_k-1}D_\tau\beta_{j,j+1}}{(\rho^*\alpha_{n_{k-1}}^{**,0}\log\vartheta)^{\frac12}\vartheta^{\lambda_*\rho^*}}\\
			&&-\frac{D_{k-1}''\beta_{n_{k-1},n_{k-1}+1}}
			{(\rho^*\alpha_{n_{k-1}}^*\log\vartheta)^{\frac12}\vartheta^{2\lambda_{**}\rho^*}})+h.o.t.
		\end{eqnarray*}
		in the case $d\geq3$ and $N=2$,
		\begin{eqnarray*}
			\mathcal{J}(\rho_*,\overrightarrow{\alpha}_{**,0})=\vartheta(\sum_{j=1}^dm_j+\frac{\sum_{j\in\mathfrak{m}_{*}}B_j\delta_j}{(\rho_*\vartheta\log\vartheta)^{\nu_*}}
			-\frac{\sum_{j=1}^{n_1}C_j}{(\rho^*\log\vartheta)\vartheta^{\lambda_*\rho^*}})+h.o.t.
		\end{eqnarray*}
		in the case of $N=3$ and $d=2$ and
		\begin{eqnarray*}
			\mathcal{J}(\rho_*,\overrightarrow{\alpha}_{**,0})&=&\vartheta(\sum_{j=1}^dm_j+\frac{\sum_{j\in\mathfrak{m}_{*}}B_j\delta_j}{(\rho_*\vartheta\log\vartheta)^{\nu_*}}
			-\frac{\sum_{j=n_{k-1}+1}^{n_k-1}D_\tau\beta_{j,j+1}\log(\rho^*\alpha_{n_{k-1}}^*\log\vartheta)}{(\rho^*\alpha_{n_{k-1}}^*\log\vartheta)^{2}\vartheta^{\lambda_*\rho^*}}\\
			&&-\frac{D_{k-1}''\beta_{n_{k-1},n_{k-1}+1}\log(\rho^*\alpha_{n_{k-1}}^*\log\vartheta)}
			{(\rho^*\alpha_{n_{k-1}}^*\log\vartheta)\vartheta^{2\lambda_{**}\rho^*}})+h.o.t.
		\end{eqnarray*}
		in the case $d\geq3$ and $N=3$ for $\vartheta>0$ sufficiently large, where $\lambda_*,\lambda_{**}>0$ are constants which are independent of $\ve$.  Thus, by direct calculations and taking $\ve>0$ sufficiently small if necessary, we can see that
		\begin{eqnarray*}
			\max_{\rho_*\geq\ve}\mathcal{J}(\rho_*,\overrightarrow{\alpha}_{**,0})
		\end{eqnarray*}
		is attained at $\rho_*=\frac{\nu_*}{\lambda_*}+o(1)$ in the case $(a)$ and $(c)$, and
		\begin{eqnarray*}
			\min_{\rho_*\geq\ve}\mathcal{J}(\rho_*,\overrightarrow{\alpha}_{**,0})
		\end{eqnarray*}
		attained at $\rho_*=\frac{\nu_*}{\lambda_*}+o(1)$ in the case $(b)$ and $(d)$.  It follows from the fact that $\overrightarrow{\alpha}_{*,0}(\rho_*)$ is the maximum point in the case $(a)$ that $\max_{\rho_*\geq\ve,\overrightarrow{\alpha}_*\in\mathcal{M}}\mathcal{J}(\rho_*,\overrightarrow{\alpha}_*)$
		is attained by some points of $(\rho_{***},\overrightarrow{\alpha}_{***})$ in the case $(a)$, which implies that $\mathcal{J}(\overrightarrow{\rho},\overrightarrow{\alpha})$ has a critical point with
		$\rho_{0}^*>0$ and $\alpha_{j,0}^*\in(0, 2\pi)$ for all $j$ if $\vartheta>0$ sufficiently large in the case $(a)$.  In the case $(b)$ and $(d)$, by the fact that $\overrightarrow{\alpha}_{*,0}(\rho_*)$ is the minimum point, we know that $\min_{\rho_*\geq\ve,\overrightarrow{\alpha}_*\in\mathcal{M}}\mathcal{J}(\rho_*,\overrightarrow{\alpha}_*)$
		is attained by some points $(\rho_{***},\overrightarrow{\alpha}_{***})$, which implies that $\mathcal{J}(\overrightarrow{\rho},\overrightarrow{\alpha})$ has a critical point with
		$\rho_{0}^*>0$ and $\alpha_{j,0}^*\in(0, 2\pi)$ for all $j$ if $\vartheta>0$ sufficiently large in the case $(b)$ and $(d)$.  In the case $(c)$, we denote $\widetilde{\mathcal{J}}(\rho_*,\overrightarrow{\alpha}_*)=-\mathcal{J}(\rho_*,\overrightarrow{\alpha}_*)$.
		Then, we know that $\min_{\rho_*\in\mathcal{N}}\max_{\overrightarrow{\alpha}_*\in\mathcal{M}}\widetilde{\mathcal{J}}(\rho_*,\overrightarrow{\alpha}_*)$
		is attained by some points of int$(\mathcal{N}\times\mathcal{M})$, where $\mathcal{N}=\{\rho_*\geq\ve\}$.
		Again, we denote this point by $(\rho_{***},\overrightarrow{\alpha}_{***})$.  Now, let
		\begin{eqnarray*}
			\mathcal{M}_{\delta,*}=\{\overrightarrow{\alpha}_*\in\mathcal{M}\mid |\overrightarrow{\alpha}_*-\overrightarrow{\alpha}_{***}|\leq\delta, \rho_*=\rho_{***}\}
		\end{eqnarray*}
		and
		\begin{eqnarray*}
			\mathcal{N}_*=\{\rho_*\geq\ve\mid \overrightarrow{\alpha}_*=\overrightarrow{\alpha}_{**,0}\}.
		\end{eqnarray*}
		where we take $\delta>0$ small such that
		\begin{eqnarray*}
			\max_{\partial\mathcal{M}_{\delta,*}}\widetilde{\mathcal{J}}(\rho_*,\overrightarrow{\alpha}_*)
			<\min_{\mathcal{N}_*}\widetilde{\mathcal{J}}(\rho_*,\overrightarrow{\alpha}_*).
		\end{eqnarray*}
		Since $\partial\overline{\mathcal{M}}_{\delta,*}$ is homeomorphic to the ball of $\bbr^{k-2}$, it is standard to show that $\partial\mathcal{M}_{\delta,*}$ links to $\mathcal{N}_*$ (cf. \cite[Chapter~2]{W96}).  Thus, by the linking theorem (cf. \cite[Theorem~2.9]{W96}),
		\begin{eqnarray*}
			\mathcal{C}=\min_{\varphi\in\Gamma}\max_{(\rho_*,\overrightarrow{\alpha}_*)\in\mathcal{M}_{\delta,*}}\widetilde{\mathcal{J}}(\varphi(\rho_*,\overrightarrow{\alpha}_*))
		\end{eqnarray*}
		is a critical value, where $\Gamma=\{\varphi\in C(\mathcal{M}_{\delta,*}, \bbr^{k-1})\mid \varphi(\partial\mathcal{M}_{\delta,*})=id\}$.  It follows that $\mathcal{J}(\overrightarrow{\rho},\overrightarrow{\alpha})$ has a critical point with
		$\rho_{j,0}^*>0$ and $\alpha_{j,0}^*\in(0, 2\pi)$ for all $j$ if $\vartheta>0$ sufficiently large in the case $(c)$.
	\end{proof}
	
	We close this section by the proof of Theorem~\ref{thm0001}.
	
	\vskip0.12in
	
	\noindent\textbf{Proof of Theorem~\ref{thm0001}:} This proof follows immediately from Lemma~\ref{lem0005} and Proposition~\ref{prop0002}.
	\hfill$\Box$

	\section{Acknowledgements}
	The research of J. Wei is partially supported by NSERC of Canada and the research of Y. Wu is supported by NSFC (No. 11971339, 12171470).

\end{document}